\documentclass[12pt]{amsart}

\usepackage[dvips]{graphics}
\usepackage{epsfig}
\usepackage{amssymb}
\usepackage{amsthm}
\usepackage{amscd}
\usepackage[mathscr]{eucal}
\usepackage{enumitem}

\theoremstyle{plain}
\newtheorem{theorem}{Theorem}[section]
\newtheorem{corollary}[theorem]{Corollary}
\newtheorem{lemma}[theorem]{Lemma}
\newtheorem{subclaim}[theorem]{Subclaim}
\newtheorem{proposition}[theorem]{Proposition}

\newtheorem{claim}[theorem]{Claim}
\newtheorem{question}[theorem]{Question}
\theoremstyle{definition}
\newtheorem{definition}[theorem]{Definition}
\newtheorem{remark}[theorem]{Remark}

\newtheorem{notation}[theorem]{Notation}
\newtheorem{example}[theorem]{Example}

\theoremstyle{remark}

\newcommand{\ob}{\textup{Ob}}
\newcommand{\mor}{\textup{Mor}}

\newcommand{\dom}{\operatorname{dom}}

\newcommand{\tp}{\operatorname{tp}}
\newcommand{\stp}{\operatorname{stp}}

\newcommand{\bdd}{\operatorname{bdd}}
\newcommand{\acl}{\operatorname{acl}}
\newcommand{\dcl}{\operatorname{dcl}}
\newcommand{\cl}{\operatorname{cl}}
\newcommand{\id}{\operatorname{id}}

\newcommand{\ov}{\overline}
\newcommand{\be}{\begin{enumerate}}
\newcommand{\ee}{\end{enumerate}}

\renewcommand{\phi}{\varphi}

\newcommand{\la}{\langle}
\newcommand{\ra}{\rangle}

\def\Ind{\setbox0=\hbox{$x$}\kern\wd0\hbox to 0pt{\hss$\mid$\hss}
\lower.9\ht0\hbox to 0pt{\hss$\smile$\hss}\kern\wd0}

\def\Notind{\setbox0=\hbox{$x$}\kern\wd0\hbox to 0pt{\mathchardef
\nn=12854\hss$\nn$\kern1.4\wd0\hss}\hbox to
0pt{\hss$\mid$\hss}\lower.9\ht0 \hbox to
0pt{\hss$\smile$\hss}\kern\wd0}

\def\ind{\mathop{\mathpalette\Ind{}}}

\newbox\noforkbox \newdimen\forklinewidth
\forklinewidth=0.3pt \setbox0\hbox{$\textstyle\smile$}
\setbox1\hbox to \wd0{\hfil\vrule width \forklinewidth depth-2pt
 height 10pt \hfil}
\wd1=0 cm \setbox\noforkbox\hbox{\lower 2pt\box1\lower
2pt\box0\relax}
\def\unionstick{\mathop{\copy\noforkbox}\limits}

\def\nonfork_#1{\unionstick_{\textstyle #1}}

\setbox0\hbox{$\textstyle\smile$} \setbox1\hbox to \wd0{\hfil{\sl
/\/}\hfil} \setbox2\hbox to \wd0{\hfil\vrule height 10pt depth
-2pt width
               \forklinewidth\hfil}
\newbox\doesforkbox
\setbox\doesforkbox\hbox{\lower 2pt\box1 \lower
2pt\box2\lower2pt\box0\relax}
\def\nunionstick{\mathop{\copy\doesforkbox}\limits}

\def\fork_#1{\nunionstick_{\textstyle #1}}

\newcommand{\C}{\operatorname{\mathfrak{C}}}
\renewcommand{\P}{{\mathcal P}}

\newcommand{\Aut}{\operatorname{Aut}}

\newcommand{\G}{\mathcal{G}}

\newcommand{\bd}{\partial}

\def\cA{\mathcal{A}}
\def\cC{\mathcal{C}}
\def\CF{\mathcal{F}}
\def\cT{\mathcal{T}}
\def\cV{\mathcal{V}}
\def\cD{\mathcal{D}}
\def\cG{\mathcal{G}}

\def\Z{\operatorname{Z}}
\def\CF{\mathcal{F}}

\def\bd{\partial}

\def\supp{\textup{supp}}

\def\cd{\cdot}
\def\ul{\underline}

\def\bsigma{\mbox{\boldmath $\sigma$}}

\def\bmu{\mbox{\boldmath $\mu$}}

\title[Homology groups in model theory]{Amalgamation functors and homology groups in model theory}
\author{John Goodrick}
\author{Byunghan Kim}
\author{Alexei Kolesnikov}

\address{Department of Mathematics\\ Universidad de los Andes \\
Bogot\'{a}, Colombia}
\address{Department of Mathematics\\ Yonsei University\\
134 Shinchon-dong, Seodaemun-gu\\
Seoul 120-749, South Korea}
\address{Department of Mathematics\\ Towson University, MD\\
USA}

\email{jr.goodrick427@uniandes.edu.co}
\email{bkim@yonsei.ac.kr}
\email{akolesnikov@towson.edu}

\thanks{The  2nd author was supported by NRF grant 2010-0016044.
The third author was partially supported by NSF grant DMS-0901315.}


\begin{document}

\begin{abstract}
We present definitions of homology groups $H_n$, $n\ge 0$, associated to a family of ``amalgamation functors''. We show that if the generalized amalgamation properties hold, then the homology groups are trivial. We compute the group $H_2$ for strong types in stable theories and show that any profinite abelian group can occur as the group $H_2$ in the model-theoretic context.
\end{abstract}

\maketitle

The work described in this paper was originally inspired by Hrushovski's discovery \cite{Hr} of striking connections between amalgamation properties and definable groupoids in models of a stable first-order theory. Hrushovski showed that if the theory fails 3-uniqueness (the model-theoretic definition of this property is in Section~3), then there must exist a groupoid whose sets of objects and morphisms, as well as the composition of morphisms, are definable in models of a first-order theory. In~\cite{GK}, an explicit construction of such a groupoid was given and it was shown in~\cite{gkk} that the group of automorphisms of each object of such a groupoid must be abelian profinite. The morphisms in the groupoid construction in~\cite{GK} arise as equivalence classes of ``paths'', defined in a model-theoretic way. In some sense, the groupoid construction paralleled that of the construction of a fundamental groupoid in a topological space. Thus it seems natural to ask whether it is possible to define the notion of a homology group in model-theoretic context and, if yes, would the homology group be linked to the group described in~\cite{GK,gkk}. We find that the answer is ``yes'' to both questions.

In this paper, we describe a way to define homology groups for a wide class of first-order theories. We start by describing the construction using category-theoretic language; the few model-theoretic references in Section~1 are provided only as motivation for the notions introduced there. Our goal is to separate, as much as possible, general arguments that do not require heavy use of model-theoretic context from the arguments that do. It turns out that, even at this level of generality, it is possible to show that the homology groups have a fairly simple structure (see, for example, Theorem~\ref{Hn_shells} or Corollary~\ref{Hn_pockets}).

We then define, in Section 2, several natural homology groups in the model-theoretic context (e.g., the type-homology groups and set-homology groups) and show that such groups must be isomorphic. The latter is, of course, what one would expect to see; but the argument turned out to be fairly involved. We show that the homology groups of a complete type of the theory are related to its amalgamation properties: if a type $p$ has $k$-amalgamation for every $k \leq n$, then $H_{n-2}(p) = 0$ (Corollary~\ref{trivial_homology}), and when $4$-amalgamation fails, $H_2(p)$ can be nontrivial, even in a stable theory (see the examples in Section~5). In particular, we show that \emph{any} profinite abelian group can be the group $H_2(p)$ for a suitable $p$ in a stable theory.

Section 4 is devoted to the proof of an analogue of Hurewicz's theorem: in a stable theory, the group $H_2(p)$ is isomorphic to a certain automorphism group $\Gamma_2(p)$ which is analogous to a fundamental group.  It turns out that $\Gamma_2(p)$ is always abelian, so there is no need to take the abelianization as in the usual Hurewicz theorem. But in Section~7, we construct a different canonical ``fundamental" group for the type $p$ which seems to give more information: this new group need not be abelian, and the group $\Gamma_2(p)$ is in the center of the new group.


Section 6 gives examples of homology groups in unstable theories, where it is still unclear what model-theoretic properties are measured by these groups.

Amalgamation properties have already been much studied by researchers in simple theories (for instance, in \cite{KKT}), and recently the first and third authors of this note investigated analogies with homotopy theory rather than homology theory in \cite{GK}. In some sense, this  paper is a companion to \cite{GK}.
For general background on simple theories, the reader is referred to the book \cite{wagner}, which explains nonforking, hyperimaginaries, and much more.

\section{Simplicial homology in a category}

In this section, we define simplicial homology groups in a more general category-theoretic setting than our intended applications to model theory.  We aim to provide a general framework for our homological computations and separate some of the category-theoretic arguments from the model-theoretic ones.  This section uses model theory only as a source of examples.

The homology groups defined in this section are not, in general, homology groups associated with the entire category $\cC$, but rather are homology groups of a particular class of ``amalgamation functors'' $\cA$ from a class of certain partially ordered sets viewed as categories into the category $\cC$. The class $\cA$ is assumed to satisfy certain basic closure properties (see Definition~\ref{amenable} below).

\subsection{Basic definitions and facts.}

Throughout this section, let $\cC$ be a category.  If $s$ is a set, then we consider the power set $\P(s)$ of $s$ to be a category with a single inclusion map $\iota_{u,v} : u \rightarrow v$ between any pair of subsets $u$ and $v$ with $u \subseteq v$.  A subset $X \subseteq \P(s)$ is called \emph{downward-closed} if whenever $u \subseteq v \in X$, then $u \in X$. In this case we consider $X$ to be a full subcategory of $\P(s)$.  An example of a downward-closed collection that we will use often below is $\P^-(s) := \P(s) \setminus \{s\}$.

We are interested in subfamilies  of functors $f:X \rightarrow \cC$ for downward-closed subsets $X \subseteq \P(s)$ for various finite subset sets $s$ of the set of natural numbers. Before specifying the desirable closure properties of a collection $\cA$ of such functors, we need some auxiliary definitions.

\begin{definition}
(1) Let $X$ be a downward closed subset of $\P(s)$ and let $t\in X$. The symbol $X|_t$ denotes the set $\{u \in \P(s\setminus t)\mid t\cup u\in X\}\subseteq X$.

(2) For $s$, $t$, and $X$ as above, let $f:X\to \cC$ be a functor. Then \emph{the localization of $f$ at $t$} is the functor $f|_t:X|_t\to \cC$ such that
$$f |_t (u) = f(t \cup u)$$ and whenever $u \subseteq v \in X|_t$, $$f |_t(\iota_{u,v}) = f(\iota_{u \cup t, v \cup t}).$$

(3) Let $X\subset \P(s)$ and $Y\subset \P(t)$ be downward closed subsets, where $s$ and $t$ are finite sets of natural numbers. Let $f:X\to \cC$ and $g:Y\to \cC$ be functors. We say that $f$ and $g$ are \emph{isomorphic} if there is an order-preserving bijection $\sigma: s \rightarrow t$ such that $Y = \{\sigma(u) : u \in X\}$ and a family of isomorphisms $\langle h_u : f(u) \rightarrow g(\sigma(u)) : u \in X \rangle$ in $\cC$ such that for any $u \subseteq v \in X$, the following diagram commutes:
$$
\begin{CD}
f(u) @>h_u>> g(\sigma(u))\\
@VV{f(\iota_{u,v})}V @VV{g(\iota_{\sigma(u),\sigma(v)})}V\\
f(v) @>h_v>> g(\sigma(v))
\end{CD}
$$
\end{definition}

\begin{remark}
If $X$ is a downward closed subset of $\P(s)$ and $t\in X$, then $X|_t$ is a downward closed subset of $\P(s\setminus t)$. Moreover $X|_t$ does not depend on the choice of $s$.
\end{remark}

\begin{definition}
\label{amenable}
Let $\cA$ be a non-empty collection of functors $f: X \rightarrow \cC$ for various non-empty downward-closed
subsets $X \subseteq \P(s)$ for all finite sets $s$ of natural numbers. We say that $\cA$ is \emph{amenable} if it
satisfies all of the following properties:
\begin{enumerate}
\item (Invariance under isomorphisms) Suppose that $f: X \rightarrow \cC$ is in $\cA$ and $g:Y\to \cC$ is isomorphic to $f$. Then $g\in \cA$.

\item (Closure under restrictions and unions) If $X \subseteq \P(s)$ is downward-closed and $f: X \rightarrow \cC$ is a functor, then $f \in \cA$ if and only if for every $u \in X$, we have that $f \upharpoonright \P(u) \in \cA$.

\item (Closure under localizations) Suppose that $f : X \rightarrow \cC$ is in $\cA$ for some $X \subseteq \P(s)$ and $t \in X$. Then $f|_t:X|_t\to \cC$ is also in $\cA$.

\item (De-localization) Suppose that $f: X \rightarrow \cC$ is in $\cA$ and $t \in X \subseteq \P(s)$ is such that $X|_t=X \cap \P(s \setminus t)$. Suppose that the localization $f |_t : X \cap \P(s \setminus t) \rightarrow \cC$ has an extension $g : Z \rightarrow \cC$ in $\cA$ for some $Z \subseteq \P(s \setminus t)$.  Then there is a map $g_0 : Z_0 \rightarrow \cC$ in $\cA$ such that $Z_0 = \{u \cup v : u \in Z, v \subseteq t\}$, $f \subseteq g_0$, and $g_0 |_t = g$.
\end{enumerate}
\end{definition}

\begin{remark}
For example, we could take $\cC$ to be all boundedly (or algebraically) closed subsets of the monster model of a first-order theory, and let $\cA$ be all functors which are ``independence-preserving'' (in Hrushovski's terminology
\cite{Hr}) and such that every face $f(u)$ is the bounded (or algebraic) closure of its vertices; then $\cA$ is amenable.
Other examples of amenable collections come from further restricting $\cA$ by requiring, for instance, that all the ``vertices'' $f(\{i\})$ of functors $f \in \cA$ be of a certain type, or by restricting the possible types of edges, faces, et cetera. These examples will  be explained more precisely in Section 2.
\end{remark}

\begin{remark}
Note that any functor $f: X \rightarrow \cC$ has a ``base'' $f(\emptyset)$ which is embedded into each $f(u)$ for $u \in X$. This base does not play an important role in computing the homology groups, but it does have model-theoretic significance. In particular, it is often convenient to fix this base.
\end{remark}

\begin{definition}
Let $B \in \ob(\cC)$ and suppose $f(\emptyset)=B$.  We say that $f$ \emph{is over $B$} and we let $\cA_B$ denote the set of all functors $f \in \cA$ that are over $B$.
\end{definition}

In model-theoretic applications, there will always be an initial object of $\cA$ which will be the natural choice for $B$ (either the ``empty type'' for type homology, or the ``empty tuple'' for set homology).

\begin{remark}
\label{amenable2}
It is easy to see that condition~(2) in Definition~\ref{amenable} is equivalent to the conjunction of the following two conditions:

(Closure under restrictions) If $f : X \rightarrow \cC$ is in $\cA$ and $Y \subseteq X$ with $Y$ downward-closed, then $f \upharpoonright Y$ is also in $\cA$.

(Closure under unions) Suppose that $f: X \rightarrow \cC$ and $g: Y \rightarrow \cC$ are both in $\cA$ and that
$f\upharpoonright X\cap Y=g\upharpoonright X\cap Y$.
Then the union $f \cup g: X \cup Y \rightarrow \cC$ is also in $\cA$.

For instance, if these two conditions are true and $f: X \rightarrow \cC$ is a functor from a downward-closed set $X$ such that $f \upharpoonright \P(u) \in \cA$ for every $u \in X$, then if $u_1, \ldots, u_n$ are maximal sets in $X$, we can use closure under unions $(n-1)$ times to see that $f \in \cA$ (since it is the union of the functors $f \upharpoonright \P(u_i)$).
\end{remark}

From now on, we assume that $\cA$ is a nonempty amenable collection of functors mapping into the category $\cC$. As we mentioned in the above remark, every functor in $\cA$ can be described as the union of functors whose domain is $\P(s)$ for for finite set $s$. Such functors will play a central role in this paper.

\begin{definition}
Let $n\ge 0$ be a natural number. A \emph{(regular) $n$-simplex in $\cC$} is a functor $f : \P(s) \rightarrow \cC$ for some set $s \subseteq \omega$ with $|s| = n+1$.  The set $s$ is called the \emph{support of $f$}, or $\supp(f)$.

Let $S_n(\cA; B)$ denote the collection of all regular $n$-simplices in $\cA_B$.
Then put $S(\cA;B):=\bigcup_{n} S_n(\cA;B)$
and $S(\cA):=\bigcup_{B \in \ob(\cC)} S(\cA;B)$.

Let $C_n (\cA; B)$ denote the free abelian group generated by $S_n(\cA; B)$; its elements are called \emph{$n$-chains in $\cA_B$}, or \emph{$n$-chains over $B$}.  Similarly, we define $C (\cA; B):=\bigcup_{n} C_n(\cA;B)$ and  $C (\cA):=\bigcup_{B \in \ob(\cC)} C(\cA;B)$. The \emph{support of a chain $c$} is the union of the supports of all the simplices that appear in $c$ with a nonzero coefficient.
\end{definition}

\begin{remark}
This is more or less a special case of what are known as ``simplicial objects in the category $\cC$,'' except that we do not equip our simplices with degeneracy maps.
\end{remark}

The adjective ``regular'' in the definition above is to emphasize that none of our simplices are ``degenerate:'' their domains must be \emph{strictly} linearly ordered.  It is more usual to allow for degenerate simplices, but for our purposes, this extra generality does not seem to be useful.  Since all of our simplices will be regular, we will omit the word ``regular'' in what follows.

\begin{definition}
If $n \geq 1$ and $0 \leq i \leq n$, then the \emph{ith boundary operator} $\bd^i_n : C_n (\cA;B) \rightarrow C_{n-1} (\cA; B)$ is defined so that if $f$ is a regular $n$-simplex with domain $\P(s)$, where $s = \{s_0, \ldots, s_n\}$ with $s_0 < \ldots < s_n$, then $$\bd^i_n(f) = f \upharpoonright \P(s \setminus \{s_i\})$$ and extended linearly to a group map on all of $C_n (\cA; B)$.

If $n \geq 1$ and $0 \leq i \leq n$, then the \emph{boundary map} $\bd_n : C_n(\cA; B) \rightarrow C_{n-1}(\cA; B)$ is defined by the rule
$$\bd_n(c) = \Sigma_{0 \leq i \leq n} (-1)^i \bd^i_n (c).$$

We write $\bd^i$ and $\bd$ for $\bd^i_n$ and $\bd_n$, respectively, if $n$ is clear from context.
\end{definition}

\begin{definition}
The kernel of $\bd_n$ is denoted $Z_n(\cA; B)$, and its elements are called \emph{($n$-)cycles}.  The image of $\bd_{n+1}$ in $C_n(\cA; B)$ is denoted $B_n(\cA; B)$. The elements of $B_n(\cA; B)$ are called \emph{($n$-)boundaries}.
\end{definition}

It can be shown (by the usual combinatorial argument) that $B_n(\cA; B) \subseteq Z_n
(\cA; B)$, or more briefly, ``$\bd_n\circ \bd _{n+1} = 0$.''  Therefore we can define simplicial homology groups relative to $\cA$:

\begin{definition}
The \emph{$n$th (simplicial) homology group of $\cA$ over $B$} is $$H_n(\cA; B) = Z_n(\cA; B) / {B_n(\cA; B)}.$$
\end{definition}

There are two natural candidates for the definition of the boundary of a 0-simplex. One possibility is to define $\partial_0 (f) = 0$ for all $f\in S_0(\cA;B)$. Another possibility is to extend the definition of an $n$-simplex to $n=-1$; namely a $(-1)$-simplex $f$ is an object $f(\emptyset)$ in $\cC$. Then the definition of a boundary operator extends naturally to the operator $\bd_0:f\in S_0(\cA;B)\mapsto B$.

As we show in Lemma~\ref{h0}, computing the group $H_0$ in a specific context using the first definition gives $H_0\cong \mathbb{Z}$ while using the second definition we get $H_0= 0$. Thus, the difference between the approaches is parallel to that between the homology and reduced homology groups \cite{B}.

Next we define the amalgamation properties.  We use the convention that $[n]$ denotes the $(n+1)$-element set $\{0, 1, \ldots, n\}$.

\begin{definition}
Let $\cA$ be a non-empty amenable family of functors into a category $\cC$ and let $n\ge 1$.
\begin{enumerate}
\item $\cA$ has \emph{$n$-amalgamation} if for any functor $f : \P^-([n-1]) \rightarrow \cC$, $f\in \cA$, there is an $(n-1)$-simplex $g \supseteq f$ such that $g \in \cA$.

\item $\cA$ has \emph{$n$-complete amalgamation} or \emph{$n$-CA} if $\cA$ has $k$-amalgamation for every $k$ with $1 \leq k \leq n$.

\item $\cA$ has \emph{strong $2$-amalgamation} if whenever $f: \P(s) \rightarrow \cC$, $g: \P(t) \rightarrow \cC$ are simplices in $\cA$ and $f \upharpoonright \P(s\cap t) = g \upharpoonright \P(s \cap t)$, then $f \cup g$ can be extended to a simplex $h: \P(s\cup t) \rightarrow \cC$ in $\cA$.

\item $\cA$ has \emph{$n$-uniqueness} if for any functor $f: \P^-([n-1]) \rightarrow \cA$ and any two $(n-1)$-simplices $g_1$ and $g_2$ in $\cA$ extending $f$, there is a natural isomorphism $F: g_1 \rightarrow g_2$ such that $F \upharpoonright \dom(f)$ is the identity.

\end{enumerate}
\end{definition}

\begin{remark}
The definition of $n$-amalgamation can be naturally extended to $n=0$: $\cA$ has $0$-amalgamation if it contains a functor $f:\{\emptyset\} \to \cC$. This condition is satisfied for any non-empty amenable family of functors.
\end{remark}

\begin{definition}
We say that an amenable family of functors $\cA$ is \emph{non-trivial} if $\cA$ is non-empty, has 1-amalgamation, and satisfies the strong 2-amalgamation property.
\end{definition}

Everywhere below, we are dealing with {\em non-trivial} amenable families of functors. The following claim is immediate from the definitions; we include the proof because it illustrates a typical use of 1-amalgamation, strong 2-amalgamation, and other properties of the amenable families.

\begin{claim}\label{adding_vertex}
Let $\cA$ be a non-trivial amenable family of functors and let $f$ be an $n$-simplex with support $s=\{s_0,\dots,s_n\}$. For any $m\in \omega$ such that $m > s_i$ for $i=0,\dots n$, there is an $(n+1)$-simplex $h$ with support $s\cup \{m\}$ such that $\bd _{n+1}^{n+1}(h)=f$.
\end{claim}

\begin{proof}
Let $\cA$, $f$, and $m$ satisfy the assumptions of the claim. Since $\cA$ is closed under restrictions, the functor $f\upharpoonright \{\emptyset\}$ is in $\cA$. By 1-amalgamation, the functor $f\upharpoonright \{\emptyset\}$ has an extension to a functor $g:\P([0])\to \cC$, $g\in \cA$. If $\sigma$ is the natural functor from the category $\P(\{m\})$ to $\P(\{0\})$, then the functor $g':\P(\{m\})\to \cC$ defined by $g'=g\circ \sigma$ is isomorphic to $g$. Since $\cA$ is closed under isomorphisms, we have $g'\in \cA$. Finally, using strong 2-amalgamation, we obtain a simplex $h$ with support $\{s_0,\dots,s_n,m\}$ that extends $f$ and $g'$. By construction, $h$ is an $(n+1)$-simplex such that $\bd _{n+1}^{n+1}(h)=f$.
\end{proof}

In particular, we get the following corollary.

\begin{corollary}
If $\cA$ is a non-trivial amenable family of functors, then $\cA$ contains an $n$-simplex for each $n\ge 1$.
\end{corollary}

\subsection{Computing homology groups}
We now establish facts that will describe the general structure of the homology groups. The goal is to show that, under appropriate assumptions, the homology group $H_n$ is equal to the set of cosets $c+B_n(\cA;B)$ for a set of very simple $n$-cycles $c$. This, in turn, will help with the calculation of the groups in the model-theoretic examples.

We start by defining special kinds of $n$-chains which are useful for computing homology groups.

\begin{definition}
\label{shell}
If $n \geq 1$, an \emph{$n$-shell} is an $n$-chain $c$ of the form $$\pm \sum_{0 \leq i \leq n + 1} (-1)^i f_i,$$ where $f_0, \ldots, f_{n+1}$ are $n$-simplices such that whenever $0 \leq i < j \leq n+1$, we have $\bd^i f_j = \bd^{j-1} f_i$.
\end{definition}

Note that any $n$-shell is an $n$-cycle.
In addition, if $f$ is any $(n+1)$-simplex, then $\bd f$ is an $n$-shell.  An $n$-shell should be thought of as an attempt to create an $(n+1)$-simplex by piecing together the simplices $f_0, \ldots, f_{n+1}$ along their faces, and the equation $\bd^i f_j = \bd^{j-1} f_i$ says that we may make the appropriate identifications between these faces.

\begin{definition}
\label{fan}
 If $n \geq 1$, and \emph{$n$-fan} is
an $n$-chain of the form
$$\pm\sum_{i\in\{0,..,\widehat k,...,  n+1\}} (-1)^i f_i$$ for some $k \in [n+1]$, where the $f_i$ are $n$-simplices such that
whenever $0 \leq i < j \leq n$, we have  $\bd^i f_j = \bd^{j-1} f_i$. In other words,
an $n$-fan is an $n$-shell missing one term.
\end{definition}

\begin{remark}
\begin{enumerate}
\item If $c$ is an $n$-fan, then $\bd c$ is an $(n-1)$-shell.
\item $\cA$ has $n$-amalgamation if and only if every $(n-2)$-shell in $\cA$ is the boundary of an $(n-1)$-simplex in $\cA$. And $\cA$ has $n$-uniqueness if and only if every $(n-2)$-shell in $\cA$ is the boundary of at most one $(n-1)$-simplex in $\cA$ up to isomorphism.
\end{enumerate}
\end{remark}

\begin{lemma}
\label{shellcycle}
If $n \geq 2$ and $\cA$ has $n$-CA, then every
$(n-1)$-cycle is a sum of $(n-1)$-shells. Namely, for each $c\in Z_{n-1}(\cA; B)$, $c = \sum_i \alpha_i f_i$, there is a
finite collection of  $(n-1)$-shells  $c_i\in Z_{n-1}(\cA; B)$
such that $c=\sum_i (-1)^n \alpha_i c_i$.

Moreover, if $S$ is the support of the chain $c$ and $m$ is any element not in $S$, then we can choose $\sum_i \alpha_i c_i$ so that its support is $S \cup \{m\}$.
\end{lemma}

\begin{proof}
Suppose that $c = \sum_i \alpha_i f_i$ is a cycle in $Z_{n-1}(\cA; B)$, where each $f_i$ is an $(n-1)$-simplex, and $\alpha_i \in \mathbb{Z}$.
Let $I$ be the set of all pairs $(i,j)$ such that $f_i$ appears in $c$ and $0 \leq j \leq n-1$.  For each $(i,j) \in I$, let $g_{ij} = \bd^j f_i$
(so $g_{ij}$ is an $(n-2)$-simplex).

\begin{claim}\label{cancelling}
There are $(n-1)$-simplices $h_{ij} \in \cA_B$ for each pair $(i,j)\in I$ such that:

\begin{enumerate}
\item[(a)] If $\supp(g_{ij}) = s_{ij}$, then $\supp(h_{ij}) = s_{ij} \cup \{m\}$;
\item[(b)] If $g_{ij}=g_{k\ell}$, then $h_{ij}=h_{k\ell}$;
\item[(c)] For each $i$, $c_i:=(\sum_{0\leq j\leq n-1} (-1)^jh_{ij})+(-1)^{n}f_i$ is an $(n-1)$-shell.
(Namely, for $j'<j<n$, we have $\bd^{j'} h_{ij}=\bd^{j-1}h_{ij'}$ and $\bd^{j} f_i=g_{ij}=\bd^{n-1}h_{ij}$.)
\end{enumerate}
\end{claim}

\begin{proof}[Proof of Claim.]
First, pick any $0$-simplex $g^*$ in $\cA_B$ with support $\{m\}$. We define $h_{ij}$ with ``bottom face'' $g_{ij}$ as follows. First set  $h_{ij}\upharpoonright \P(s_{ij})=g_{ij}$. Then set $h_{ij}(\{m\})=g^*(\{m\})$. For $k\in s_{ij}$, we use $2$-amalgamation and let $h_{ij}(\{k,m\}) \in \cA$ be an amalgam of $g^*$ and $g_{ij}\upharpoonright \P(\{k\})$.  Also, for any $k \in s_{ij} \cap s_{pq}$ such that $g_{ij}(\{k\}) = g_{pq}(\{k\})$, we can assume that $h_{ij}(\{k,m\}) = h_{pq}(\{k,m\})$ (by choosing the same amalgam for each).

If $n=2$, stop here, and we have constructed all $h_{ij}$. Now assume $n>2$.
For $k, \ell \in s_{ij}$ with $k < \ell$, we can use $3$-amalgamation to pick a $2$-simplex $h_{ij}(\{k, \ell, m\}) \in \cA$ such that $$\bd(h_{ij}(\{k, \ell, m\}) = h_{ij}(\{\ell,m\}) - h_{ij}(\{k,m\}) + g_{ij}(\{k,\ell\}),$$ and again we can ensure that if $\{k,\ell\} \subseteq s_{ij} \cap s_{pq}$ then $h_{ij}(\{k,\ell,m\}) = h_{pq}(\{k,\ell,m\})$.

We can continue this procedure inductively, using $n$-CA
to build the simplices $h_{ij}(t)$ for $t \subseteq s_{ij} \cup \{m\}$ of size $4, 5, \ldots, n$, and we can ensure that $\bd^{n-1} h_{ij} = g_{ij}$ and that conditions (a) and (b) hold.  Let $(s_{ij})_k$ denote the $k$th element of $s_{ij}$ in increasing order (where $0 \leq k \leq n-3$).  If $j'<j\leq n-1$, due to our construction,
$$(\bd^{j'}h_{ij})\upharpoonright \P(s_{ij} \setminus \{(s_{ij})_{j'}\})=
\bd^{j'}g_{ij}=\bd^{j-1}g_{ij'}=(\bd^{j-1}h_{ij'})\upharpoonright \P(s_{ij'} \setminus \{(s_{ij'})_{j-1}\}).$$
(so, $t:=s_{ij} \setminus \{(s_{ij})_{j'}\}= s_{ij'} \setminus \{(s_{ij'})_{j-1}\}$). Notice then
$$t\cup\{m\}=\supp(\bd^{j'}h_{ij})=\supp(\bd^{j-1}h_{ij'}).$$
Then again by our construction, $\bd^{j'}h_{ij}=\bd^{j-1}h_{ij'}$.
Therefore $h_{ij}$ satisfies (c) as well, so we have proved the Claim.
\end{proof}

Now for the sum of $(n-1)$-shells $d:=\sum_i \alpha_i c_i$, we have
$$d=\sum_i \alpha_i\left[\left(\sum_{0\leq j\leq n-1} (-1)^jh_{ij}\right)+(-1)^{n}f_i\right]$$
$$= \left[\sum_i \sum_{0 \leq j \leq n-1} (-1)^j \alpha_i h_{ij}\right] + (-1)^{n} c.$$
Since $\bd c = 0$, and $h_{i j} = h_{k \ell}$ whenever $\bd^jf_i = \bd^\ell f_k$, the first term above is zero.
We have proved the lemma.
\end{proof}

The above lemma allows to make the following conclusions about the structure of the groups $H_n(\cA; B)$.

\begin{corollary}\label{shellgen} Assume $\cA$ has $n$-CA for some $n\geq 2$.
Then $H_{n-1}(\cA; B)$ is generated by $$\{[c] \,  : \, c\mbox{ is an } (n-1)\mbox{-shell over } B\}.$$
In particular, if any $(n-1)$-shell over $B$ is a boundary, then so is any $(n-1)$-cycle.
\end{corollary}

\begin{corollary}
\label{trivial_homology}
If $\cA$ has $n$-CA for some $n \geq 3$, then $H_{n-2} (\cA; B) = 0$.
\end{corollary}

\begin{proof}
Since $\cA$ has $n$-amalgamation, any $(n-2)$-shell is  the boundary of
an $(n-1)$-simplex. Then due to $(n-1)$-CA and Corollary~\ref{shellgen},
 any $(n-2)$-cycle over $B$ is a  boundary.
\end{proof}

In fact, Corollary~\ref{shellgen} can be strengthened. We show in Theorem~\ref{Hn_shells} that if $\cA$ has $(n+1)$-CA for some $n \geq 1$, then the group $H_n(\cA; B)$ has the universe $\{ [c] : c \textup{ is an } n\textup{-shell over } B \textup{ with support } [n+1] \}$. That is, we are able to show that a linear combination of $n$-shells with integer coefficients is equal, up to a boundary, to an $n$-shell; and we show that the support can be restricted to a prescribed set. To accomplish this, we need to introduce an additional notion.

\begin{definition}
If $n\ge 1$, an \emph{$n$-pocket} is an $n$-cycle of the form $f-g$, where $f$ and $g$ are $n$-simplices with support $S$ (where $S$ is an $(n+1)$-element set).
\end{definition}

The condition that the boundary of an $n$-pocket $f-g$ is zero implies that $f\restriction \P(s) = g\restriction \P(s)$ for each $n$-element subset $s$ of $S$.

\begin{lemma}\label{trivial_pocket}
Let $\cA$ be a non-trivial amenable family of functors and suppose that $f,g\in S_n(\cA)$ are \emph{isomorphic} functors such that $\partial_n f = \partial_n g$. Then the $n$-pocket $f-g$ is a boundary.
\end{lemma}

\begin{proof}
We may assume that the support of both $f$ and $g$ is $[n]$. Using Claim~\ref{adding_vertex}, we can pick an $(n+1)$-simplex $\widehat{f}$ with the support $[n+1]$ such that $\bd^{n+1}_{n+1}\widehat{f}=f$. Let $\langle \alpha_u:g(u)\to f(u) \mid u\in \P([n])\rangle$ be a family of isomorphisms in $\cC$ that witness the isomorphism of $f$ and $g$.
Define an $(n+1)$-simplex $\widehat{g}$ by letting, for $u\subset v\in \P([n+1])$,
$$
\widehat{g}(u):=\begin{cases} \widehat{f}(u),  & u\ne [n], \\
                            g([n]), & u=[n],
            \end{cases}
$$
and
$$
\widehat{g}(\iota_{u,v}):=\begin{cases} g(\iota_{u,v}), & u,v\in \P([n]),\\
                                    \widehat{f}(\iota_{u,v})\circ \alpha_u, & u\in \P([n]), v\notin \P([n]),\\
                                    \widehat{f}(\iota_{u,v}), & u,v\notin \P([n]).
                        \end{cases}
$$
It is routine to check that $\widehat{g}$ is indeed an $(n+1)$-simplex and that $f-g$ is the boundary of the $(n+1)$-chain $(-1)^{n+1}(\widehat{f}-\widehat{g})$.
\end{proof}

\begin{lemma}
\label{fans and shells}
Suppose that $n \ge 1$ and $\cA$ has $(n+1)$-amalgamation.  Then for any $n$-fan

$$g = \pm\sum_{i\in\{0, \ldots,\widehat k, \ldots,  n+1\}} (-1)^i f_i$$
there is some $n$-simplex $f_k$ and some $(n+1)$-simplex $f$ such that $g + (-1)^k f_k = \bd f$.

\end{lemma}

\begin{proof}
Without loss of generality, $\supp(g) = [n+1]$ and $k = n+1$. Because $\cA$ is closed under unions, the union $\widetilde{f}$ of all the simplices $f_i$ (for $i \in [n]$) is also in $\cA$.  Let $h=\widetilde{f} |_{\{n+1\}}$, which is in $\cA$ by localization.  By $(n+1)$-amalgamation, there is a functor $\hat f \in \cA$ with support $[n]$ extending $h$.  Finally, applying de-localization to $\hat f$, we obtain a functor $ f : \P([n+1]) \rightarrow \cC$ in $\cA$ such that for every $i$ with $0 \leq i \leq n$,

$$ f \upharpoonright \P(\{0, \ldots, \hat i, \ldots, n+1 \}) = f_i.$$

Letting $f_{n+1} = f \upharpoonright \P([n])$, we get the result.

\end{proof}

The next lemma says that $n$-pockets are equal to $n$-shells, ``up to a boundary.''

\begin{lemma}\label{shells and pockets}
Let $\cA$ be a non-trivial amenable family of functors that has the $(n+1)$-amalgamation property for some $n\ge 1$. For any $B\in \cC$, any $n$-shell in $\cA_B$ with support $[n+1]$ is equivalent, up to a boundary in $B_n(\cA; B)$, to an $n$-pocket in $\cA_B$ with support $[n]$. Conversely, any $n$-pocket with support $[n]$ is equivalent, up to a boundary, to an $n$-shell with support $[n+1]$.

\end{lemma}

\begin{proof}
Suppose $c$ is an $n$-shell with support $[n+1]$ and $$c = \pm \sum_{0 \leq i \leq n + 1} (-1)^i f_i$$ where $\supp(f_i) = [n + 1] \setminus \{i\}$.  Applying Lemma~\ref{fans and shells} to the $n$-fan $c - f_{n+1}$, we obtain a second $n$-shell

$$
c' = \left(\pm \sum_{0 \leq i \leq n} (-1)^i f_i\right) + (-1)^{n+1} f'
$$
such that $c'$ is a boundary.  Then $c - c' = \pm (f_{n+1} - f')$ is an $n$-pocket.

Conversely, let $h-h'$ be an $n$-pocket with support $[n]$. By Claim~\ref{adding_vertex}, there is
$\hat h$ extending $h$ with support $[n+1]$ such that $\partial^{n+1}\hat h=h$.  Then clearly
$$d:=\partial \hat h- (-1)^{n+1}(h-h')= \sum^{n+1}_{i=0}(-1)^i\partial^i\hat h -(-1)^{n+1}(h-h')$$
is an $n$-shell equivalent to $\pm (h-h')$.
\end{proof}

From Corollary~\ref{shellgen} and Lemma~\ref{shells and pockets} we get the following.

\begin{corollary}
If $\cA$ has $3$-amalgamation, then $H_2(\cA; B)$ is generated by equivalence classes of $2$-pockets.
\end{corollary}

\begin{lemma}[Prism lemma]
\label{prism} Let $n\ge 1$.
Suppose that $\cA$ is a non-trivial amenable family of functors that has $(n+1)$-amalgamation.
Let $f-f'$ be an $n$-pocket in $\cA_B$ with support $s$, where $|s|=n+1$. Let $t$ be an $(n+1)$-element set disjoint from $s$.
Then given  $n$-simplex  $g$ in $\cA_B$ with the domain $\P(t)$, there is an $n$-simplex $g'$  such that $g-g'$ forms an $n$-pocket  in $\cA_B$
and  is equivalent, modulo $B_n(\cA; B)$, to the pocket $f-f'$. We may choose $g'$ first and then find $g$ to have the same conclusion.
\end{lemma}

\begin{proof}
We begin with some auxiliary definitions.

\begin{definition}\label{n-path}
Let $f:\P(\{0,\dots,n\})\to \cC$ and $g:\P(\{n+1,\dots,2n+1\})\to \cC$ be $n$-simplices. An \emph{$(n+1)$-path from $f$ to $g$} is a chain $p$ of $(n+1)$-simplices
$$
p=\sum_{k=0}^n (-1)^{n(k+1)}h_k
$$
such that:
\begin{enumerate}
\item
$\dom (h_k) = \P(\{k,\dots,k+n+1\})$ for each $k=0,\dots,n$;
\item
$\partial^0 h_n = g$ and $\partial^{n+1} h_0 = f$; and
\item
$\partial^0 h_k = \partial^{n+1} h_{k+1}$ for each $k=0,\dots,n-1$.
\end{enumerate}
\end{definition}

\begin{definition}
Let $f:\P(\{0,\dots,n\})\to \cC$ and $g:\P(\{n+1,\dots,2n+1\})\to \cC$ be $n$-simplices; let $p=\displaystyle \sum_{k=0}^n (-1)^{n(k+1)}h_k$ be an $(n+1)$-path from $f$ to $g$.

(1)\
For each $k=0,\dots,n$, the $(n+1)$-simplex $h_k$ will be called the \emph{$k$-th leg of the path}.

(2)\
For each leg $h_k$ of the path, the face $\partial^{n+1} h_k$ is the \emph{origin face of the leg}, $\partial^{0} h_k$ is the \emph{destination face of the leg}, and the remaining faces of the simplex $h_k$ are the \emph{wall faces}.
\end{definition}

With this terminology, the conditions (2) and (3) of Definition~\ref{n-path} are saying that the destination face of the last leg in the path is the simplex $g$, the origin face of the first leg is $f$, and that the destination face of $k$th leg is the same as the origin face of the $(k+1)$st leg.

\begin{claim}\label{path_claim}
Let $p = \displaystyle \sum_{k=0}^n (-1)^{n(k+1)} h_k$ be a path from $f$ to $g$. Then
$$
\partial p = g-f+\sum_{i=1}^n\sum_{k=0}^n (-1)^{i+n(k+1)} \partial^ih_k.
$$
That is, the boundary of the path is $g-f$ plus a linear combination of the wall faces of the path.
\end{claim}

\begin{proof}
Using the definition of the boundary of a chain and changing the order of summation, we have:
\begin{multline*}
\partial p = \partial \left(\sum_{k=0}^n (-1)^{n(k+1)} h_k\right)
\\
  =\sum_{k=0}^n\sum_{i=0}^{n+1} (-1)^{i+n(k+1)}\partial^ih_k
   =\sum_{i=0}^{n+1} \sum_{k=0}^n (-1)^{i+n(k+1)}\partial^ih_k
\\
  =\sum_{k=0}^n (-1)^{n(k+1)} \partial^0h_k +\sum_{i=1}^{n} \sum_{k=0}^n (-1)^{i+n(k+1)}\partial^ih_k
    +\sum_{k=0}^n (-1)^{n+1+n(k+1)}\partial^{n+1}h_k.
\end{multline*}
It remains to show that
$$
\sum_{k=0}^n (-1)^{n(k+1)} \partial^0h_k +\sum_{k=0}^n (-1)^{n+1+n(k+1)}\partial^{n+1}h_k = g-f.
$$
Indeed, the first sum can be rewritten as follows:
\begin{multline*}
\sum_{k=0}^n (-1)^{n(k+1)} \partial^0h_k
\\
= \sum_{k=0}^{n-1}(-1)^{n(k+1)}\partial^0 h_k + (-1)^{n(n+1)}\partial^0h_n
=\sum_{k=0}^{n-1}(-1)^{n(k+1)}\partial^0 h_k + g.
\end{multline*}
Taking into account that $(-1)^{n+1+n(k+1)} = -(-1)^{nk}$, we rewrite the second sum as
$$
\sum_{k=0}^n -(-1)^{nk}\partial^{n+1}h_k = -f - \sum_{k=1}^n (-1)^{nk} \partial^{n+1}h_k
=-f -\sum_{k=0}^{n-1} (-1)^{n(k+1)} \partial^{n+1}h_{k+1}.
$$
Condition (3) in Definition~\ref{n-path} implies that
$$
\sum_{k=0}^{n-1}(-1)^{n(k+1)}\partial^0 h_k = \sum_{k=0}^{n-1} (-1)^{n(k+1)} \partial^{n+1}h_{k+1},
$$
so the claim follows.
\end{proof}

Here is the plan for the rest of the proof of Lemma~\ref{prism}. We start with the $n$-pocket $f-f'$ in $\cA_B$ whose support is $s$ for an $(n+1)$-element set $s$.  Then given an arbitrary $n$-simplex $g \in \cA_B$ with domain $t$, we build a path $p$ from $f$ to $g$ in $\cA_B$. This is done in Claim~\ref{second_prism} below.

Next we use the walls of the path $p$ and the the simplex $f'$ to gradually build path $p'$ in $\cA_B$ to some simplex $g'$ such that the walls of $p$
are the same as the walls of $p'$. It will follow immediately from the Claim~\ref{path_claim} that
$$
f-f' +\partial (p-p') = g-g',
$$
since the walls of the paths will cancel.

\begin{claim}
\label{second_prism}
If $\cA$ is a non-trivial amenable family of functors, and $f: \P([n]) \rightarrow \cC$ and $g: \P(\{n+1, \ldots, 2n+1\})$ are $n$-simplices in $\cA$, then there is an $(n+1)$-path from $f$ to $g$.
\end{claim}

\begin{proof}
For $k=0,\dots,n$, let $U_k=\{k,\dots,k+n+1\}$.  We build $(n+1)$-simplices $h_k : \P(U_k) \rightarrow \cC$ in $\cA_B$ by induction on $k$.  For $k=0$, use strong $2$-amalgamation to find an $h_0 \in \cA_B$ extending both $f$ and $g \upharpoonright \P(\{n+1\})$.  Given $h_{k-1}$ in $\cA_B$, we can use strong $2$-amalgamation again to build $h_k : \P(U_k) \rightarrow \cC$ in $\cA_B$ such that $\bd^{n+1} h_k = \bd^0 h_{k-1}$ and $h_k$ extends $g \upharpoonright \P(\{n+1, n+2, \ldots, k+n+1 \})$.  Then $p=\displaystyle \sum_{k=0}^n (-1)^{n(k+1)}h_k$ is a path from $f$ to $g$.
\end{proof}

Now we construct a path from $f'$ to some $n$-simplex $g'$ using the walls of $p$. The walls of $p$ are the following simplices:
$$
\{\partial^i h_k \mid i=1,\dots,n;\ k=0,\dots,n\}.
$$
By induction on $k=0,\dots,n$, we construct simplices $h'_k$ in $\cA_B$ such that:
\begin{enumerate}
\item
$\partial^{n+1}h'_0 = f'$;
\item
for $i=1,\dots,n$ and $k=0,\dots,n$ we have $\partial^i h'_k = \partial^i h_k$;
\item
$\partial^0 h'_k = \partial^{n+1} h'_{k+1}$ for each $k=0,\dots,n-1$.
\end{enumerate}

For $k=0$, consider the $(n-1)$-simplices $(\partial^ih_0)|_{\{0\}}$, for $i=1,\dots,n$, and the $(n-1)$-simplex $f'|_{\{0\}}$ (that is, we take the walls of the 0th leg of the path and the starting face, which are all $n$-simplices, and localize them at the corner $\{0\}$).
Using $(n+1)$-amalgamation, we can embed these into an $n$-simplex $h''_0$ in $\cA_{f'(\{0\})}$.   Then we apply de-localization to $h''_0$ to get $h'_0 \in \cA_B$ into which the wall faces and the starting face are all embedded. Taking the face $\partial^0h'_0$, we now have the starting face for the leg $h'_1$, and so on, until we get to $g'$.  This completes the proof of Lemma~\ref{prism}. The same argument works when $g'$ is chosen first.

\end{proof}

\begin{corollary}
\label{Hn_generators} Let $n\ge 1$.
Suppose $\cA$ is a non-trivial amenable family that has $(n+1)$-CA.
The group $H_n(\cA; B)$ is generated by equivalence classes $n$-shells with support $[n+1]$.
\end{corollary}

\begin{proof}
We know by Corollary~\ref{shellgen} above that $H_n(\cA; B)$ is generated by the equivalence classes of $n$-shells.  By Lemma~\ref{shells and pockets}, each equivalence class containing an $n$-shell also contains an $n$-pocket. Each $n$-pocket is equivalent to an $n$-pocket with support $[n]$, via one or two applications of the prism lemma above. The conclusion now follows from the converse clause of Lemma~\ref{shells and pockets}.
\end{proof}

We have a shell version of the prism lemma as well.

\begin{lemma}[Prism lemma, shell version]
\label{shellprism}
Let $\cA$ be a non-trivial amenable family of functors that satisfies $(n+1)$-amalgamation for some $n \geq 1$.
Suppose that an $n$-shell $f:=\sum_{0\leq i\leq n+1} (-1)^{i}f_i$ and an $n$-fan
$g^-:=\sum_{i\in \{0,...,\hat k,... ,n+1\}} (-1)^{i}g_i$ are given, where $f_i,g_i$ are $n$-simplices over $B$, $\supp(f)=s$ with $|s|=n+2$, and $\supp(g^-)=t=\{t_0,...,t_{n+1}\}$, where $t_0<...<t_{n+1}$ and $s\cap t=\emptyset$. Then there is an
$n$-simplex $g_k$ over $B$ with support $\bd_k t := t \setminus \{t_k\}$ such that  $g:=g^-+(-1)^{k}g_{k}$
is  an $n$-shell over $B$ and  $f-g\in B_n(\cA;B)$.
\end{lemma}

\begin{proof}
Assume $f$ and
$g^-$ are given, as supposed.
Now by Lemma~\ref{fans and shells}, there are an $n$-simplex $f'_k$ with $\dom(f'_{k})=\dom(f_{k})$ and an $(n+1)$-chain (indeed a single $(n+1)$-simplex) $c$, such that $f+(-1)^{k}(f'_{k}-f_{k})=\bd c$, i.e. $f$ is equivalent to
an $n$-pocket $(-1)^{k+1}(f'_{k}-f_{k})$. Again by Lemma~\ref{fans and shells},
there are an $(n+1)$-simplex $d$ with $\dom(d)=\P(t)$ and an $n$-simplex $g'_{k}$
such that $\bd d=g^-+(-1)^{k}g'_{k}.$ Hence
$$f-g^-=(-1)^{k+1}(f'_{k}-f_{k})+(-1)^{k}g'_{k}+\bd(c-d).$$
Then by the prism lemma (Lemma~\ref{prism}), there is an $n$-simplex $g_{k}$ such that
$(-1)^{k+1}(f'_{k}-f_{k})+(-1)^{k}g'_{k}$ is equivalent to $(-1)^{k}g_{k}$ up to
a boundary. Hence $f$ is equivalent to $g=g^-+(-1)^{k}g_{k}$ up to a boundary.
Moreover, since $g-g'$ is a pocket, in particular a cycle, clearly $g$ is an
$n$-shell.
\end{proof}

The next theorem gives an even simpler standard form for elements of $H_n(\cA; B)$.  Note that it is a strengthening of Corollary~\ref{Hn_generators} above, which only says that $H_n(\cA; B)$ is \emph{generated} by shells with support $[n+1]$.

\begin{theorem}
\label{Hn_shells}
If $\cA$ is a non-trivial amenable family of functors with $(n+1)$-CA for some $n \geq 1$, then $$H_n(\cA; B) = \left\{ \left[c\right] : c \textup{ is an } n\textup{-shell over } B \textup{ with support } [n+1] \right\}.$$
\end{theorem}

\begin{proof} By Corollary~\ref{Hn_generators}, it suffices to show the following: if $d$ and $e$ are any two $n$-shells in $\cA_B$ with support $[n+1]$, then $[d + e] = [d']$ for some $n$-shell $d'$ in $\cA_B$ with the same support.

First, pick any $n$-fan in $\cA_B$ $$c = f_{\widehat{0}} - f_{\widehat{1}} + \ldots + (-1)^n f_{\widehat{n}}$$ such that the domain of $f_{\widehat{i}}$ is $\P(\{0, \ldots, \widehat{i}, \ldots, n+1\})$.  Applying the
shell prism lemma to $d$ and to $e$ separately, we see that we can assume (up to equivalence modulo $B_n(\cA; B)$) that $d = c + (-1)^{n+1} g$ and $e = c + (-1)^{n+1} h$ for some $n$-simplices $g$ and $h$.  By Lemma~\ref{fans and shells}, we can pick another $n$-simplex $f_{\widehat{n+1}}$ such that $d_0 := c + (-1)^{n+1} f_{\widehat{n+1}}$ is in $B_n(\cA; B)$.

Next, use Lemma~\ref{fans and shells} two more times to pick an $n$-simplices $k_0$ and $k_1$ such that $$d_1 := k_0 - f_{\widehat{1}} + f_{\widehat{2}} - \ldots + (-1)^{n+1} g$$ and $$d_2 := f_{\widehat{0}} - k_1 + f_{\widehat{2}} - \ldots + (-1)^{n+1} h$$ are both in $B_n(\cA; B)$, where the ``$\ldots$'' is filled in with the appropriate $f_{\widehat{i}}$'s.  Finally, let $$d_3 := -k_0 + k_1 - f_{\widehat{2}} + \ldots + (-1)^n f_{\widehat{n+1}}.$$

Then since $d_0, d_1,$ and $d_2$ are in $B_n(\cA; B)$, $$[d_3] = [d_0 + d_1 + d_2 + d_3].$$ On the other hand, simply by canceling terms, we compute:

$$ d_0 + d_1 + d_2 + d_3 = 2 f_{\widehat{0}} - 2 f_{\widehat{1}} + 2 f_{\widehat{2}} - \ldots + 2(-1)^n f_{\widehat{n}} + (-1)^{n+1} g  + (-1)^{n+1} h$$ $$= d + e.$$

\end{proof}

Now using Theorem~\ref{Hn_shells} and Lemma~\ref{shells and pockets}, we get the following.

\begin{corollary}\label{Hn_pockets}
If $\cA$ is a non-trivial amenable family with $(n+1)$-CA (for some $n \geq 1$), then
$$
H_n(\cA;B ) = \left\{ [c] : c \textup{ is an } n\textup{-pocket in } \cA \textup{ over $B$ with support } [n] \right\}.
$$
\end{corollary}

\section{Homology groups in model theory}

In this section, we define some amenable classes of functors that arise in model theory. The properties of the classes of functors
were the motivation for Definition~\ref{amenable}. Given either a complete rosy theory $T$ or a complete type $p \in S(A)$ in a rosy theory, we will define both ``type homology groups'' $H^t_n(T)$ (or $H^t_n(p)$) or the ``set homology groups'' $H^{set}_n(T)$ (or $H^{set}_n(p)$).  As we show below, these definitions will lead to isomorphic homology groups.

For the remainder of this paper, we assume that $T$ is a rosy theory having elimination of hyperimaginaries.   The reason for this  is so that we have a nice independence notion.  Throughout, ``independent'' or ``nonforking'' will mean independence with respect to thorn nonforking. But the assumption is for convenience not for full generality.
For example if $T$ is   simple, then one may assume the independence is usual nonforking in
$\C^{heq}$ while replacing $\acl$ by $\bdd$ and so on.  But due to elimination of hyperimaginaries
 thorn forking is equivalent to usual forking in simple $T$ \cite{EO}. Moreover there are non-rosy examples having suitable independence notions that fit in our amenable category  context
\cite{KK}.

\subsection{Type homology}

We will work with $*$-types -- that is, types with possibly infinite sets of variables -- and to avoid some technical issues, we will place an absolute bound on the cardinality of the variable sets of the types we consider.  Fix some infinite cardinal $\kappa_0 \geq |T|$. We will assume that every $*$-type has no more than $\kappa_0$ free variables.  We also fix a set $\cV$ of variables such that $|\cV| > \kappa_0$ and assume that all variables in $*$-types come from the set $\cV$ (which is a ``master set of variables.'')  We work in a monster model $\C=\C^{eq}$ which is saturated in some cardinality greater than $2^{|\cV|}$.  We let $\bar \kappa=|\C|$. As we will see in the next section, the precise values of $\kappa_0$ and $|\cV|$ will not affect the homology groups.

Given a set $A$, strictly speaking we should write ``\textbf{a} complete $*$-type of $A$'' instead of ``\textbf{the} complete $*$-type of $A$'' -- there are different such types corresponding to different choices of which set of variables to use, and this distinction is crucial for our purposes.

If $X$ is any subset of the variable set $\cV$, $\sigma: X \rightarrow \cV$ is any injective function, and $p(\overline{x})$ is any $*$-type such that $\overline{x}$ is contained in $X$, then we let $$\sigma_* p = \left\{\varphi(\sigma(\overline{x})) : \varphi(\overline{x}) \in p \right\}.$$

\begin{definition}
If $A$ is a small subset of the monster model, then $\cT_A$ is the category such that
\begin{enumerate}
\item The objects of $\cT_A$ are are all the complete $*$-types in $T$ over $A$, including (for convenience) a single ``empty type'' $p_\emptyset$ with no free variables;

\item $\mor_{\cT_A}(p(\overline{x}), q(\overline{y}))$ is the set of all injective maps $\sigma : \overline{x} \rightarrow \overline{y}$ such that $\sigma_*(p) \subseteq q$
\end{enumerate}


\end{definition}

Note that this definition gives a notion of two types $p(\overline{x})$ and $q(\overline{y})$ being ``isomorphic:'' namely, that $q$ can be obtained from $p$ by relabeling variables.

\begin{definition}
\label{type_simplex}
If $A$ is a small subset of the monster model, a \emph{closed independent type-functor based on $A$} is a functor $f:X \to \cT_A$ such that:
\begin{enumerate}
\item $X$ is a downward-closed subset of $\P(s)$ for some finite $s \subseteq \omega$.

\item
Suppose $w \in X$ and $u,v \subseteq w$.  Let $\overline{x}_t$ be the variable set of $f(t)$ and (whenever $r \subseteq t \subseteq s$) let $(f^r_t)_* : p_r(\overline{x}_r) \rightarrow p_t(\overline{x}_t)$ be the image of the inclusion map under the functor $f$.   Then whenever $\overline{a}$ realizes the type $f(w)$ and $\overline{a}_u$, $\overline{a}_v$, and $\overline{a}_{u \cap v}$ denote subtuples corresponding to the variable sets $f^u_w (\overline{x}_u)$, $f^v_w (\overline{x}_v)$, and $f^{u \cap v}_w (\overline{x}_{u \cap v})$, then

$$\overline{a}_u \nonfork _{A \cup \overline{a}_{u \cap v}} \overline{a}_v.$$

\item
For all non-empty $u\in X$ and any $\overline{a}$ realizing $f(u)$, we have (using the notation above) $\overline{a} = \acl\left(A \cup \bigcup_{i\in u} \overline{a}_{\{i\}}\right)$.
\end{enumerate}

(The adjective ``closed'' in the definition refers to the fact that, by (3), all the types $f(u)$ are $*$-types of algebraically closed tuples.)

Let $\cA^t(T; A)$ denote all closed independent type-functors based on $A$.

\end{definition}

\begin{remark}
\label{independence}
It follows from the definition above and the basic properties of nonforking that if $f$ is a closed independent type-functor based on $A$ and $u \in \dom(f)$ is a set of size $k$, then $f(u)$ is the type of the algebraic closure of an $B$-independent set $\{\overline{a}_1, \ldots, \overline{a}_k\}$, where $B$ is some realization of the type $f^{\emptyset}_u(f(\emptyset))$ -- namely, let $\overline{a}_i$ be the subtuple of some realization $\overline{a}$ of $f(u)$ corresponding to the variables in $f^{\{i\}}_u(\overline{x}_{\{i\}})$.
\end{remark}

\begin{definition}
If $A = \acl(A)$ is a small subset of the monster model and $p \in S(A)$, then a \emph{closed independent type-functor in $p$} is a closed independent type-functor $f: X \rightarrow \cT_A$ based on $A$ such that if $X \subseteq \P(s)$ and $i \in s$, then $f(\{i\})$ is the complete $*$-type of $\acl(C \cup \{b\})$ over $A$, where $C$ is some realization of $f(\emptyset)$ and $b$ is some realization of a nonforking extension of $p$ to $C$.

Let $\cA^t(p)$ denote all closed independent type-functors in $p$.

\end{definition}

\begin{proposition}
The sets $\cA^t(T;A)$ and $\cA^t(p)$ are non-trivial amenable families of functors.
\end{proposition}

\begin{proof}
The proofs are essentially the same for the two classes of functors, and we point out the differences below.  We will prove that these classes are isomorphism invariant, closed under restrictions and unions (see Remark~\ref{amenable2}), and closed under localization and de-localization.

Isomorphism invariance, closure under restrictions, and closure under unions are simple to check directly from the definitions, and closure under localizations just comes down to the fact that for any independent set of elements $A$ and any $B \subseteq A$, the set $A \setminus B$ is independent over $B$.  For closure under de-localization, given the functors $f$ and $g$ as in condition~(3) of Definition~\ref{amenable}, the idea behind constructing the functor $g_0$ is as follows: first, following Remark~\ref{independence}, suppose that $|t| = k$ and $f(t)$ is the type of the $B$-independent set $\{a_1, \ldots, a_k\}$ where $B$ is a realization of $f(\emptyset)$ and, in the case of $\cA^t(p)$, each $a_i$ realizes a nonforking extension of $p$ over $B$.  Suppose that $u \cup v \in Z_0$ where $u \in Z$, $v \subseteq t$, and $|v| = \ell$.  Then the type $$\left(g^{\emptyset}_u \circ f^v_t \right) (f(v))$$ is the type of $\acl(C \cup \{b_1, \ldots, b_\ell\})$ for some $C$ and $b_1, \ldots, b_\ell$ such that $C$ is a realization of $f(\emptyset)$, the set $\{b_1, \ldots, b_\ell\}$ is $C$-independent, and, in the case of $\cA^t(p)$, each $b_i$ realizes a nonforking extension of $p$ to $C$.  Now $g(u)$ is the type of $$\acl(D \cup \{c_1, \ldots c_m\}),$$ where $m = |u|$, $D$ realizes $g(\emptyset)$, the set $\{c_1, \ldots, c_m\}$ is $D$-independent, and in the case of $\cA^t(p)$, each $c_i$ realizes a nonforking extension of $p$ over $D$.  We may assume that $\{b_1, \ldots, b_\ell\} \subseteq D$, and we let $g_0(u \cup v)$ be the type (over $A$) of the set $$\acl(C \cup \{b_1, \ldots, b_\ell; c_1, \ldots, c_m\}).$$  Similarly, we can define the maps $g_0(\iota_{x,y})$ by combining the images of inclusions under $f$ and under $g$: if $x = u \cup v$ and $y = u' \cup v'$, where $u \subseteq u' \in Z$ and $v \subseteq v' \subseteq t$, then we define $g_0(\iota_{x,y})$ by combining the maps $f(\iota_{v,v'})$ and $g(\iota_{u,u'})$.  Again, it is clear that the resulting functor $g_0$ is a closed independent type-functor since for any $x \in \dom(g_0)$, $g_0(x)$ is the type over $A$ of the algebraic closure of some $C$-independent set.

Finally, strong $2$-amalgamation for $\cA^t(T;A)$ and $\cA^t(p)$ follows from the existence of nonforking extensions.

\end{proof}

\begin{definition}
If $A$ is a small subset of $\C$, then we write $S_n\cT_A$ as an abbreviation for $S_n(\cA^t(T;A); p_\emptyset)$ (the collection of closed $n$-simplices in $\cA^t(T;A)$ over the empty type $p_\emptyset$), $B_n \cT_A$ and $Z_n\cT_A$ for the boundary and cycle groups, and $H^t_n(T;A)$ for the homology group $H_n(\cA^t(T;A); p_\emptyset)$.

If $A = \acl(A)$ and $p \in S(A)$, then we use the abbreviation $S_n\cT(p)$ for the collection of all closed $n$-simplices in $\cA^t(T;A)$ over $p_\emptyset$ of type $p$, and similarly we use the abbreviations $B_n \cT(p)$, $Z_n\cT(p)$, and $H^t_n(p)$.
\end{definition}

\subsection{Set homology}

\begin{definition}
Let $A$ be a small subset of $\C$.  By $\cC_A$ we denote the category of all subsets (not necessarily algebraically closed) of $\C$ of size no more that $\kappa_0$, where morphisms are partial elementary maps over $A$ (that is, fixing $A$ pointwise).
\end{definition}

\begin{definition}
A \emph{closed independent set-functor based on $A$} is a functor $f:X\to \cC_A$ such that:
\begin{enumerate}

\item $X$ is a downward-closed subset of $\P(s)$ for some finite $s \subseteq \omega$; and
$f(\emptyset) \supseteq A$.

\item
If $w \in X$ and $u,v \subseteq w$, and if $f^x_y$ denotes the image under $f$ of the inclusion map $x \subseteq y$ in $\P(s)$, then $$f^u_w(u)\nonfork _{A \cup f^{u \cap v}_w(u\cap v)} f^v_w(v).$$
\item
For all non-empty $u\in X$, we have $f(u) = \acl(A \cup \bigcup_{i\in u} f^{\{i\}}_u(\{i\}))$.
\end{enumerate}

Let $\cA^{set}(T; A)$ denote all closed independent set-functors based on $A$.

\end{definition}

\begin{definition}
If $A = \acl(A)$ is a small subset of the monster model and $p \in S(A)$, then a \emph{closed independent set-functor in $p$} is a closed independent set-functor $f: X \rightarrow \cC_A$ based on $A$ such that if $X \subseteq \P(s)$ and $i \in s$, then $f(\{i\})$ is a set of the form $\acl(C \cup \{b\})$ where $C = f^{\emptyset}_{\{i\}}(f(\emptyset))\supseteq A$ and $b$ realizes some nonforking extension of $p$ to $C$.

Let $\cA^{set}(p)$ denote all closed independent set-functors in $p$.

\end{definition}

Just as in the previous subsection (and by an identical argument), we have:

\begin{proposition}
The sets $\cA^{set}(T;A)$ and $\cA^{set}(p)$ are non-trivial amenable families of functors.
\end{proposition}

\begin{definition}
If $A$ is a small subset of $\C$, then we write ``$S_n \cC_A$'' for $S_n(\cA^{set}(T;A); A)$ (the collection of closed $n$-simplices in $\cA^{set}(T;A)$ {\em over} $A$), and similarly we write $B_n \cC_A$ and $Z_n \cC_A$ for the boundary and cycle groups over $A$, and use the notation $H^{set}_n(T;A)$ for the homology group $H_n(\cA^{set}(T;A); A)$.

If $A = \acl(A)$ and $p \in S(A)$, then we use similar abbreviations $S_n \cC(p):=S_n(\cA^{set}(p); A)$, $B_n\cC(p)$, $Z_n\cC(p)$, and $H^{set}_n(p)$ for the type $p$ versions of the groups {\em over $A$} above.
\end{definition}

\begin{proposition}
\label{homology_equivalence}
\begin{enumerate}
\item For any $n$ and any $A \in \cC$, $H^t_n(T; A) \cong H_n^{set}(T; A)$.
\item For any $n$ and any complete type $p \in S(A)$, $H^t_n(p) \cong H_n^{set}(p)$.
\end{enumerate}
\end{proposition}

\begin{proof}

The idea is that we can build a correspondence $F: S \cC_A \rightarrow S \cT_A$ which maps each set-simplex $f$ to its ``complete $*$-type'' $F(f)$.  Note that this will involve some non-canonical choices: namely, which variables to use for $F(f)$, and in what order to enumerate the various sets in $f$ (since our variable set $\cV$ is indexed and thus implicitly ordered).  We will write out a proof of part~(1) of the proposition, and part~(2) can be proved similarly by relativizing to $p$.

Let $S_{\leq n} \cC_A$ and $S_{\leq n} \cT_A$ denote, respectively, $\bigcup_{i \leq n} S_i \cC_A$ and $\bigcup_{i \leq n} S_i \cT_A$.  We will build a sequence of maps $F_n : S_{\leq n} \cC_A \rightarrow S_{\leq n} \cT_A$ whose union will be $F$.  Given such an $F_n$, let $\ov{F}_n: C_{\leq n} \cC_A \rightarrow C_{\leq n} \cT_A$ be its natural extension to the class of all set-$k$-chains over $A$ for $k \leq n$.

\begin{claim}
There are maps $F_n : S_{\leq n} \cC_A \rightarrow S_{\leq n} \cT_A$ such that:

\begin{enumerate}
\item $F_{n+1}$ is an extension of $F_n$;
\item If $f \in S_{\leq n} \cC_A$ and $\dom(f) = \P(s)$, then $\dom(F_n(f)) = \P(s)$ and $\left[F_n(f)\right](s)$ is a complete $*$-type of $f(s)$ over $A$;
\item For any $k \leq n$, any $f \in S_k \cC_A$, and any $i \leq k$, $F_n(\bd^i f) = \bd^i \left[F_n(f)\right]$; and
\item $F_n$ is surjective, and in fact for every $g \in S_k\cT_A$ (where $0 \leq k \leq n$), there are \textbf{more} than $2^{|\cV|}$ simplices $f \in S_k\cC_A$ such that $F_n(f) = g$.
\end{enumerate}

\end{claim}

\begin{proof}

We prove the claim by induction on $n$.  The case where $n=0$ is simple: only conditions (2) and (4) are relevant, and note that we can insure (4) because the monster model $\C$ is $(2^{|\cV|})^+$-saturated and there are at most $2^{|\cV|}$ elements of $S_0 \cT_A$.  So suppose that $n >0$ and we have $F_0, \ldots, F_n$ satisfying all these properties, and we want to build $F_{n+1}$.  We build $F_{n+1}$ as a union of a chain of partial maps from $S_{\leq n+1} \cC_A$ to $S_{\leq n+1} \cT_A$ extending $F_n$ (that is, functions whose domains are subsets of $S_{\leq n+1} \cC_A$).

\begin{subclaim}
Suppose that $F: X \rightarrow S_{\leq n+1} \cT_A$ is a function on a set $X \subseteq S_{\leq n+1} \cC_A$ of size at most $(2^{|\cV|})^+$ and that $F$ satisfies (1) through (3).  Then for any simplex $g \in S_{n+1} \cT_A$, there is an extension $F_0$ of $F$ satisfying (1) through (3) such that $|\dom(F_0)| \leq (2^{|\cV|})^+$ and:
\begin{center}
(*) There are $(2^{|\cV|})^+$ distinct $f \in S_{n+1} \cC_A$ such that $F'(f) = g$.
\end{center}
\end{subclaim}

\begin{proof}
Let $\bd g = g_0 - g_1 + \ldots + (-1)^n g_n$ (where $g_i = \bd^i g$), and let $\P(s)$ be the domain of $g$.  By induction, each $g_i$ is the image under $F_n$ of $(2^{|\cV|})^+$ different $n$-simplices in $\cC_A$; let $\langle f^j_i : j < (2^{|\cV|})^+ \rangle$ be a sequence of distinct simplices such that for every $j < (2^{|\cV|})^+$, $F_n(f^j_i) = g_i$.   By saturation of the monster model, for each $j < (2^{|\cV|})^+$ we can pick an $(n+1)$-simplex $f_j \in \cC_A$ with domain $\P(s)$ such that $\bd f_j = f^j_0 - f^j_1 + \ldots + (-1)^n f^j_n$ and $\tp(f_j (s)) = g(s)$.  Then the $f_j$ are all distinct, so we can let $F_0 = F \cup \{(f_j, g) : j < (2^{|\cV|})^+\}$.
\end{proof}

Now by the subclaim, we can use transfinite induction to build a \textbf{partial} map $F' : S_{\leq n+1} \cC_A \rightarrow S_{\leq n+1} \cT_A$ satisfying (1) through (4) (also using the fact that there only (at most) $2^{|\cV|}$ different simplices in $S_{n+1} \cT_A$ and the fact that the union of a chain of partial maps from $S_{\leq n+1} \cC_A$ to $S_{\leq n+1} \cT_A$ satisfying conditions (2) and (3) will still satisfy these conditions).

Finally, we can extend $F'$ to a function on all of $S_{\leq n+1} \cC_A$ by a second transfinite induction, extending $F'$ to each $f: \P(s) \rightarrow \cC_A$ in $\cC_A$ one at a time; to ensure that properties (2) and (3) hold, we just have to pick $F_{n+1}(f)$ to be some $(n+1)$-simplex with the same domain $\P(s)$ whose $n$-faces are as specified by $F_n$ and such that $\left[F_{n+1}(f)\right](s)$ is a complete $*$-type of $f(s)$ over $A$.

\end{proof}

Now let $F = \bigcup_{n < \omega} F_n$.  By property~(3) above, it follows that for any chain $c \in C \cC_A$, we have $\ov{F}(\bd c) = \bd \left[\ov{F}(c) \right]$.  Hence $\ov{F}$ maps $Z_n\cC_A$ into $Z_n\cT_A$ and $B_n\cC_A$ into $B_n\cT_A$, and so $\ov{F}$ induces group homomorphisms $\varphi_n : H^{set}_n(T;A) \rightarrow H^t_n(T;A)$.  Verifying that $\varphi_n$ is injective amounts to checking that whenever $\ov{F}(c) \in B_n\cT_A$, the set-chain $c$ is in $B_n\cC_A$, but this is staightforward: if, say, $\ov{F}(c) = \bd c'$, then we can pick a set-simplex $\hat{c}$ ``realizing'' $c'$ such that $\bd \hat{c} = c$.  Finally, condition~(4) implies that $\varphi_n$ is surjective, so $H^{set}_n(T;A) \cong H^t_n(T;A)$.

\end{proof}

\begin{remark}
Since Proposition~\ref{homology_equivalence} is true for any choices of $\kappa_0$, $\cV$, and the monster model $\C$ as long as $|T| \leq \kappa_0 < |\cV|$ and $2^{|\cV|} \leq |\C|$, it follows that our homology groups (with the restriction of the set $A$) do not depend on the choices of $\kappa_0$, $|\cV|$, or the monster model.

Without specifying a base set $A$, one could also define $C_n(T)$ to be the direct sum $\bigoplus_{i<\bar\kappa} C_n \cC_{A_i}$ where $\{A_i | i < \bar \kappa\}$ is the collection of all small subsets of $\C$, and similarly $Z_n(T)$, $B_n(T)$, and
$H_n(T):=Z_n(T)/B_n(T)$.  Then the boundary operator $\bd$
sends $n$-chains to $(n-1)$-chains componentwise. Hence it follows
$H_n(T)=\bigoplus_{i<\bar\kappa} H_n(T;A_i).$ This means the homology groups defined without specifying a base set depends on the choice of monster model, and so this approach would not give invariants for the theory $T$.
\end{remark}

\subsection{An alternate definition of the set homology groups}

In our definition of the set homology groups $H^{set}_n(T; A)$ and $H^{set}_n(p)$ (where $p \in S(A)$), we have been implicitly assuming that the base set $A$ is fixed pointwise by all of the elementary maps in a set-simplex -- this is built into our definition of $\cC_A$, which says that morphisms are elementary maps \emph{over $A$}.  It turns out that we get an equivalent definition of the homology groups if we allow the base set to be ``moved'' by the images of the inclusion maps in a set-simplex, as we will show in this subsection.

\begin{definition}
\begin{enumerate}
\item A \emph{set-$n$-simplex weakly over $A$} is a set-$n$-simplex $f: \P(s) \rightarrow \cC$ over $\emptyset$ such that $f(\emptyset) = A$.
\item If $p \in S(A)$, then a set-$n$-simplex $f: \P(s) \rightarrow \cC$ is \emph{weakly of type $p$} if it is a closed simplex, $f(\emptyset) = A$, and for every $i \in s$, $$f(\{i\}) = \acl\left(f^{\emptyset}_{\{i\}}(A) \cup \{a_i\} \right)$$ for some $a_i$ such that $\tp(a_i / f^{\emptyset}_{\{i\}}(A))$ is a conjugate of $p$.
\end{enumerate}
\end{definition}

Let $C'_n \cC_A$ be the collection of all set-$n$-simplices weakly over $A$.  Note that the boundary operator $\bd$ takes an $n$-simplex weakly over $A$ to a chain of $(n-1)$-simplices weakly over $A$, and so we can define ``weak set homology groups over $A$,'' which we denote $H'_n(T; A)$.  Similarly, we can define $H'_n(p)$, the ``weak set homology groups of $p$,'' from chains of set-simplices that are weakly of type $p$.

\begin{proposition}
\label{homology_equivalence_2}
\begin{enumerate}
\item For any $n$ and any $A \in \cC$, $H'_n(T; A) \cong H_n^{set}(T; A)$.
\item For any $n$ and any complete type $p \in S(A)$, $H'_n(p) \cong H_n^{set}(p)$.
\end{enumerate}
\end{proposition}

\begin{proof}
As usual, the two parts have identical proofs, and we only prove the second part.

We will identify $S'_0\cC_A$ as a big single {\em complex} as follows.
Due to our cardinality assumption, for each $n<\omega$, there are $\bar \kappa$-many
0-simplices in   $S'_0\cC_A$ having the common domain $\P(\{ n\})$.  Then
we consider the following domain set
$\cD_0=\{\emptyset\}\cup\{\{(n,i)\}|\ n<\omega, i<\bar\kappa\}$. Now as said
we identify $S'_0\cC_A$ as    a {\em single} functor
 $F'_0$ from $\cD_0$ to $\cC$ such that
 $F'_0(\emptyset)=A$, and $F'_0(\{(n,i)\})=(f')^n_i(\{n\})$ where
 $(f')^n_i\in S'_0\cC_A$ is the corresponding
 0-simplex with $((f')^n_i)^{\emptyset}_{\{n\}}=(F'_0)^{\emptyset}_{\{(n,i)\}}$.
 Similarly we consider $S_0\cC_A$ as a functor $F_0$ from
 $\cD_0$ to $\cC_A$ such that
 $F_0(\emptyset)=A$, and $F_0(\{(n,i)\})=f^n_i(\{n\})\equiv (f')^n_i(\{n\})$
 where $f^n_i\in S_0\cC_A$ is the corresponding
 0-simplex over $A$ with $(f^n_i)^{\emptyset}_{\{n\}}=(F_0)^{\emptyset}_{\{(n,i)\}}$.
 Now $F'_0$ and $F_0$ are naturally isomorphic by $\eta^0$ with
 $\eta^0_{\emptyset}=$the identity map of $A$, and suitable
 $\eta^0_{\{(n,i)\}}$ sending $(f')^n_i(\{n\})$ to $f^n_i(\{n\})$.

Now for $S'_1\cC_A$, note that for each pair  $(f')^{n_0}_{i_0},(f')^{n_1}_{i_1} $ with $n_0<n_1$,
there are $\bar \kappa$-many
1-simplices $f'_j$ in   $S'_1\cC_A$ having the common domain $\P(\{ n_0,n_1\})$
with $\bd^0f'_j=(f')^{n_1}_{i_1}$ and $\bd^1f'_j=(f')^{n_0}_{i_0}$. Hence we now put the
domain set  $\cD_1=\cD_0\cup \{\{(n_0,i_0),(n_1,i_1),{j}\}|\ n_0<n_1<\omega; i_0,i_1,j<\bar\kappa\}$.
Then we identify  $S'_1\cC_A$ as    a functor
 $F'_1$ from $\cD_1$ to $\cC$ such that $F'_1\upharpoonright \cD_0=F'_{01}$, and
  $F'_1(\{(n_0,i_0),(n_1,i_1),{j}\})$ corresponds $j$th 1-simplex having
 $(f')^{n_0}_{i_0},(f')^{n_1}_{i_1} $ as 0-faces.  Similarly we try to identify
 $S'_1\cC_A$ as    a functor
 $F_1$ from $\cD_1$ to $\cC_A$, extending $F_0$. But to make $F'_1$ and $F_1$ isomorphic, we need extra care
 when defining $F_1$. For each $j<\bar \kappa$ and
 a set $a'_j=f'_j(\{n_0,n_1\})$ of  corresponding 1-simplex $f'_j$, assign
 an embedding $\eta^1_j=\eta^1_{\{(n_0,i_0),(n_1,i_1),{j}\}}$ sending
 $a'_j$ to $a_j$, extending the inverse of
 $(f'_j)^{\emptyset}_{\{n_0,n_1\}}$. Then we define
 $F_1(\{(n_0,i_0),(n_1,i_1),{j}\})=a'_j$, and
 $$(F_1)^{\{(n_k,i_k)\}}_{\{(n_0,i_0),(n_1,i_1),{j}\}}=
 \eta^1_j\circ(f'_j)^{\{n_k\}}_{\{n_0,n_1\}}\circ(\eta^0_{\{(n_k,i_k)\}})^{-1}.$$
 Now then clearly $\eta^1$ with $\eta^1\upharpoonright \cD_0=\eta^0$ is an isomorphism
 between $F'_1$ and $F_1$.

   By iterating this argument we can respectively identify $S'_n\cC_A$ and $S_n\cC_A$,
   as functors $F'_n$ and $F_n$ having the same domain $\cD_n$ extending $\cD_1$.
   Moreover we can also construct an isomorphism $\eta^n$, extending $\eta^1$, between $F'_n$ and $F_n$.
 Note that each $x\in \cD_n-\cD_{n-1}$ corresponds an $n$-simplex $f'\in S'_n\cC$, and $\eta^n_x$
  corresponds an $n$-simplex over $A$ $f\in S_n\cC$.
 This correspondence $f'\mapsto f$ induces a  bijection
from $C_n'\cC_A$ to $C_n\cC_A$, mapping $c'\mapsto c$,
which indeed is an isomorphism of the two groups.
Notice that  by the construction, if an $n$-shell $c'$ is the boundary of some
$(n+1)$-simplex $f'$, then $c$ is the boundary of $f$.
In general,
it follows $(\bd d)'=\bd d'$ \ \ (*).
Thus this correspondence also induces an isomorphism  between $Z'_n(T;A)$ and $Z_n(T;A)$.
  Moreover it follows from (*) that the correspondence sends $B'_n(T;A)$ to $B_n(T;A)$: Let
   $c'=\bd d'\in B'_n(T;A)$. Then by (*), we have $c=\bd d\in B_n(T;A).$
  Conversely for $c'\in Z'_n(T;A)$, assume $c=\bd e\in B_n(T;A).$
  Now for $e'$, again by (*),
 $\bd e'=c'$. Hence we have $c'\in B'_n(T;A) $.
 \end{proof}

\section{Homology groups and model-theoretic amalgamation properties: basic facts}

From now on, we will usually drop the superscripts $t$ and $set$ from $H^t_n(p)$ and $H^{set}_n(p)$, since these groups are isomorphic, and use ``$H_n(p)$'' to refer to the isomorphism class of these two groups.  In computing the groups below, we generally use $H^{set}_n(p)$ rather than $H^t_n(p)$.

First, we observe that $H_0$ does not give any information, since it is always isomorphic to $\mathbb{Z}$, if $\bd_0(f)=0$ for any $0$-simplex $f$; or is trivial if $\bd_0(f)$ is defined to be $f(\emptyset)$:

\begin{lemma}\label{h0}
\begin{enumerate}

\item If $\bd_0(f)=0$, then for any complete type $p$ over an algebraically closed set $A$, $H_0(p) \cong \mathbb{Z}$ and for any small subset $A$ of $\C$, $H_0(T; A) \cong \mathbb{Z}$.

\item If $\bd_0(f)=f(\emptyset)$, then both groups in (1) are trivial.

\end{enumerate}
\end{lemma}

\begin{proof}
Both parts of the lemma can be proved by essentially the same argument, so we only write out the proof for the group $H_0(p)$ in (1).

For the proof we will define an augmentation map $\epsilon$ as in
topology.  Since we can add parameters to the language for $A$, we can assume that $A = \emptyset$.

Define $\epsilon: C_0\cC(p) \to \mathbb{Z}$ by
$\epsilon(c)=\sum_i n_i$ for  a 0-chain $c=\sum_i n_if_i$ of type $p$. Then $\epsilon$ is
a homomorphism such that $\epsilon (b)=0$ for any 0-boundary $b$ (since $\epsilon (\bd f) = 0$ for any $1$-simplex $f$).  Thus
$\epsilon$ induces a homomorphism $\epsilon_*:H_0(p)\to \mathbb{Z}$. Note that \emph{any} $0$-chain $c$ is in $Z_0(p)$, so clearly $\epsilon_*$ is onto. We claim that $\epsilon_*$ is one-to-one, i.e. $\ker \epsilon_*= B_0(p)$.
Given a $0$-chain $c=\sum_{i \in I} n_if_i$  such that $\epsilon_*(c)=\sum_{i \in I} n_i=0$, we shall show $c$ is a boundary. Pick some natural number $m$ greater than every $k_i$  where  $\dom f_i=\P(\{k_i\})$.
Let $a_i=\acl(a_i)=f_i(\{k_i\})$.  Then choose $a$ realizing $p$ such that $a \ind \{a_i : i \in I\}$.
Now let $g_i$ be a closed 1-simplex of $p$ such that $\dom g_i=\P(\{ k_i,m\})$, $g_i(\{k_i\})=a_i$, and $g_i(\{m\})=a$. Then $\bd g_i=c_m-f_i$, where $c_m$ is the $0$-simplex such that $c_m(\emptyset) = \emptyset$ and $c_m(\{m\}) = a$.
Then $c+\bd(\sum_i n_i g_i)=\sum_i n_if_i+\sum_in_i(c_m-f_i)=(\sum_in_i)c_m=0.$ Hence $c$ is a 0-boundary,
and $H_0(p) \cong \mathbb{Z}$.

\end{proof}

For $k > 0$, the homology groups $H_k(T; A)$ and $H_k(p)$ are related to standard amalgamation properties.  The $n$-amalgamation and $n$-uniqueness properties for simple theories can be stated in terms of shells, and the following is equivalent to the usual definition.

\begin{definition}
Assume $T = T^{eq}$.

\begin{enumerate}
\item If $A$ is a small subset of $\C$, then $T$ has \emph{$n$-amalgamation property over $A$} if for every $(n-2)$-shell $c$ over $A$, there is an $(n-1)$-simplex $f$ such that $c = \bd f$.
\item The complete type $p$ has \emph{$n$-amalgamation} if every $(n-2)$-shell $c$ of type $p$, there is an $(n-1)$-simplex $f$ such that $c = \bd f$.
\item A theory $T$ has \emph{$n$-uniqueness based on $A$} if for every $(n-2)$-shell $c$ based on $A$ and any two $(n-1)$-simplices $f$ and $g$ such that $\bd f = \bd g = c$, $f$ and $g$ are isomorphic ``over $c$:'' that is, there is an isomorphism $\varphi: f \rightarrow g$ such that $\varphi$ induces the identity map between $f(u)$ and $g(u)$ whenever $u \subseteq \dom(f)$ has size less than $n-1$.  Similarly, we can define what it means for a type $p \in S(A)$ to have $n$-uniqueness by considering shells of type $p$.
\item A theory $T$ has \emph{$n$-amalgamation} if it has $n$-amalgamation over every small subset $A$ of $\C$.
\end{enumerate}
\end{definition}

\begin{remark}
Intuitively, the $n$-amalgamation property says that we can find a collection $S$ of $n$ independent points such that the algebraic closure of each proper subset of $S$ satisfies a certain specified type (and these types must satisfy obvious coherence conditions).  The mismatch between the ``$n$'' in $n$-amalgamation and the dimension of the simplex comes from the fact that there are $n$ vertices in an $(n-1)$-simplex, and it is the types of the vertices that we are trying to amalgamate.
\end{remark}

\begin{remark}
If $T$ is simple, then $T$ automatically has $n$-amalgamation for $n = 1, 2,$ or $3$: $1$-amalgamation is vacuous, $2$-amalgamation is equivalent to the existence of nonforking extensions, and $3$-amalgamation is by the Independence Theorem \cite{KP}.   If $T$ is stable, then $T$ has $2$-uniqueness by stationarity.
As well-known, non-simple rosy theory can not have $3$-amalgamation but it may have $n$-amalgamation for all $n\geq 4$ (e.g. the theory of dense linear ordering).
\end{remark}

\begin{definition}
A theory $T$ (or a complete type $p$) has \emph{$n$-complete amalgamation} (or ``$n$-CA'') if it has $k$-amalgamation for every $k \leq n$.
\end{definition}

Now we can restate Corollary~\ref{Hn_shells} above as:

\begin{proposition}\label{shellgennew} Assume $T$ has $n$-CA based on $A=\acl(A)$ for $n\geq 2$.
Then $H_{n-1}(T;A)=\{[c]|\ c\mbox{ is an } (n-1)-\mbox{shell over $A$  with support } [n]\}$.
In particular,
if any $(n-1)$-shell is a boundary
then so is any $(n-1)$-cycle.
\end{proposition}

\begin{corollary}
\label{trivial_homology}
Suppose $n \geq 3$.
Assume $T$ has $n$-CA based on  $A=\acl(A)$. Then $H_{n-2}(p) = 0$ for $p\in S(A)$,
and $H_{n-2}(T;A)=0$.

\end{corollary}

However, the converse of the above corollary is false in general: the theory of the random tetrahedron-free hypergraph does not have $4$-amalgamation, but all of its homology groups are trivial (Example~\ref{tet.free}).

\begin{corollary}\label{trivialH1}
If $T$ is simple, then $H_1(T; A)=0$ and $H_1(p) = 0$ for any complete type $p$ in $T$.
\end{corollary}

This result is extended to any 1-type in an $o$-minimal theory in Example~\ref{ominh1} below.

\section{Computing $H_2(p)$ (the ``Hurewicz theorem'')}

We assume throughout this section that $T$ is a stable theory and that $p$ is a strong type (without loss of generality, over the empty set).  We will prove that  the type homology group $H_2(p)$ is isomorphic to a certain automorphism group $\Gamma_2(p)$ defined below.  This can be thought of as an analogue of Hurewicz's theorem in algebraic topology, which says that for a path connected topological space $X$, the first homology group $H_1(X)$ is isomorphic to the abelianization of the  homotopy group $\pi_1(X)$.  Just as there is a higher-dimensional version of Hurewicz's theorem for $H_n(X)$ under the hypothesis that $X$ is $(n-1)$-connected, we hope that there is a higher-dimensional generalization of our result under the hypothesis that  the theory $T$ has $(n+1)$-complete amalgamation.  In other words, maybe $n$-CA is analogous to a topological connectedness property.

Throughout this section, ``$\ov{a}$'' denotes the algebraic closure of an element $a$, considered as a possibly infinite ordered \textbf{tuple}, but the choice of ordering will not play any important role in what follows.  Implicit in the argument below is that if $a \equiv a_0,$ then there are orderings $\ov{a}, \ov{a_0}$ of their algebraic closures such that $\ov{a} \equiv \ov{a_0}$.
Moreover, $\Aut(A/B)$ denotes the group of
elementary maps from $A$ {\em onto} $A$ fixing $B$ pointwise.

First, suppose that $C = \{a_i : i \in s\}$ is an independent set of realizations of the type $p$.  Pick some $a$ realizing $p$ such that $a \ind C$, and let $$\widetilde{a_s} := \overline{a_s} \cap \dcl\left(\bigcup_{i \in s} \overline{a, a_{s \setminus \{i\}}}\right).$$  Note that since $T$ is stable, by stationarity, the set $\widetilde{a}_s$ does not depend on the particular choice of $a$.

Fix some integer $n \geq 2$, and let $\{a_0, \ldots, a_{n-1}\}$ be an independent set of $n+1$ realizations of $p$.  Recall our notation that $[k] = \{0, \ldots, k\}$, so that $\widetilde{a}_{[n-1]} = \widetilde{a}_{\{0, \ldots, n-1\}}.$  Let $$B_n = \bigcup_{0 \leq i \leq n-1} \overline{a_{\{0, \ldots, \widehat{i}, \ldots, n-1\}}}.$$  Finally, we let $\Gamma_n(p) = \Aut(\widetilde{a}_{[n-1]} / B_n)$.

Note that $\widetilde{a}_{[n-1]}$ is a subset of $\acl(a_0, \ldots, a_{n-1})$, so $\Gamma_n(p)$ is a quotient of the full automorphism group $\Aut(\overline{a_{[n-1]}} / B_n)$ (namely, the quotient of the subgroup of all automorphisms fixing $\widetilde{a}_{[n-1]}$ pointwise).

Now we can state the main result of this section:

\begin{theorem}
\label{hurewicz}
If $T$ is stable, $p$ is stationary, and $(a,b) \models p^{(2)}$, then $H_2(p) \cong \Gamma_2(p)$.
\end{theorem}

An immediate consequence of this theorem plus Corollary~\ref{trivial_homology} above is:

\begin{corollary}
If $p$ is a strong type in a stable theory, then $p$ has $3$-uniqueness (or equivalently, $p$ has $4$-amalgamation) if and only if $H_2(p) \cong 0$.
\end{corollary}

\begin{question}
\label{hurewicz2}
If $T$ is stable with $(n+1)$-complete amalgamation, then is $H_n(p)$ isomorphic to $\Gamma_n(p)$?
\end{question}

\subsection{Preliminaries on definable groupoids}

Here we review some material from \cite{GK} and \cite{gkk} on definable groupoids that we need for the proof of Theorem~\ref{hurewicz}.  We also make a minor correction to a lemma from \cite{GK} and set some notation that will be used later.  Recall that we assume $T$ is stable.

We know from \cite{GK} that in a stable theory, failures of $3$-uniqueness (or equivalently, of $4$-amalgamation) are linked with type-definable connected groupoids which are not retractable.  (See that paper for definitions of these terms.)  It turns out that the groupoid $\cG$ associated to such a failure of $3$-uniqueness can even be assumed to have abelian ``vertex groups'' $\mor_{\cG}(a,a)$ (this is proved in Section 2 of \cite{gkk}).

Given an $\acl(\emptyset)$-definable connected groupoid $\cG'$ such that the groups $\cG'_a := \mor_{\cG'}(a,a)$ are all finite and \emph{abelian}, we can define canonical isomorphisms between any two groups $\cG'_a$ and $\cG'_b$ via conjugation by some (any) $h \in \mor_{\cG'}(a,b)$.   Therefore we can quotient $\bigcup_{a \in \ob(\cG')} \cG'_a$ by this system of commuting automorphisms to get a {\em binding} group $G'$, and note that $G'$ can be thought of as a subset of $\acl^{eq}(\emptyset)$.  In fact, even if the mentioned  groupoid $\cG$ is only type-definable
(more precisely,  relatively definable due to the explanation after Claim~\ref{compost}), we can still associate
the binding group $G$ with a subset of $\acl(\emptyset)$: first find a definable connected extension $\cG'$ of $\cG$ in which $\cG$ is a full faithful subcategory, then apply this argument to $\cG'$. If $h \in \cG_a$, let $\left[h\right]_{G'}$ be the corresponding element of $G$ (so identify $G$ and  $G'$).

Next we recall from \cite{gkk} the definition of a ``full symmetric witness to the failure of $3$-uniqueness.''  For the present paper, we modify the definition slightly so that a full symmetric witness is a tuple $W$ containing a formula $\theta$ witnessing the key property.  (Later we will need to keep track of this formula).

\begin{definition}\label{full_symm_witness}
A \emph{full symmetric witness to non-$3$-uniqueness} (over the set $A$) is a tuple $(a_0, a_1, a_2, f_{01}, f_{12}, f_{02}, \theta(x,y,z))$ such that $a_0, a_1, a_2$ and $f_{01},f_{12}, f_{02}$ are finite tuples, $\theta(x,y,z)$ is a formula over $A$, and:
\begin{enumerate}
\item
$f_{ij} \in \overline{a}_{ij}$;
\item
$f_{01} \notin \dcl(\overline{a}_0 \overline{a}_1)$;
\item
$a_0a_1f_{01} \equiv_A a_1a_2f_{12} \equiv_A a_0a_2f_{02}$;
\item
$f_{01}$ is the unique realization of
$\theta(x, f_{12}, f_{02})$, the element $f_{12}$ is the unique realization of $\theta(f_{01}, y, f_{02})$, and
$f_{02}$ is the unique realization of $\theta(f_{01}, f_{12}, z)$; and
\item
$\tp(f_{01}/\ov{a}_0\ov{a}_1)$ is isolated by $\tp(f_{01}/a_0a_1)$.
\end{enumerate}
\end{definition}

The following (proved in \cite{gkk}) is the key technical point saying that we have ``enough'' symmetric witnesses:

\begin{proposition}\label{full symm}
If $T$ does not have $3$-uniqueness, then there is a set $A$ and a full symmetric witness to non-$3$-uniqueness over $A$.

In fact, if $(a_0, a_1, a_2)$ is the beginning of a Morley sequence and $f$ is any element of $\ov{a_{01}} \cap \dcl(\ov{a_{02}}, \ov{a_{12}})$ which is not in $\dcl(\ov{a_0}, \ov{a_1})$, then there is some full symmetric witness $(a'_0, a'_1, a'_2, f', g, h, \theta)$ such that $f \in \dcl(f')$ and $a_i \in \dcl(a'_i) \subseteq \ov{a_i}$ for $i = 0, 1, 2$.
\end{proposition}

The next lemma states a crucial point in the construction of type-definable groupoids from witnesses to the failure of $3$-uniqueness.  This was not isolated as a lemma in \cite{GK}, though the idea was there.

\begin{lemma}
\label{full_symm2}
If $(a_0, a_1, a_2, f_{01}, f_{12}, f_{02}, \theta(x,y,z))$ is a full symmetric witness, and if $f \equiv_{a_0 a_1} f_{01}$ and $g \equiv_{a_1 a_2} f_{12}$, then $$(f, g, \ov{a_0}, \ov{a_1}, \ov{a_2}) \equiv (f_{01}, f_{12}, \ov{a_0}, \ov{a_1}, \ov{a_2}).$$
\end{lemma}

\begin{proof}
By clause (5) in the definition of a full symmetric witness, $(f,  \ov{a_0}, \ov{a_1}) \equiv (f_{01},  \ov{a_0}, \ov{a_1})$ and $(g,  \ov{a_1}, \ov{a_2}) \equiv (f_{12}, \ov{a_1}, \ov{a_2})$.  It follows (by the stationarity of types over $\ov{a_1}$) that $$(f, g, \ov{a_0}, \ov{a_1}, \ov{a_2}) \equiv (f_{01}, g, \ov{a_0}, \ov{a_1}, \ov{a_2})$$ and $$(f_{01}, g, \ov{a_0}, \ov{a_1}, \ov{a_2}) \equiv (f_{01}, f_{12}, \ov{a_0}, \ov{a_1}, \ov{a_2}),$$ and the lemma follows.
\end{proof}

Given any full symmetric witness to the failure of $3$-uniqueness, we can construct from it a connected, type-definable groupoid:

\begin{proposition}
\label{groupoid_construction}
Let $W = (a_0, a_1, a_2, f, g, h, \theta(x,y,z))$ be a full symmetric witness (over $\emptyset$).  Then from $W$ we can construct a connected groupoid $\cG_W$ which is type definable over $\acl(\emptyset)$ and has the following properties:

\begin{enumerate}
\item The objects of $\cG_W$ are the realizations of the type $p = \stp(a_1)$.
\item Let $$SW_{a_0, a_1} := \{f' : f' \equiv_{a_0, a_1} f \},$$  
There is a bijection $f \mapsto [f]^{a_0, a_1}_{\cG_W}$ from $SW_{a_0, a_1}$ onto $\mor_{\cG_W}(a_0, a_1)$ which is definable over $(a_0, a_1)$.
\item If $f_0,f_1 \in \mor_{\cG}(a_0, a_1)$, then $f_0 \equiv_{a_0, a_1} f_1$.
\item The ``vertex groups'' $\mor_{\cG_W}(a,a)$ are finite and abelian.
\end{enumerate}

\end{proposition}

\begin{proof}
We build $\cG_W$ using a slight modification of the construction described in subsection~2.2 of \cite{GK}.  The problem with the construction in that paper is that Remark~2.8 there is incorrect as stated: in general, just because $(a,b,f) \equiv (a_0, a_1, f_{01}) \equiv (b, c, g)$, it does not follow that $(a,b,c,f,g) \equiv (a_0, a_1, a_2, f_{01}, f_{12})$ (if the tuples $a_i$ are not algebraically closed, $f_{01}$ may contain elements of $\acl(a_0) \setminus \dcl(a_0)$, and this could cause $\tp(a, f,g )$ to differ from $\tp(a_0, f_{01}, f_{12})$).  However, Lemma~\ref{full_symm2} and the fact that we are using a \emph{full} symmetric witness eliminates this problem.  In particular, if $(a,b,f) \equiv (a_0, a_1, f_{01}) \equiv (b, c, g)$, then there is a unique element ``$g \circ f$'' such that $\models \theta(f, g, g \circ f)$ and $(a, c, g \circ f) \equiv (a_0, a_2, f_{02})$.

From here, everything else in the construction of the type-definable groupoid $\cG = \cG_W$ in \cite{GK} works.  Property (1) of the proposition follows directly from the construction, and property (2) is just like Lemma~2.14 of \cite{GK}.  Because of the definable bijection in (2), any two morphisms in $\mor_{\cG}(a_0, a_1)$ have the same type, yielding (3).  Finally, property (4) is Corollary~2.7 of \cite{gkk}.

\end{proof}


Next, here is a more detailed version Proposition~2.15 from \cite{gkk}, which we will use later.

\begin{proposition}
\label{corrected2}
Suppose that $(a_0, a_1, a_2, f_{01}, f_{12}, f_{02}, \theta)$ is a full symmetric witness, and $\cG$ is the associated type-definable groupoid as in Proposition~\ref{groupoid_construction}.  If $SW$ is the set $\{f' : \tp(f'/a_0, a_1) = \tp(f_{01} / a_0, a_1)\}$, then there is a group isomorphism $$\psi^0_{\cG}: \mor_{G}(a_1,a_1) \rightarrow \Aut(SW / a_0, a_1)$$ defined by the rule: if $g \in \mor_{\cG}(a_1, a_1)$, then $\psi^0_{\cG}(g)$ is the unique element $\sigma \in \Aut(SW / a_0, a_1)$ which induces the same left action on $\mor_{\cG}(a_0, a_1)$ as left composition by $g$.

Furthermore, the inclusion map $\Aut(SW/ \ov{a_0}, \ov{a_1}) \rightarrow \Aut(SW / a_0, a_1)$ is surjective, so we actually have an isomorphism

$$\psi_{\cG}: \mor_{G}(a_1,a_1) \rightarrow \Aut(SW / \ov{a_0}, \ov{a_1}).$$

\end{proposition}

\begin{proof}
The ``Furthermore ...'' clause was not in Proposition~2.15 of \cite{gkk}, but it follows from the fact that the witness is fully symmetric: if $f'$ is any element of $SW$, then clause (5) of the definition of a symmetric witness implies that $\tp(f' / \ov{a_0}, \ov{a_1}) = \tp(f_{01} / \ov{a_0}, \ov{a_1})$, and so there is an element $\sigma \in \Aut(SW / \ov{a_0}, \ov{a_1})$ such that $\sigma(f_{01}) = f'$.  This means that there are at least $| \mor_{\cG}(a_1, a_0) |$ different elements in $\Aut(SW / \ov{a_0}, \ov{a_1})$; but, by the first part of the proposition, there are only $| \mor_{\cG}(a_1, a_0) |$ elements in $\Aut(SW / a_0, a_1)$.  Since this is a finite set, the injective inclusion map $\Aut(SW/ \ov{a_0}, \ov{a_1}) \rightarrow \Aut(SW / a_0, a_1)$ is surjective.

\end{proof}

\subsection{Proof of Theorem~\ref{hurewicz}}

We assume throughout the proof that $p \in S(\emptyset)$ and $\acl(\emptyset) = \dcl(\emptyset)$ (since we can add constants for the parameters of $p$ if necessary).  It follows directly from the definitions that if $p = \tp(a)$ and $p' = \tp(a')$ where $a$ and $a'$ are interalgebraic, then $H_n(p) = H_n(p')$.  Therefore, by Proposition~\ref{full symm} above, we may assume that there are some $(a_0, a_1, a_2)$ realizing $p^{(3)}$ and a full symmetric witness $(a_0, a_1, a_2, f_{01}, f_{12}, f_{02}, \theta(x,y,z))$ to this failure.  We pick one such witness which we fix throughout the proof. Note that we assume the $f_{ij}$'s to be finite tuples, and also that there may be more than one such witness (which is the interesting case).  We assume that there is \emph{at least} one such witness, since otherwise $H_2(p)$ and $\Gamma_2(p)$ are both trivial.

As already observed in \cite{gkk}, the symmetric witnesses in the type $p$ form a directed system.  To make this more precise, pick some $(a_0, a_1, a_2)$ realizing $p^{(3)}$ (which we fix for the remainder of the subsection).  Now we build a directed system of full symmetric witnesses as follows:

\begin{claim}
\label{witness system}
There is a directed partially ordered set $\langle I, \leq \rangle$ and and $I$-indexed collection of symmetric witnesses $\langle W_i : i \in I \rangle$ such that for any $i$ and $j$ in $I$:

\begin{enumerate}
\item $W_i = (a^i_0, a^i_1, a^i_2, f_{01}^i, f_{12}^i, f_{02}^i, \theta^*_i(x_i, y_i, z_i))$ is a full symmetric witness to failure of $3$-uniqueness;
\item $a^i_0, a^i_1 \in \dcl(f_{01}^i)$;
\item if $i \leq j$, then $f_{01}^i \in \dcl(f_{01}^j)$, $a^i_0 \in \dcl(a^j_0) \subseteq \ov{a_0}$, and $a^i_0 a^j_0 \equiv a^i_1 a^j_1 \equiv a^i_2 a^j_2$;
\end{enumerate}

and satisfying the maximality conditions $$\widetilde{a_{\{0,1\}}} = \dcl\left(\bigcup_{i \in I} f^i_{01} \right)$$ and $$\ov{a_0} = \dcl \left(\bigcup_{i \in I} a^i_0 \right) .$$

\end{claim}

\begin{proof}
We will build the partial ordering $\langle I, \leq \rangle$ as the union of a countable chain of partial orderings $I_0 \subseteq I_1 \subseteq \ldots$ such that for any $i, j \in I_n$ there is a $k \in I_{n+1}$ such that $i \leq k$ and $j \leq k$. Then the partial ordering $I = \bigcup_{n \in \omega} I_n$ will be directed.

First, let

$$W^0_i = \langle a^i_0, a^i_1, a^i_2, f_{01}^i, f_{12}^i, f_{02}^i, \theta^*_i(x_i, y_i, z_i) : i \in I_0 \rangle$$

be any collection of full symmetric witnesses large enough to satisfy the two maximality conditions in the statement of the Claim, where $I_0$ is a trivial partial ordering in which no two distinct elements are comparable.  For the induction step, suppose that we have the partial ordering $I_n$ (for some $n \in \omega$) and full symmetric witnesses $(a^i_0, \ldots, \theta^*_i(x_i, y_i, z_i))$ for each $i \in I_n$.  First, we can build a partial ordering $I_{n+1}$ by adding one new point immediately above every pair of points in $I_n$ and such that any two new points in $I_{n+1} \setminus I_n$ are incomparable. Then by Proposition~\ref{full symm}, there are corresponding full symmetric witnesses $(a^i_0, \ldots, \theta^*_i)$ for each $i \in I_{n+1} \setminus I_n$ such that if $j$ and $k$ are less than or equal to $i$, then $f^j_{01}, f^k_{01} \in \dcl(f^i_{01})$ and $a^j_0, a^k_0 \in \dcl(a^i_0)$. Similarly, we can ensure condition (2) (that $a^i_0, a^i_1 \in \dcl(f^i_{01})$) for the new symmetric witnesses.

\end{proof}

Let $p_i = \stp(a^i_0)$ and $\cG^*_i$ be the type-definable groupoid constructed from the full symmetric witness $W_i$ as in Proposition~\ref{groupoid_construction} above.  So $\ob(\cG_i) = p_i(\C)$ and the groups $\mor_{\cG^*_i}(a^i_0,a^i_0)$ are finite and abelian, and we have the corresponding finite abelian groups $G^*_i$.  As explained above, can (and will) assume that the groups $G^*_i$ are subsets of $\acl(\emptyset)$.

For any $i \in I$, let $SW_i$ be the set of all realizations of $\tp(f_{01}^i / a^i_0, a^i_1)$ (which is a finite set).  If $(a,b) \models p_i^{(2)}$, let $SW(a,b)$ be the image of $SW_i$ under an automorphism of $\C$ that maps $(a^i_0, a^i_1)$ to $(a, b)$. Recall from Proposition~\ref{groupoid_construction} that we have a definable map $f \mapsto \left[ f \right]^{a,b}_{\cG_i}$ from $SW(a,b)$ onto $\mor_{\cG_i}(a,b)$, from which we can define an inverse map $g \mapsto \langle g \rangle^{a,b}_{\cG_i}$ from $\mor_{\cG_i}(a,b)$ to $SW(a,b)$.  For convenience, we will write these maps as ``$\left[ \cdot \right]^{a,b}_i$'' and ``$\langle \cdot \rangle^{a,b}_i$,'' or even just ``$\left[ \cdot \right]_i$'' and ``$\langle \cdot \rangle_i$'' when $(a,b)$ is clear from context.

\begin{lemma}
\label{commuting_pi}
There are systems of relatively $\emptyset$-definable functions $\langle \pi_{j,i} : i \leq j, j \in I \rangle$ and $\langle \tau_{j,i} : i \leq j, j \in I \rangle $ (that is, they are the intersection of an $\emptyset$-definable relation with the product of their domain and range) such that whenever $i \leq j$,
\begin{enumerate}
\item $\tau_{j,i} : p_j(\C) \rightarrow p_i(\C)$,
\item $\pi_{j,i} : \bigcup_{(a,b) \models p_j^{(2)}} SW(a,b) \rightarrow SW(\tau_{j,i}(a), \tau_{j,i}(b))$,
\item $\tau_{j,i}(a^j_0) = a^i_0$,
\item $\tau_{j,i}(a^j_1) = a^i_1$,
\item $\pi_{j,i}$ maps $SW_j$ onto $SW_i$, and
\item $\pi_{j,i}(f^j_{01}) = f^i_{01}$,
\end{enumerate}

and whenever $i \leq j \leq k$,

\begin{enumerate}[resume]
\item $\tau_{j,i} \circ \tau_{k,j} = \tau_{k,i}$ and
\item $\pi_{j,i} \circ \pi_{k,j} = \pi_{k,i}$.

\end{enumerate}
\end{lemma}

\begin{proof}
First, the maps $\tau_{j,i}$ can be constructed satisfying (1), (3), and (4) using the facts that $a^i_0 \in \dcl(a^j_0),$ $a^i_0 \in \dcl(a^j_1)$, and $a^i_0 a^j_0 \equiv a^i_1 a^j_1$ (from clause (3) of Claim~\ref{witness system}). Now if $i \leq j \leq k$, since $\tau_{k,i}(x) = \tau_{j,i} \circ \tau_{k,j} (x)$ is true for $x = a_k$, this holds for \emph{every} $x$ in the domain of $\tau_{k,i}$ (because the domain is a complete type), and so (7) holds.

If $i \leq j$, then since $f^i_{01} \in \dcl(f^j_{01})$, we can pick a relatively definable map $\pi_{j,i}$ such that $\pi_{j,i}(f^j_{01}) = f^i_{01}$. As before, if $i \leq j \leq k$, since $\pi_{k,i}(x) = \pi_{j,i} \circ \pi_{k,j}(x)$ holds for $x = f^k_{01}$, it holds for any $x$ in any of the sets $SW(a,b)$ for $(a,b) \models p^{(2)}$, so (8) holds.



\end{proof}

Ideally, we would like the functions $\pi_{j,i}$ and $\tau_{j,i}$ of Lemma~\ref{commuting_pi} to induce a commuting system of functors $F_{j,i} : \cG^*_j \rightarrow \cG^*_i$, which we could use to construct and inverse limit $\cG$ of $\langle \cG^*_i : i \in I \rangle$.  This is essentially what we do, and we will then show that the group $\mor_{\cG}(\ov{a_0}, \ov{a_0})$ is isomorphic to both $H_2(p)$ and $\Gamma_2(p)$.  However, first we need to modify the formulas $\theta^*_i$ slightly for this to be true.

The key to making all of this work is the following technical lemma.

\begin{lemma}
\label{theta_tweak}
There is a family of formulas $\langle \theta_i(x_i, y_i, z_i) : i \in I\rangle$ such that
\begin{enumerate}
\item $W_i$ is still a full symmetric witness with $\theta^*_i(x_i, y_i, z_i)$ replaced by $\theta_i(x_i, y_i, z_i)$ and $f_{02}^i$ replaced by another element of $SW(a^i_0,a^i_2),$ and
\item whenever $i \leq j$, $f \in SW(a^j_0,a^j_1)$, $g \in SW(a^j_1,a^j_2)$, and $h \in SW(a^j_0,a^j_2)$, then

$$\models \theta_j(f, g, h) \rightarrow \theta_i(\pi_{j,i}(f), \pi_{j,i}(g), \pi_{j,i}(h)).$$
\end{enumerate}
\end{lemma}

\begin{proof}
Recall from above that $(a_0, a_1, a_2)$ realizes $p^{(3)}$.  We use Zorn's Lemma to find a maximal subset $J \subseteq I$ and formulas $\theta_j(x_j, y_j, z_j)$ for each $j \in J$ satisfying the following properties:

\begin{enumerate}
\setcounter{enumi}{2}
\item For every $j \in J$, there are elements $f_j$, $g_j,$ and $h_j$ such that $(a^j_0, a^j_1, a^j_2, f_j, g_j, h_j, \theta_j(x_j, y_j, z_j))$ is a full symmetric witness; and

\item If $j_1, \ldots, j_n \in J$ and $(a^{j_s}_0, a^{j_s}_1, a^{j_s}_2, f_{j_s}, g_{j_s}, h_{j_s}, \theta_{j_s})$ is a full symmetric witness for $s = 1, \ldots, n$, and if $f_{j_1} \ldots f_{j_n} \equiv g_{j_1} \ldots g_{j_n}$, then $f_{j_1} \ldots f_{j_n} \equiv h_{j_1} \ldots h_{j_n}$.

\end{enumerate}

\begin{claim}
$J = I$.
\end{claim}

\begin{proof}
Suppose towards a contradiction that there is some $k \in I \setminus J$.  Let $F_J = \langle f^\alpha \rangle$ be a (possibly infinite) tuple listing every element of $\bigcup_{j \in J} SW(a^j_0, a^j_1)$, and let $a^J_i$ (for $i \in \{0,1,2\}$) be a tuple listing $\{a^j_i : j \in J\}$, ordered the same way as $F_J$.  Pick $f_k \in SW(a^k_0, a^k_1)$, and then pick $G_J = \langle g^\alpha \rangle$ and $g_k$ such that $F_J f_k a^J_0 a^J_1 \equiv G_J g_k a^J_1 a^J_2$.  Note that $g^\alpha \in SW(a^j_1, a^j_2)$.  Next pick a tuple $H_J = \langle h^\alpha \rangle$ such that if $f^\alpha \in SW(a^j_0, a^j_1)$, then $\models \theta_j(f^\alpha, g^\alpha, h^\alpha)$.  The element $h_j$ is well-defined because if it happens that $f^\alpha$ is also in $SW(a^{j'}_0, a^{j'}_1)$ for some $j' \neq j$, and if we let $h'$ be the unique element such that $\models \theta_{j'}(f^\alpha, g^\alpha, h')$, then by property~(4), the fact that $f^\alpha f^\alpha \equiv g^\alpha g^\alpha$ implies that $f^\alpha f^\alpha \equiv h^\alpha h'$, and so $h' = h^\alpha$.

By the assumption (4) on the set $J$, $F_J \equiv H_J$.  Finally, pick an element $h_k$ such that $F_J f_k \equiv H_J h_k$.  By Corollary~2.14 of \cite{gkk}, there is a formula $\theta_k$ such that $(a^k_0, a^k_1, a^k_2, f_k, g_k, h_k, \theta_k)$ is a full symmetric witness.

We claim that $J \cup \{k\}$ with $\theta_k$ satisfies condition (4) above, contradicting the maximality condition on the set $J$.  Indeed, suppose that $j_1, \ldots, j_n \in J$, and the tuples $$(a^{j_s}_0, a^{j_s}_1, a^{j_s}_2, f_{j_s}, g_{j_s}, h_{j_s}, \theta_{j_s})$$ (for $s = 1, \ldots, n$) and $$(a^k_0, a^k_1, a^k_2, f'_k, g'_k, h'_k, \theta_k)$$ are full symmetric witnesses, and that $f_{j_1} \ldots f_{j_n} f'_k \equiv g_{j_1} \ldots g_{j_n} g'_k.$  By the stationarity of $\tp(f'_k / \ov{a^1})$ and $\tp(g'_k / \ov{a^1})$, there is a $\sigma \in \Aut(\C / \ov{a^0}, \ov{a^1}, \ov{a^2})$ such that $\sigma(f'_k) = f_k$ and $\sigma(g'_k) = g_k$ for the $f_k$ and $g_k$ from the previous paragraph.  By the same argument and using the fact that $F_J \equiv G_J$, we can also assume that if $\sigma ((f_{j_1}, \ldots, f_{j_n})) = (f^{\alpha_1}, \ldots, f^{\alpha_n})$, then $\sigma((g_{j_1}, \ldots, g_{j_n}) = (g^{\alpha_1}, \ldots, g^{\alpha_n})$ (that is, the two tuples $(f_{j_1}, \ldots, f_{j_n})$ and $(g_{j_1}, \ldots, g_{j_n})$ map to corresponding subtuples of $F_J$ and $G_J$).  It follows that $\sigma(h'_k) = h_k$ and $\sigma(h_{j_s}) = h^{\alpha_s}$ for each $s$ between $1$ and $n$.  By our construction, $f^{\alpha_1} \ldots f^{\alpha_n} f_k \equiv h^{\alpha_1} \ldots h^{\alpha_n} h_k$, and so by taking preimages under $\sigma$, we get that $f_{j_1} \ldots f_{j_n} f'_k \equiv h_{j_1} \ldots h_{j_n} h'_k$.
\end{proof}

Finally, we check that condition (2) of the lemma holds for our new formulas $\theta_i$.  Suppose that $i \leq j$, $f \in SW(a^j_0, a^j_1)$, $g \in SW(a^j_1, a^j_2)$, $h \in SW(a^j_0, a^j_2)$, and $\models \theta_j(f, g, h)$.  Let $f_0 = \pi_{j,i}(f)$, and pick $g_0$ such that $f f_0 \equiv g g_0$. Then $g_0 = \pi_{j,i}(g)$.  Finally, let $h_0$ be the unique element such that $\models \theta_i(f_0, g_0, h_0)$.  By condition~(4) above, $h h_0 \equiv f f_0$, and so $h_0 = \pi_{j,i}(h)$.  Thus $\models \theta_i(\pi_{j,i}(f), \pi_{j,i}(g), \pi_{j,i}(h))$ as desired.

\end{proof}

For each $i \in I$, let $\cG_i$ be the type-definable groupoid obtained from the symmetric witness $W_i$ with the modified formula $\theta_i$ from Lemma~\ref{theta_tweak}.  Once again, the groups $\mor_{\cG_i}(a,a)$ are finite and abelian for any $a \in \ob(\cG_i)$, so we have the corresponding finite abelian groups $G_i$ which we consider as subsets of $\acl(\emptyset)$.

\begin{lemma}
\label{coherence}
If $i \leq j \in I$, $(a,b,c) \models p_j^{(3)}$, $f \in \mor_{\cG_j}(a,b)$, and $g \in \mor_{\cG_j}(b,c)$, then

$$\left[ \pi_{j,i}(\langle g \circ f \rangle_j ) \right]_i = \left[ \pi_{j,i}(\langle g \rangle_j) \right]_i \circ \left[ \pi_{j,i}(\langle f \rangle_j ) \right]_i $$

(where $\circ$ denotes composition in the groupoids $\cG_j$ and $\cG_i$).
\end{lemma}

\begin{proof}
By Proposition~2.12 of \cite{gkk}, $\theta_j$ defines groupoid composition between generic triples of objects in $\cG_j$, so $$\models \theta_j (\langle f \rangle_j, \langle g \rangle_j, \langle g \circ f \rangle_j).$$  So by Lemma~\ref{theta_tweak}, $$\models \theta_i ( \pi_{j,i}(\langle f \rangle_j), \pi_{j,i} (\langle g \rangle_j), \pi_{j,i}(\langle g \circ f \rangle_j) ).$$  By Proposition~2.12 again, the Lemma follows.
\end{proof}

If $i \leq j \in I$ and $(a,b) \models p_j^{(2)}$, then because $SW(\tau_{j,i}(a), \tau_{j,i}(b)) \subseteq \dcl(SW(a,b))$, we have a canonical surjective group map $$\rho^{a,b}_{j,i} : \Aut(SW(a,b) / \ov{a}, \ov{b}) \rightarrow \Aut(SW(\tau_{j,i}(a), \tau_{j,i}(b)) / \ov{a}, \ov{b}),$$ and these maps satisfy the coherence condition that $\rho^{a,b}_{k,i} = \rho^{a,b}_{j,i} \circ \rho^{a,b}_{k,j}$ whenever $i \leq j \leq k$.  We will write ``$\rho_{j,i}$'' for $\rho^{a,b}_{j,i}$ if $(a,b)$ is clear from context.

For every $i \in I$, we also have a group isomorphism $\psi_i : \mor_{\cG_i}(a^i_1, a^i_1) \rightarrow \Aut(SW_i / \ov{a_0}, \ov{a_1})$ as in Proposition~\ref{corrected2} above.

The following is similar to Claim~2.17 of \cite{gkk}, except that here we have expanded this to a system of \emph{groupoid} maps.

\begin{lemma}
\label{chi}
For every $i \leq j \in I$, we define a map $\chi_{j,i} : \cG_j \rightarrow \cG_i$ by the rules:

\begin{enumerate}
\item if $a \in \ob(\cG_j)$, then $\chi_{j,i}(a) = \tau_{j,i}(a)$; and
\item if $f \in \mor_{\cG_j}(a,b)$, $c \models p_j | (a,b)$, and $f = g \circ h$ for some $g \in \mor_{\cG_j}(c,b)$ and $h \in \mor_{\cG_j}(a,c)$, then $$\chi_{j,i}(f) = \left[ \pi_{j,i}(\langle h \rangle_j) \right]_i \circ \left[ \pi_{j,i}(\langle g \rangle_j) \right]_i.$$
\end{enumerate}

Then the maps $\chi_{j,i}$ satisfy:

\begin{enumerate}
\setcounter{enumi}{2}
\item $\chi_{j,i}$ is a well-defined functor;
\item $\chi_{j,i}$ is \emph{full}: every morphism in $\mor(\cG_i)$ is in the image of $\chi_{j,i}$;
\item $\chi_{j,i}$ is type-definable over $\acl(\emptyset)$;
\item if $(a,b) \models p_j^{(2)}$ and $f \in \mor_{\cG_j}(a,b)$, then the formula for $\chi_{j,i}(f)$ simplifies to $$\chi_{j,i}(f) = \left[ \pi_{j,i}(\langle f \rangle_j)\right]_i ;$$
\item $\chi_{k,i} = \chi_{j,i} \circ \chi_{k,j}$ whenever $i \leq j \leq k$; and
\item for any $i \leq j$, the following diagram commutes: \bigskip

$\begin{CD}
\mor_{\cG_j}(a^j_1,a^j_1) @>\chi_{j,i}>> \mor_{\cG_i}(a^i_1,a^i_1)\\
@VV\psi_jV @VV\psi_iV\\
\Aut(SW_j / \ov{a_0}, \ov{a_1}) @>\rho_{j,i}>> \Aut(SW_i / \ov{a_0}, \ov{a_1})
\end{CD}$

\end{enumerate}
\end{lemma}

\begin{proof}
Suppose that $f \in \mor_{\cG_j}(a,b)$.  To check that $\chi_{j,i}(f)$ is well-defined (and does not depend on the choices of $c, g,$ and $h$), first note that given $c \models p_j | (a,b)$ and morphisms $g, h$ as in (2), the morphism $h$ is uniquely determined from $f$ and $g$, and for any other $g' \in \mor_{\cG_j}(c,b)$, $\tp(f,g/a,b,c) = \tp(f,g'/a,b,c)$ (by Lemma~\ref{full_symm2}).  So the choices of $f$ and $g$ do not matter once we have picked $c$, and the choice of $c$ does not matter by the stationarity of $p_j$.

To show that $\chi_{j,i}$ is a functor, suppose that $a, b,$ and $c$ realize $p_j$, $f \in \mor_{\cG_j}(a,b)$, and $g \in \mor_{\cG_j}(b,c)$.  To compute the images of $f$ and $g$, we pick $(d,e) \models p_{j}^{(2)} | (a,b,c)$ and $f_0 \in \mor_{\cG_j}(a,d), f_1 \in \mor_{\cG_j}(d,b)$, $g_0 \in \mor_{\cG_j}(b,e),$ and $g_1 \in \mor_{\cG_j}(e,c)$ such that $f = f_1 \circ f_0$ and $g = g_1 \circ g_0$.  Then by the definition given in (2) of the Lemma,

$$\chi_{j,i}(f) = \left[ \pi_{j,i}(\langle g_1 \circ g_0 \circ f_1 \rangle_j) \right]_i \circ \left[ \pi_{j,i}(\langle f_0 \rangle_j) \right]_i.$$

By Lemma~\ref{coherence} twice, this equals

$$\left[ \pi_{j,i}(\langle g_1 \rangle_j ) \right]_i \circ \left[ \pi_{j,i}(\langle g_0 \rangle_j)\right]_i \circ \left[ \pi_{j,i}(\langle f_1 \rangle_j)\right]_i \circ \left[ \pi_{j,i}(\langle f_0 \rangle_j)\right]_i.$$

But the composition of the first two terms above equals $\chi_{j,i}(g)$ and the composition of the third and fourth terms equals $\chi_{j,i}(f)$, so $\chi_{j,i}(g \circ f) = \chi_{j,i}(g) \circ \chi_{j,i}(f)$.

Suppose that $a, b \in \ob(\cG_i)$ and $f \in \mor_{\cG_i}(a,b)$.  Pick some $c \models p_i | (a,b)$, and pick $g \in \mor_{\cG_i}(c,b)$ and $h \in \mor_{\cG_i}(a,c)$ such that $f = g \circ h$.  Since $$(\langle g \rangle_i, c, b) \equiv (f^i_{01}, a^i_0, a^i_1) \equiv (\langle h \rangle_i, a, c),$$ we can find elements $g'$ and $h'$ such that $\pi_{j,i}(g') = \langle g \rangle_i$ and $\pi_{j,i}(h') = \langle h \rangle_i$.  Let $f^* = \left[ g' \right]_j \circ \left[ h' \right]_j$.  Unwinding the definitions, we see that $$\chi_{j,i}(f^*) = \left[ \pi_{j,i}(g')\right]_i \circ \left[ \pi_{j,i}(h')\right]_i = \left[\langle g \rangle_i \right]_i \circ \left[ \langle h \rangle_i \right]_i = g \circ h = f.$$  This establishes that the functor $\chi_{j,i}$ is full.

The fact that $\chi_{j,i}$ is type-definable is simply by the definability of types in stable theories, and in fact the action of $\chi_{j,i}$ on the objects and morphisms of $\cG_j$ is given by the intersection of a \emph{definable} set with the type-definable sets $\ob(\cG_j)$ and $\mor(\cG_j)$.

The formula (6) follows directly from the definition of $\chi_{j,i}(f)$ in (2) and Lemma~\ref{coherence}.

Next we prove (7).  Suppose that $i \leq j \leq k$.  If $a \in \ob(\cG_k)$, then $\chi_{j,i} \circ \chi_{k,j}(a) = \tau_{j,i}(\tau_{k,j}(a)) = \tau_{k,i}(a) = \chi_{k,i}(a)$.  If $a, b, c \in \ob(\cG_k)$ and $f = g \circ h$ are in (2) of the Lemma (with $j$ replaced by $k$), then by the definition of the $\chi$ maps,

$$\chi_{j,i} \circ \chi_{k,j} (f) = \chi_{j,i} \left(\left[ \pi_{k,j}(\langle h \rangle_k)\right]_j \circ \left[ \pi_{k,j}(\langle g \rangle_k)\right]_j \right)$$ $$= \left[\pi_{j,i}\left(\langle \left[ \pi_{k,j}(\langle h \rangle_k)\right]_j \rangle_j\right) \right]_i \circ \left[ \pi_{j,i}\left(\langle\left[ \pi_{k,j}(\langle g \rangle_k)\right]_j \rangle_j\right)\right]_i$$ $$=\left[ \pi_{j,i} \left(\pi_{k,j}(\langle h \rangle_k) \right) \right]_i \circ  \left[ \pi_{j,i} \left(\pi_{k,j}(\langle g \rangle_k) \right) \right]_i $$ $$= \left[ \pi_{k,i}(\langle h \rangle_k) \right]_i \circ \left[ \pi_{k,i}(\langle g \rangle_k) \right]_i = \chi_{k,i}(f).$$

Finally, we check (8).  Suppose $i \leq j$ and $f \in \mor_{\cG_j}(a^j_1, a^j_1)$.  To show that $\psi_i (\chi_{j,i}(f)) = \rho_{j,i} ( \psi_j(f))$, we pick some arbitrary $k_0 \in \mor_{\cG_i}(a^i_0, a^i_1)$ and show that

\setcounter{equation}{8}

\begin{equation}\label{8} \left[\psi_i (\chi_{j,i}(f))\right](k_0) = \left[\rho_{j,i} ( \psi_j(f))\right](k_0) . \end{equation}

On the one hand, by definition of $\psi_i$,

$$\left[\psi_i (\chi_{j,i}(f))\right](k_0) = \chi_{j,i}(f) \circ k_0.$$

To compute the right-hand side of equation~\ref{8}, pick some $k \in \mor_{\cG_j}(a^j_0, a^j_1)$ such that $\left[ \pi_{j,i}(\langle k \rangle_j) \right]_i = k_0$.  Then $$\left[\psi_j(f)\right] (k) = f \circ k,$$ and $\rho_{j,i}(\psi_j(f))$ must move $k_0 = \left[ \pi_{j,i}(\langle k \rangle_j) \right]_i$ to the element which is defined from $\left[\psi_j(f)\right](k)$ in the same way that $k_0$ is defined from $k$, so $$\left[\rho_{j,i} ( \psi_j(f))\right](k_0) = \left[\pi_{j,i}(\langle f \circ k \rangle_j) \right]_i.$$  By (6) and the functoriality of $\chi_{j,i}$, $$\left[\rho_{j,i} ( \psi_j(f))\right](k_0) = \chi_{j,i}(f \circ k) = \chi_{j,i}(f) \circ \chi_{j,i}(k) = \chi_{j,i}(f) \circ \left[ \pi_{j,i}(\langle k \rangle_j)\right]_i$$ $$= \chi_{j,i}(f) \circ k_0.$$  So both sides of equation~\ref{8} equal $\chi_{j,i}(f) \circ k_0$, and we are done.
\end{proof}

Finally, we define maps on the $p$-simplices and homology groups.  Throughout, we will work with the \emph{set} homology group (and set-simplices, et cetera) for convenience.

First, for every $i \in I$, we pick an arbitrary ``selection function'' $\alpha^0_i: S_0 \cC(p) \rightarrow p_i(\C)$ such that $\alpha^0_i(a) \in \dcl(a)$.  (This is a technical point, but the $0$-simplices in $S_0 \cC(p)$ are \emph{algebraic closures} of realizations of $p_i$, and there might be no canonical way to get a realization of $p_i$ from a $0$-simplex.  Thus we need the choice functions $\alpha^0_i$.)

Next, we pick selection functions $\alpha_i : S_1 \cC(p) \rightarrow \mor(\cG_i)$ (for every $i \in I$) as follows.  Suppose that $\dom(f) = \P(\{n_0, n_1\})$ for $n_0 < n_1$, and for $x \in \{n_0, n_1\}$, let ``$f_x$'' stand for $f^{\{x\}}_{\{n_0, n_1\}} (\alpha^0_i(f \upharpoonright \P(\{x\})))$ (remembering that things in the image of $\alpha^0_i$ are realizations of $p_i$, which are also objects in $\ob(\cG_i)$).  Then we pick $\alpha_i(f)$ such that $\alpha_i(f) \in \mor_{\cG_i}(f_{n_0}, f_{n_1}).$  Just as in the proof of Lemma~\ref{commuting_pi}, we can use an inductive argument to ensure that if $i \leq j$ then $\chi_{j,i}(\alpha_j(f)) = \alpha_i(f)$.

Finally, want to extend $\alpha_i$ to a selection function $\epsilon_i : S_2 \cC(p) \rightarrow G_i$.  To ease notation here and in what follows, we set the following notation:

\begin{notation}
\label{faces}
Whenever $f \in S_n \cC(p)$, $\dom(f) = \P(s)$, and $k \in s$, let $$f^i_{k,s} := f^{\{k\}}_s\left(\alpha_i^0(f \upharpoonright \P(\{k\})) \right),$$ and note that $f^i_{k,s}$ is a realization of $p_i$, that is, an object in $\cG_i$.  Similarly, if $\{k,\ell\} \subseteq s$ and $k < \ell$, let $$f^i_{\{k, \ell\},s} :=  f^{\{k, \ell\}}_{s}\left(\alpha_i(f \upharpoonright \P(\{k, \ell \})) \right),$$ which is a morphism in $\mor_{\cG_i}(f^i_{k,s}, f^i_{\ell,s})$.
\end{notation}

\begin{definition}
We define $\epsilon_i : S_2 \cC(p) \rightarrow G_i$ by the rule: if $\dom(f) = \P(s)$, where $s = \{n_0, n_1, n_2\}$ and $n_0 < n_1 < n_2$, then we define $\epsilon_i(f)$ as $$\epsilon_i(f) := \left[ \left( f^i_{\{n_0, n_2\}, s} \right)^{-1} \circ f^i_{\{n_1, n_2\},s} \circ f^i_{\{n_0, n_1\},s} \right]_{G_i}.$$  (Recall that if $f \in \mor_{\cG_i}(a,a)$, then ``$[f]_{G_i}$'' denotes the corresponding element of the group $G_i$.)

These functions $\epsilon_i$ can be extended linearly from $S_2 \cC(p)$ to the collection of all $2$-chains $C_2(p)$, and by abuse of notation we also call this new function $\epsilon_i$.
\end{definition}

The next lemma is a technical point that will be useful for later computations.

\begin{lemma}
\label{epsilon_coherence}
If $i \in I$ and $f \in S_n(p)$ for any $n \geq 3$, $\dom(g) = \P(t)$, and $\{a,b,c\} \subseteq s \subseteq t$ with $a < b < c$, then  $$\epsilon_i(f \upharpoonright \P(\{a,b,c\})) = \left[(f^i_{\{a,c\}, s})^{-1} \circ f^i_{\{b,c\}, s} \circ f^i_{\{a,b\}, s}\right]_{G_i}.$$

\end{lemma}

\begin{proof}
Remember that we identify the elements of $G_i$ with elements of $\acl(\emptyset)$.  Because the transition map $f^{\{a,b,c\}}_s$ fixes $\acl(\emptyset)$ pointwise, $$f^{\{a,b,c\}}_s(\epsilon_i(f \upharpoonright \{a,b,c\})) = \epsilon_i(f \upharpoonright \{a,b,c\}).$$  Therefore the left-hand side of the equation above equals $f^{\{a,b,c\}}_s(\epsilon_i(f \upharpoonright \{a,b,c\}))$, which is the equivalence class (in $G_i$) of

$$ \left[f^{\{a,b,c\}}_s \circ f^{\{a, c\}}_{\{a,b,c\}} (\alpha_i(f \upharpoonright \P(\{a, c\}))) \right]^{-1} \circ$$ $$\left[f^{\{a,b,c\}}_s \circ f^{\{b,c\}}_{\{a,b,c\}} (\alpha_i(f \upharpoonright \P(\{b, c\})))\right]  \circ \left[f^{\{a,b,c\}}_s \circ f^{\{a, b\}}_{\{a,b,c\}} (\alpha_i(f \upharpoonright \P(\{a, b\})))\right]$$

$$= \left[ f^{\{a, c\}}_{s} (\alpha_i(f \upharpoonright \P(\{a, c\}))) \right]^{-1} \circ$$ $$\left[ f^{\{b,c\}}_{s} (\alpha_i(f \upharpoonright \P(\{b, c\})))\right]  \circ \left[ f^{\{a, b\}}_{s} (\alpha_i(f \upharpoonright \P(\{a, b\})))\right],$$

as desired.

\end{proof}

\begin{lemma}
\label{epsilon_boundaries}
If $c \in B^{set}_2(p)$, then for any $i \in I$, $\epsilon_i(c) = 0$.
\end{lemma}

\begin{proof}
By linearity, it suffices to check that $\epsilon_i(\bd(g)) = 0$ for any $g \in S^{set}_3 (p)$.  For simplicity of notation, we assume that $\dom(g) = \P(s)$ where $s = \{0, 1, 2, 3\}$.  To further simplify, we write ``$g_{i,j}$'' for $g^i_{\{i,j\}, s}$.

If $0 \leq j < k < \ell \leq 3$, by Lemma~\ref{epsilon_coherence}, $$\epsilon_i(g \upharpoonright \{j,k,\ell\}) = \left[g_{j,\ell}^{-1} \circ g_{k,\ell} \circ g_{j,k}\right]_{G_i}.$$

Therefore $\epsilon_i(\bd(g))$ equals $$\left[g_{1,3}^{-1} \circ g_{2,3} \circ g_{1,2}\right]_{G_i} - \left[g_{0,3}^{-1} \circ g_{2,3} \circ g_{0,2} \right]_{G_i} $$ $$+ \left[g_{0,3}^{-1} \circ g_{1,3} \circ g_{0,1} \right]_{G_i} - \left[g_{0,2}^{-1} \circ g_{1,2} \circ g_{0,1} \right]_{G_i}$$ $$= \left[g_{0,1}^{-1} \circ g_{1,3}^{-1} \circ g_{2,3} \circ g_{1,2} \circ g_{0,1}\right]_{G_i} - \left[g_{0,3}^{-1} \circ g_{2,3} \circ g_{0,2} \right]_{G_i} $$ $$+ \left[g_{0,3}^{-1} \circ g_{1,3} \circ g_{0,1} \right]_{G_i} - \left[g_{0,2}^{-1} \circ g_{1,2} \circ g_{0,1} \right]_{G_i}$$ $$ = - \left[g_{0,2}^{-1} \circ g_{1,2} \circ g_{0,1} \right]_{G_i} - \left[g_{0,3}^{-1} \circ g_{2,3} \circ g_{0,2} \right]_{G_i} + \left[g_{0,3}^{-1} \circ g_{1,3} \circ g_{0,1} \right]_{G_i} $$ $$+ \left[g_{0,1}^{-1} \circ g_{1,3}^{-1} \circ g_{2,3} \circ g_{1,2} \circ g_{0,1}\right]_{G_i}$$ $$= \left[ \left(g_{0,1}^{-1} \, g_{1,2}^{-1} \, g_{0,2}\right) \circ \left( g_{0,2}^{-1} \, g_{2,3}^{-1} \, g_{0,3}\right) \circ \left( g_{0,3}^{-1} \, g_{1,3} \, g_{0,1} \right) \circ \left( g_{0,1}^{-1} \, g_{1,3}^{-1} \, g_{2,3} \, g_{1,2} \, g_{0,1} \right)\right]_{G_i},$$ but everything in the last expression cancels out.

\end{proof}

By the last lemma, each $\epsilon_i$ induces a well-defined function $\widetilde{\epsilon_i} : H_2(p) \rightarrow G_i$.

Now we relate the $\epsilon_i$ maps to the groupoid maps $\chi_{j,i} : \cG_j \rightarrow \cG_i$.  For $i \in I$, let $\overline{\psi_i} : G_i \rightarrow  \Aut(SW_i / \ov{a_0}, \ov{a_1})$ be the map induced by $\psi_i : \mor_{G_i}(a^1_i,a^1_i) \rightarrow \Aut(SW_i / \ov{a_0}, \ov{a_1})$, and let $\overline{\chi_{j,i}} : G_j \rightarrow G_i$ be the surjective group homomorphism induced by the functor $\chi_{j,i}$ from Lemma~\ref{chi}.

Everything coheres:

\begin{lemma}
\label{epsilon_chi}
If $i \leq j \in I$ and $f \in S_2(p)$, then $\overline{\chi_{j,i}} (\epsilon_j(f)) = \epsilon_i(f)$.
\end{lemma}

\begin{proof}
For convenience, we assume that $\dom(f) = [2] = \{0,1,2\}$.  Also, in the proof of this lemma, we write ``$f^i_{k,\ell}$'' for ``$f^i_{\{k,\ell\}, [2]}$'' (as in Notation~\ref{faces}).

\begin{claim}
If $i \leq j$, then $\chi_{j,i}(f^j_{k,\ell}) = f^i_{k,\ell}$.
\end{claim}

\begin{proof}
The left-hand side is, by definition, equal to

$$\chi_{j,i} \left( f^{\{k,\ell\}}_{[2]} ( \alpha_j(f \upharpoonright \{k, \ell\}))  \right) = \left[ \pi_{j,i}\left( \langle f^{\{k,\ell\}}_{[2]} (\alpha_j(f \upharpoonright \{k, \ell\})) \rangle_j \right) \right]_i$$

(using (6) of Lemma~\ref{chi}).  But the map $f^{\{k,\ell\}}_{[2]}$ is elementary and the functions $\pi_{j,i}, \langle \cdot \rangle_j$, and $\left[ \cdot \right]_i$ are all definable, so this expression equals

$$f^{\{k,\ell\}}_{[2]} \left(\left[\pi_{j,i}(\langle \alpha_j(f \upharpoonright \{k,\ell\}) \rangle_j) \right]_i \right) = f^{\{k,\ell\}}_{[2]} \left(\chi_{j,i}(\alpha_j(f \upharpoonright \{k,\ell\})) \right)$$ $$=f^{\{k,\ell\}}_{[2]}\left(\alpha_i(f \upharpoonright \{k,\ell\}) \right),$$

by our choice of the $\alpha_i$ functions such that $\chi_{j,i} \circ \alpha_j = \alpha_i$.  But this last expression equals the right-hand side in the Claim.
\end{proof}

To prove the lemma, first pick some (any) morphism $g \in \mor_{\cG_j}(a^j_1, f_0)$, and note that $\epsilon_j(f)$ is an element of the group $G_j$ which is represented by the following morphism in $\mor_{\cG_j}(a^j_1, a^j_1)$:

$$g^{-1} \circ (f^j_{0,2})^{-1} \circ f^j_{1,2} \circ f^j_{0,1} \circ g .$$

So $\ov{\chi_{j,i}}(\epsilon_j(f))$ is represented by the morphism

$$\chi_{j,i} \left(g^{-1} \circ (f^j_{0,2})^{-1} \circ f^j_{1,2} \circ f^j_{0,1} \circ g \right)$$ $$= \chi_{j,i}(g)^{-1} \circ \chi_{j,i}(f^j_{0,2})^{-1} \circ \chi_{j,i}(f^j_{1,2}) \circ \chi_{j,i}(f^j_{0,1}) \circ \chi_{j,i}(g),$$ which, by the Claim above, equals $$\chi_{j,i}(g)^{-1} \circ (f^i_{0,2})^{-1} \circ f^i_{1,2} \circ f^i_{0,1} \circ \chi_{j,i}(g),$$ which, by definition, is a representative of $\epsilon_i(f)$.
\end{proof}

Let $G$ be the limit of the inverse system of groups $\langle G_i : i \in I \rangle$ with transition maps given by the $\ov{\chi_{j,i}} : G_j \rightarrow G_i$.  By Lemma~\ref{epsilon_chi}, the maps $\widetilde{\epsilon}_i$ induce a group homomorphism $\epsilon: H_2(p) \rightarrow G$.

\begin{lemma}
\label{injectivity}
The map $\epsilon: H_2(p) \rightarrow G$ is injective.  In other words, if $c \in Z_2(p)$ and $\epsilon_i(c) = 0$ for every $i \in I$, then $c \in B_2(p)$.
\end{lemma}

\begin{proof}
Since $Z_2(p)$ is generated over $B_2(p)$ by all the $2$-shells, it is enough to prove this in the case where $c$ is a $2$-shell of the form $f_{\widehat{0}} - f_{\widehat{1}} + f_{\widehat{2}} - f_{\widehat{3}}$, where $f_{\widehat{a}}$ is a $2$-simplex with domain $\P([3] \setminus \{a\})$.  We will construct a $3$-simplex $g: \P([3]) \rightarrow \C$ such that $\bd (g) = c$.

Pick some $a_3 \models p | (a_0, a_1, a_2)$, so that $(a_0, a_1, a_2, a_3) \models p^{(4)}$.  We will construct $g$ so that $g([3]) = \ov{a_{[3]}}$.  If $(b,c,d,e)$ is some permutation of $(0,1,2,3)$, then $f_{b,c,d}(\{b,c\}) = f_{b,c,e}(\{b,c\})$ (since $\bd c = 0$), and we can assume that $f_{b,c,d}(\{b,c\}) = \ov{a_{b,c}} = f_{b,c,e}(\{b,c\})$.

As a first step in defining the simplex $g$, for any $\{b,c\} \subseteq \{0,1,2,3\}$, we let $g \upharpoonright \{b,c\} = f \upharpoonright \{b,c,d\}$ (where $d$ is any other element of $[3]$), and we let the maps $g^{b}_{[3]} : \ov{a_{b}} \rightarrow \ov{a_{[3]}}$ be the inclusion maps.  We take the transition map $g^{b}_{[3]}$ (for $b \in [3]$) to be the identity map from $\ov{a_b}$ to itself.

Next we will define the transition maps $g^{b,c}_{[3]} : \ov{a_{b,c}} \rightarrow \ov{a_{[3]}}$ in such a way as to ensure compatibility with the faces $f_{\widehat{b}}$.  Before doing this, we set some notation.  First, we write ``$f^i_{xy, \widehat{z}}$'' for the set $(f_{\widehat{z}})^i_{\{x,y\}, [3] \setminus \{z\}}$ as in Notation~\ref{faces}.  Similarly, we write $$f_{bc, \widehat{d}} := (f_{\widehat{d}})^{b,c}_{[3] \setminus \{d\}} (\ov{a_{b,c}}).$$  We consider the sets $\ov{a_{b,c}}$ to be $1$-simplices in which all of the transition maps are inclusions and the ``vertices'' are $\ov{a_b}$ and $\ov{a_c}$.  This allows us to write``$\alpha_i(\ov{a_{b,c}})$.''  For $i \in I$ and $\{b,c\} \subseteq [3]$, let $e^i_{bc}$ be the ``edge'' $\alpha_i(\ov{a_{b,c}})$.

We define the maps $g^{03}_{[3]}$, $g^{13}_{[3]}$, and $g^{23}_{[3]}$ to be the identity maps.  Then we define the other three edge transition maps $g^{01}_{[3]}$, $g^{12}_{[3]}$, and $g^{02}_{[3]}$ so that for every $i \in I$,

\begin{equation}\label{123} g^{13}_{[3]}(e^i_{13}) \, g^{23}_{[3]}(e^i_{23}) \, g^{12}_{[3]}(e^i_{12}) \equiv_{\acl(\emptyset)} f^i_{13, \widehat{0}} \, f^i_{23, \widehat{0}} \, f^i_{12, \widehat{0}},\end{equation}

\begin{equation}\label{023}g^{03}_{[3]}(e^i_{03}) \, g^{23}_{[3]}(e^i_{23}) \, g^{02}_{[3]}(e^i_{02}) \equiv_{\acl(\emptyset)} f^i_{03, \widehat{1}} \, f^i_{23, \widehat{1}} \, f^i_{02, \widehat{1}},\end{equation}

and

\begin{equation}\label{013} g^{03}_{[3]}(e^i_{03}) \, g^{13}_{[3]}(e^i_{13}) \, g^{01}_{[3]}(e^i_{01}) \equiv_{\acl(\emptyset)} f^i_{03, \widehat{2}} \, f^i_{13, \widehat{2}} \, f^i_{01, \widehat{2}}. \end{equation}

Having specified values according to the three equations above, we let $g^{01}_{[3]}$, $g^{12}_{[3]}$, and $g^{02}_{[3]}$ be \emph{any} elementary extensions to the respective domains $\ov{a_{03}}$, $\ov{a_{13}}$, and $\ov{a_{23}}$.

\begin{claim}
\label{isolation1}
For any $i \in I$,
\begin{equation}\label{012}g^{02}_{[3]}(e^i_{02}) \, g^{12}_{[3]}(e^i_{12}) \, g^{01}_{[3]}(e^i_{01}) \equiv f^i_{02, \widehat{3}} \, f^i_{12, \widehat{3}} \, f^i_{01, \widehat{3}}.\end{equation}
\end{claim}

\begin{proof}
Note that by stationarity, $$g^{02}_{[3]}(e^i_{02}) \, g^{12}_{[3]}(e^i_{12}) \equiv f^i_{02, \widehat{3}} \, f^i_{12, \widehat{3}} ,$$ and to check the Claim, it suffices to show that

$$\left[ g^{02}_{[3]}(e^i_{02})^{-1} \circ g^{12}_{[3]}(e^i_{12}) \circ g^{01}_{[3]}(e^i_{01}) \right]_{G_i} = \left[(f^i_{02, \widehat{3}})^{-1} \circ f^i_{12, \widehat{3}} \circ f^i_{01, \widehat{3}}\right]_{G_i}.$$

The right-hand side equals $\epsilon_i(f_{\widehat{3}})$.  Since $\epsilon_i(c) = 0$, $$\epsilon_i(f_{\widehat{3}}) = \epsilon_i(f_{\widehat{0}}) - \epsilon_i(f_{\widehat{1}}) + \epsilon_i(f_{\widehat{2}}).$$  Let ``$g_{bc}$'' be an abbreviation for $g^{bc}_{[3]}(e^i_{bc})$.  By applying equations~\ref{123}, \ref{023}, and \ref{013} above (and performing a very similar calculation as in the proof of Lemma~\ref{epsilon_boundaries}), we get:

$$\epsilon_i(f_{\widehat{3}}) = \left[g_{13}^{-1} \circ g_{23} \circ g_{12} \right]_{G_i} - \left[g^{-1}_{03} \circ g_{23} \circ g_{02} \right]_{G_i} + \left[g^{-1}_{03} \circ g_{13} \circ g_{01}\right]_{G_i}$$

$$= \left[g^{-1}_{01} \circ g^{-1}_{13} \circ g_{23} \circ g_{12} \circ g_{01}\right]_{G_i} - \left[g^{-1}_{03} \circ g_{23} \circ g_{02} \right]_{G_i} + \left[g^{-1}_{03} \circ g_{13} \circ g_{01}\right]_{G_i}$$

$$= -\left[g^{-1}_{03} \circ g_{23} \circ g_{02} \right]_{G_i} + \left[g^{-1}_{03} \circ g_{13} \circ g_{01}\right]_{G_i} +  \left[g^{-1}_{01} \circ g^{-1}_{13} \circ g_{23} \circ g_{12} \circ g_{01}\right]_{G_i}$$

$$= \left[g^{-1}_{02} \circ g^{-1}_{23} \circ g_{03} \circ g^{-1}_{03} \circ g_{13} \circ g_{01} \circ g^{-1}_{01} \circ g^{-1}_{13} \circ g_{23} \circ g_{12} \circ g_{01} \right]_{G_i}$$

$$= \left[g^{-1}_{02} \circ g_{12} \circ g_{01}\right]_{G_i},$$

as desired.
\end{proof}

Now we must check that this coheres with the types of the given simplices $f_{\widehat{b}}$:

\begin{claim}
\label{isolation2}
If $(b,c,d,e)$ is a permutation of $[3]$ with $0 \leq b < c < d \leq 3$, then $$g^{bd}_{[3]}(\ov{a_{bd}}) \, g^{cd}_{[3]}(\ov{a_{cd}}) \, g^{bc}_{[3]}(\ov{a_{bc}}) \equiv f_{bd, \widehat{e}} \, f_{cd, \widehat{e}} \, f_{bd, \widehat{e}}.$$
\end{claim}

\begin{proof}
Let $$\widetilde{f}_{xy, \widehat{e}} := \bigcup_{i \in I} f^i_{xy, \widehat{e}}.$$  Then Claim~\ref{isolation2} follows from Claim~\ref{isolation1} above together with:

\begin{subclaim}
If $(x,y,z)$ is any permutation of $(b,c,d)$ with $x < y$, then $\tp(f_{xy, \widehat{e}} / f_{yz,\widehat{e}} f_{xz, \widehat{e}})$ is isolated by $\tp(f_{xy, \widehat{e}} / \widetilde{f}_{yz,\widehat{e}} \widetilde{f}_{xz,\widehat{e}})$.
\end{subclaim}

\begin{proof}
Note that $f_{xy,\widehat{e}} \subseteq \acl(f_{yz,\widehat{e}}, f_{xz,\widehat{e}})$ (in fact, it is in the algebraic closure of the ``vertices'' $f_{x,\widehat{e}} \subseteq f_{xz,\widehat{e}}$ and $f_{y,\widehat{e}} \subseteq f_{yz,\widehat{e}}$).  Suppose towards a contradiction that $h \in f_{xy,\widehat{e}}$ but $$\tp(h / \widetilde{f}_{yz,\widehat{e}} \widetilde{f}_{xz,\widehat{e}}) \nvdash \tp(h / f_{yz,\widehat{e}} f_{xz,\widehat{e}}).$$  This means that the orbit of $h$ under $\Aut(\C / f_{yz,\widehat{e}} f_{xz,\widehat{e}})$ is smaller than the orbit of $h$ under $\Aut(\C / \widetilde{f}_{yz,\widehat{e}} \widetilde{f}_{xz,\widehat{e}})$.  Let $\widehat{h}$ be a name for the orbit of $h$ under $\Aut(\C / f_{yz,\widehat{e}} f_{xz,\widehat{e}})$ as a set.  Then $$\widehat{h} \in \dcl(f_{yz,\widehat{e}}, f_{xz,\widehat{e}}) \setminus \dcl(\widetilde{f}_{yz,\widehat{e}}, \widetilde{f}_{xz,\widehat{e}}).$$  Since $\widehat{h} \in \dcl(f_{yz,\widehat{e}} f_{xz,\widehat{e}})$, it lies in $f^i_{xy, \widehat{e}}$ for some $i \in I$ (this is by the maximality condition on our symmetric witnesses $\langle W_i : i \in I \rangle$).  Also, $$f^i_{xy, \widehat{e}} \subseteq \dcl(f^i_{yz, \widehat{e}}, f^i_{xz, \widehat{e}})$$ due to the fact that $f^i_{xy, \widehat{e}}$ is interdefinable with the set of all morphisms in $\mor_{\cG_i}(f_{x,\widehat{e}}, f_{y, \widehat{e}})$, which can be obtained via composition in $\cG_i$ from the corresponding morphisms in $\dcl(f^i_{yz, \widehat{e}})$ and $\dcl(f^i_{xz, \widehat{e}})$.  But this contradicts the fact that $\widehat{h} \notin \dcl(\widetilde{f}_{yz,\widehat{e}} \widetilde{f}_{xz,\widehat{e}}).$
\end{proof}

\end{proof}

Claim~\ref{isolation2} implies that for each permutation $(b,c,d,e)$ of $[3]$, we can find an elementary map $g^{b,d,c}_{[3]}$ from the ``face'' $f_{\widehat{e}}([3] \setminus \{e\})$ onto $\ov{a_{b,c,d}}$ which is coherent with the maps $g^{b,c}_{[3]}$, $g^{c,d}_{[3]}$, and $g^{b,d}_{[3]}$ that we have already defined, and such that $\bd^i g = f_{\widehat{i}}$.  This completes the proof of Lemma~\ref{injectivity}.

\end{proof}

\begin{lemma}
\label{surjectivity}
The map $\epsilon: H_2(p) \rightarrow G$ is surjective.
\end{lemma}

\begin{proof}
Suppose that $g$ is any element in $G$, and that $g$ is represented by a sequence $\langle g_i : i \in I \rangle$ such that $\ov{\chi_{j,i}}(g_j) = g_i$ whenever $i \leq j$.  We will construct a $2$-chain $c = f-h$ such that $\epsilon_i(f-h) = g_i$ for every $i \in I$, which will establish the Lemma.  Let $f : \P([2]) \rightarrow \C$ be the $2$-simplex such that $f(s) = \ov{a_s}$ for every $s \subseteq [2]$ and such that every transition map in $f$ is an inclusion map.  Let $k_i = \epsilon_i(f)$.

We want to construct $h: \P([2]) \rightarrow \C$ such that $h([2]) = \ov{a_{[2]}}$, $\bd(h) = \bd(f)$, and $h^s_{[2]}$ is the identity map whenever $s \subseteq \{0,1\}$ or $s \subseteq \{1,2\}$.  The only thing left is to specify an elementary map $h^{02}_{[2]}: \ov{a_{02}} \rightarrow \ov{a_{02}}$ fixing $\ov{a_0}$ and $\ov{a_2}$ pointwise.

\begin{claim}
Suppose that $h_i \in \mor_{\cG_i}(a^i_0, a^i_2)$ is the unique element such that $\left[ h_i^{-1} \circ h^i_{12} \circ h^i_{01} \right]_{G_i} = k_i - g_i$.  Then

\begin{enumerate}
\item whenever $i \leq j$,  $\chi_{j,i}(h_j) = h_i$, and
\item $\tp(h_0, \ldots, h_i / \ov{a_0}, \ov{a_2}) = \tp(\alpha_0(\ov{a_{02}}), \ldots, \alpha_i(\ov{a_{02}}) / \ov{a_0}, \ov{a_2}).$
\end{enumerate}
\end{claim}

\begin{proof}

First we show:

\begin{subclaim}
$\chi_{i+1,i}(h^{i+1}_{12}) = h^i_{12}$ and $\chi_{i+1,i}(h^{i+1}_{01}) = h^i_{01}$.
\end{subclaim}

\begin{proof}
We check only the first equation (and the second equation has an identical proof).  By (6) of Lemma~\ref{chi},

$$\chi_{i+1,i}(h^{i+1}_{12}) = \left[ \pi_{i+1,i}(\langle h^{i+1}_{12} \rangle_{i+1}) \right]_i = \left[ \pi_{i+1,i} (\langle h^{12}_{[2]}(\alpha_{i+1}(\ov{a_{12}})) \rangle_{i+1}) \right]_i$$

$$= h^{12}_{[2]} \left( \left[ \pi_{i+1,i}(\langle\alpha_{i+1}(\ov{a_{12}}) \rangle_{i+1}) \right]_i \right) =h^{12}_{[2]}(\chi_{i+1,i}(\alpha_{i+1}(\ov{a_{12}})))$$

$$= h^{12}_{[2]}(\alpha_i(\ov{a_{12}})) = h^{12}_i.$$
\end{proof}

Note that it is enough to prove (1) of the Claim for every pair $(i,j)$ where $j = i+1$.  We apply $\ov{\chi_{i+1,i}}$ to both sides of the equation $\left[ h_{i+1}^{-1} \circ h^{i+1}_{12} \circ h^{i+1}_{01} \right]_{G_{i+1}} = k_{i+1} - g_{i+1}$.  On the right-hand side, this yields

\begin{equation}\label{rhs} \ov{\chi_{i+1,i}}(k_{i+1} - g_{i+1}) = k_i - g_i. \end{equation}

On the left-hand side, using the Subclaim, we get

\begin{equation}\label{lhs}\left[ \chi_{i+1,i} \left(h^{-1}_{i+1} \circ h^{i+1}_{12} \circ h^{i+1}_{01} \right)\right]_{\cG_i} = \left[\chi_{i+1,i}(h_{i+1})^{-1} \circ h^i_{12} \circ h^i_{01}\right]_{\cG_i}. \end{equation}

So putting together Equations~\ref{rhs} and \ref{lhs}, we get that $$\left[\chi_{i+1,i}(h_{i+1})^{-1} \circ h^i_{12} \circ h^i_{01}\right]_{\cG_i} = k_i - g_i.$$  But $h_i$ is the unique morphism in $\cG_i$ such that $\left[ h_i^{-1} \circ h^i_{12} \circ h^i_{01} \right]_{\cG_i} = k_i - g_i$, so part (1) of the Claim follows.

We prove part (2) by induction on $i \in I$.  The base case follows from $$\tp(\langle h_0 \rangle_0 / \ov{a_0}, \ov{a_2}) = \tp(\langle \alpha_0(\ov{a_{02}}) \rangle_0 / \ov{a_0}, \ov{a_2})$$ (which is true simply because both elements belong to $SW(a^i_0, a^i_2)$).  If (2) is true for $i$, then to prove it for $i+1$, it is enough to check that \begin{equation}\label{induction2}\tp(\langle h_{i+1} \rangle_{i+1}, \langle h_i \rangle_i / \ov{a_0}, \ov{a_2}) = \tp(\langle \alpha_{i+1}(\ov{a_{02}})\rangle_{i+1}, \langle \alpha_i(\ov{a_{02}})\rangle_i / \ov{a_0}, \ov{a_2}),\end{equation} since all the other elements are in the definable closure of $h_i$ and $\alpha_i(\ov{a_{01}})$ via the maps $\chi_{k,\ell}$.  To see this, first note that $$\tp(\langle h_{i+1} \rangle_{i+1} / \ov{a_0}, \ov{a_2}) = \tp(\langle \alpha_{i+1}(\ov{a_{02}})\rangle_{i+1} / \ov{a_0}, \ov{a_2})$$ just because both elements belong to $SW(a^{i+1}_0, a^{i+1}_2)$.  By part~(1) of the Claim, $\pi_{i+1,i}(\langle h_{i+1} \rangle_{i+1}) = \langle h_i \rangle_i$, and by the way we chose the $\alpha$ functions, $\pi_{i+1,i}(\langle \alpha_{i+1}(\ov{a_{02}}) \rangle_{i+1}) = \langle \alpha_i(\ov{a_{02}}) \rangle_i$.  Since the function $\pi_{i+1,i}$ is definable, Equation~\ref{induction2} follows.
\end{proof}

Given the elements $h_i$ as in the Claim above, we let $h^{02}_{[2]}: \ov{a_{02}} \rightarrow \ov{a_{02}}$ be any elementary map that fixes $\ov{a_{0}} \cup \ov{a_{2}}$ pointwise and maps each element $\alpha_i(\ov{a_{02}})$ to $h_i$.  Then $\epsilon_i(f-h) = k_i - (k_i - g_i) = g_i$, as desired..

\end{proof}

By Lemmas~\ref{injectivity} and \ref{surjectivity}, $H_2(p) \cong G$.  To finish the proof of Theorem~\ref{hurewicz}, we just need to show:

\begin{lemma}
$G \cong \Aut(\widetilde{a_0 a_1} / \ov{a_0}, \ov{a_1})$
\end{lemma}

\begin{proof}
Note that $\Aut(\widetilde{a_0 a_1} / \ov{a_0}, \ov{a_1})$ is the limit of the groups $\Aut(SW_i / \ov{a_0}, \ov{a_1})$ via the transition maps $\rho_{j,i}: \Aut(SW_j / \ov{a_0}, \ov{a_1}) \rightarrow \Aut(SW_j / \ov{a_0}, \ov{a_1})$, due to the maximality condition that every element of $\widetilde{a_0 a_1}$ lies in one of the symmetric witnesses $SW_i$.  Also, part (8) of Lemma~\ref{chi} implies that we have a commuting system \bigskip

$\begin{CD}
G_j @>\ov{\chi_{j,i}}>> G_i\\
@VV\ov{\psi_j}V @VV\ov{\psi_i}V\\
\Aut(SW_j / \ov{a_0}, \ov{a_1}) @>\rho_{j,i}>> \Aut(SW_i / \ov{a_0}, \ov{a_1})
\end{CD}$
\bigskip

But the maps $\ov{\psi_i}$ are all isomorphisms, so taking limits we get an isomorphism from $G$ to $\Aut(\widetilde{a_0 a_1} / \ov{a_0}, \ov{a_1})$.
\end{proof}

\section{Any profinite abelian group can occur as $H_2(p)$}

In this section, we construct a family of examples which prove the following:

\begin{theorem}
For any profinite abelian group $G$, there is a type $p$ in a stable theory $T$ such that $H_2(p) \cong G$.  In fact, we can build the theory $T$ to be totally categorical.
\end{theorem}

Together with Theorem~\ref{hurewicz} from the previous section, this shows that the groups that can occur as $H_2(p)$ for a type $p$ in a stable theory are \emph{precisely} the profinite abelian groups.

For the remainder of this section, we fix a profinite abelian group $G$ which is the inverse limit of the system $\langle H_i : i \in I \rangle$, where each $H_i$ is finite and abelian, $(I, \leq)$ is a directed set, and $G$ is the limit along the surjective group homomorphisms $\phi_{j,i} : H_j \rightarrow H_i$ (for every pair $i \leq j$ in $I$).  The language $L$ of $T$ will be as follows: there will be a sort $\cG_i$ for each $i \in I$, and function symbols $\chi_{j,i} : \cG_j \rightarrow \cG_i$ for every pair $i \leq j$.  The theory $T$ will say, in the usual language of categories, that each $\cG_i$ is a connected groupoid with infinitely many objects, and there will be separate composition symbols for each sort $\cG_i$.  Also, $T$ says that $\cG_i$ is a groupoid such that each vertex group $\mor_{\cG_i}(a_i,a_i)$ is isomorphic to the group $H_i$.  For convenience, pick some arbitrary $a_i \in \ob(\cG_i)$ and some group isomorphism $\xi_i : G_i \rightarrow \mor_{\cG_i}(a_i,a_i)$ (but the $\xi_i$'s are \textbf{not} a part of any model of $T$).  Then the last requirement we make on $T$ is that the function symbols $\chi_{j,i}$ define full functors from $\cG_j$ onto $\cG_i$ which induce bijections between the corresponding collections of objects, and such that for every pair $i \leq j$, the following diagram commutes:

$$\begin{CD}
H_j @>\phi_{j,i}>> H_i\\
@VV\xi_jV @VV\xi_iV\\
\mor_{\cG_j}(a_j,a_j) @>\chi_{j,i}>> \mor_{\cG_i}(a_i,a_i)
\end{CD}$$

(In other words, the functors $\chi_{j,i}$ are just ``isomorphic copies'' the group homomorphisms $\phi_{j,i}$.)

\begin{lemma}
The theory $T$ described above is complete and admits elimination of quantifiers.  If we further assume that the language is multi-sorted and that every element of a model must belong to one of the sorts $\cG_i$, then $T$ is totally categorical.
\end{lemma}

\begin{proof}
If the language is multi-sorted, then since the groupoids $\cG_i$ are all connected and there are bijections between the object sets of the various $\cG_i$, the isomorphism class of a model of $T$ is determined by the cardinality of the object set of some (any) $\cG_i$. This shows that $T$ is totally categorical, hence $T$ is complete.

For quantifier elimination, it suffices to show the following: for any two models $M_1$ and $M_2$ of $T$ with a common substructure $A$ and any sentence $\sigma$ with parameters from $A$ of the form $\sigma = \exists x \varphi(x; \ov{a})$ where $\varphi$ is quantifier-free, if $M_1 \models \sigma$, then $M_2 \models \sigma$. (See Theorem~8.5 of \cite{tentziegler}.) In this situation, let $\cl(A)$ denote the submodel of $M_1$ (and of $M_2$) generated by $A$, and in case $A = \emptyset$, let $\cl(A) = \emptyset$. Then if $M_1 \models \sigma$ as above, at least one of the following is true:

\begin{enumerate}
\item $\varphi(x; \ov{a})$ is satisfied by some $x$ in $\cl(A)$;
\item $\varphi(x; \ov{a})$ is satisfied by some morphism between two objects in $\cl(A)$;
\item For some $i \in I$, $\varphi(x; \ov{a})$ is satisfied by \emph{any} object in $\cG_i$ outside of $\cl(A)$;
\item For some $i \in I$, $\varphi(x; \ov{a})$ is satisfied by \emph{any} morphism in $\cG_i$ which goes from [or to] some particular $b \in \cl(A)$ and goes to [or from] \emph{any} object in $\cG_i$ outside of $\cl(A)$; or else
\item For some $i \in I$, $\varphi(x; \ov{a})$ is satisfied by \emph{any} morphism in $\cG_i$ whose source and target are both outside of $\cl(A)$.
\end{enumerate}

In each of the five cases above, it is straightforward to check that there is an $x$ realizing $\varphi(x; \ov{a})$ in $M_2$ as well (for the last three cases we use the fact that $\ob(\cG_i)$ is infinite).

\end{proof}

\begin{remark}
\label{acl}
If $A \subseteq \cG_i$, then we say that $b \in \ob(\cG_i)$ is \emph{connected to $A$} if either $b \in A$ or $b$ is the source or target of a morphism in $A$.  By elimination of quantifiers, it follows that for any $A \subseteq \cG_i$, $\acl(A) \cap \cG_i$ is the union of all objects $b$ that are connected to $A$ plus all morphisms $f \in \mor_{\cG_i}(b,c)$ such that $b$ and $c$ are connected to $A$.

Because of the functors $\chi_{j,i}$, it follows that for any $a$ in any $\cG_i$, $\acl(a)$ actually contains objects and morphisms from each of the groupoids $\cG_j$.  But for any $A \subseteq \C$, we can write $\acl(A)$ in the ``standard form'' $\acl(A) = \acl(A_0)$ for some $A_0 \subseteq \ob(\cG_0)$, and:

1. $\acl(A_0) \cap \ob(\cG_i) = \chi_{i,0}^{-1}(A_0)$, and

2. $\acl(A_0) \cap \mor(\cG_i)$ is the collection of all $f \in \mor_{\cG_i}(b,c)$ where $b,c \in \acl(A_0)$.
\end{remark}

\begin{lemma}
\label{weak_EI}
The theory $T$ has \emph{weak elimination of imaginaries} in the sense of \cite{Poizat}: for every formula $\varphi(\ov{x}, \ov{a})$ defined over a model $M$ of $T$, there is a smallest algebraically closed set $A \subseteq M$ such that $\varphi(\ov{x}, \ov{a})$ is equivalent to a formula with parameters in $A$.
\end{lemma}

\begin{proof}
By Lemma~16.17 of \cite{Poizat}, it suffices to prove the following two statements:

1. There is no strictly decreasing sequence $A_0 \supsetneq A_1 \supsetneq \ldots$, where every $A_i$ is the algebraic closure of a finite set of parameters; and

2. If $A$ and $B$ are algebraic closures of finite sets of parameters in the monster model $\C$, then $\Aut(\C / A \cap B)$ is generated by $\Aut(\C / A)$ and $\Aut(\C / B)$.

Statement 1 follows immediately from the characterization of algebraically closed sets in Remark~\ref{acl} above (that is, algebraic closures of finite sets are equivalent to algebraic closures of finite subsets of $\ob(\cG_0)$).

To check statement~2, suppose that $\sigma \in \Aut(\C / A \cap B)$, and assume that $A = \acl(A_0)$ and $B = \acl(B_0)$ where $A_0, B_0 \subseteq \ob(\cG_0)$.  Note that \emph{any} permutation of $\ob(\cG_0)(\C)$ which fixes $A_0$ can be extended to an automorphism of $\Aut(\C / A)$, and likewise for $B_0$ and $B$.  So as a first step, we can use the fact that $\textup{Sym}(\ob(\cG_0) / A_0 \cap B_0)$ is generated by $\textup{Sym}(\ob(\cG_0) / A_0)$ and $\textup{Sym}(\ob(\cG_0) /  B_0)$ to find an automorphism $\tau \in \Aut(\C)$ such that $\tau$ is in the subgroup generated by $\Aut(\C / A)$ and $\Aut(\C / B)$ and $\sigma \circ \tau^{-1}$ fixes $\ob(\cG_0)$ (and hence $\ob(\cG_i)$ for every $i$) pointwise.


Finally, we need to deal with the morphisms.  We claim that there is a map $\sigma^0_A \in \Aut(\C / A)$ which fixes $\ob(\cG_0)$ pointwise and such that for any $f \in \mor_{\cG_0}(b,c)$ such that at least one of $b$ and $c$ do not lie in $A$, $(\sigma^0 \circ \tau)(f) = \sigma(f)$.  (The idea is to use the recipe for constructing object-fixing automorphisms described in subsection~4.2 of \cite{GK}, using a basepoint $a_0 \in A$.)  In fact, by the same argument we can also assume that for \emph{every} $i \in I$ and for any $f \in \mor_{\cG_i}(b,c)$ such that at least one of $b$ and $c$ do not lie in $A$, $(\sigma^0 \circ \tau)(f) = \sigma(f)$.  Similarly, there is a map $\sigma^0_B \in \Aut(\C / B)$ which fixes $\ob(\cG_0)$ pointwise and for any $i \in I$, $\sigma^0_B$ only moves morphisms in $\mor_{\cG_i}(b,c)$ where $b$ and $c$ are both in $A \setminus (A \cap B)$, and such that $\sigma^0_B \circ \sigma^0_A \circ \tau = \sigma$.

\end{proof}

\begin{lemma}
\label{acl2}
If $a^0, a^1 \in \ob(\cG_i)$, then $$\acl^{eq}(a^0, a^1) = \dcl^{eq} \left(\bigcup_{i,j \in I; \, i \leq j} \mor_{\cG_j}(a^0_j, a^1_j) \right),$$ where $a^\ell_j = \chi^{-1}_{j,i}(a^\ell)$.
\end{lemma}

\begin{proof}
Suppose $g \in \acl^{eq}(a^0, a^1)$.  Then $g = b / E$ for some $(a^0, a^1)$-definable finite equivalence relation $E$.  By Lemma~\ref{weak_EI}, there is a finite tuple $\ov{d} \in \C$ (in the home sort) such that $b / E$ is definable over $\ov{d}$ and $\ov{d}$ has a minimal algebraic closure.  If the set $\acl(\ov{d})$ contained an object $a$ of $\cG_0$ other than $\chi_{i,0}(a^0)$ and $\chi_{j,0}(a^1)$, then (by quantifier elimination) $\acl(\ov{d})$ would have an infinite orbit under $\Aut(\C / a^0, a^1)$, and so $E$ would have infinitely many classes, a contradiction.  So by Remark~\ref{acl}, the set $\ov{d}$, and hence $b / E$ is definable over the union of the morphism sets $\mor_{\cG_j}(a^0_j, a^1_j)$.
\end{proof}

\textbf{From now on, we assume that all algebraic and definable closures are computed in $T^{eq}$, not just in the home sort.}

\begin{lemma}
If $a^0, a^1 \in \ob(\cG_i)$, then for any two $f, g \in \mor_{\cG_i}(a^0, a^1)$, $$\tp(f / \acl(a^0), \acl(a^1)) = \tp(g / \acl(a^0), \acl(a^1)).$$

\end{lemma}

\begin{proof}
Using the same procedure as described in subsection~4.2 of \cite{GK}, we can construct an automorphism $\sigma$ of $\C$ fixing $\ob(\cG_i)$, $\mor_{\cG_i}(a^0, a^0)$, and $\mor_{\cG_i}(a^1, a^1)$ pointwise while mapping $f$ to $g$.  (In the construction of \cite{GK}, the ``basepoint'' $a_0$ there can be chosen to be $a^0$ here, and then condition (5) of the construction plus the fact that $\cG_i$ is abelian implies that $\mor_{\cG_i}(a^1, a^1)$ is fixed.)  In fact, it is easy to see that we can even ensure that $\sigma$ fixes $\mor_{\cG_j}(\chi_{j,i}^{-1}(a^0), \chi_{j,i}^{-1}(a^0))$ and $\mor_{\cG_j}( \chi_{j,i}^{-1}(a^1),  \chi_{j,i}^{-1}(a^1))$ pointwise, so by Lemma~\ref{acl2}, $\sigma$ fixes $\acl(a^0) \cup \acl(a^1)$ pointwise.
\end{proof}

Let $p = \stp(a_0)$ for some (any) $a_0 \in \ob(\cG_0)$.

\begin{proposition}
$H_2(p) \cong G$.
\end{proposition}

\begin{proof}
Pick $(a^0, a^1, a^2) \models p^{(3)}$.  By Theorem~\ref{hurewicz} (the ``Hurewicz theorem''), it is enough to show that $\Aut(\widetilde{a^0 a^1} / \ov{a^0}, \ov{a^1}) \cong G$.  For ease of notation, let $a^k_i = \chi^{-1}_{i,0}(a^k)$ for $k = 0, 1,$ or $2$.  By Lemma~\ref{acl2} and the fact that any morphism in $\mor_{\cG_i}(a^0_i, a^1_i)$ is a composition of morphisms in $\mor_{\cG_i}(a^0_i, a^2_i)$ and $\mor_{\cG_i}(a^2_i, a^1_i)$, it follows that the set $\widetilde{a^0 a^1}$ is interdefinable with $\displaystyle\bigcup_{i  \in I} \mor_{\cG_i}(a^0_i, a^1_i) $.

So $\Aut(\widetilde{a^0 a^1} / \ov{a^0}, \ov{a^1})$ is the inverse limit of the groups $\Aut(\mor_{\cG_i}(a^0_i, a^1_i) / \ov{a^0}, \ov{a^1})$ under the natural homomorphisms $$\rho_{j,i} : \Aut(\mor_{\cG_j}(a^0_j, a^1_j) / \ov{a^0}, \ov{a^1}) \rightarrow \Aut(\mor_{\cG_i}(a^0_i, a^1_i) / \ov{a^0}, \ov{a^1})$$ induced by the fact that $\mor_{\cG_i}(a^0_i, a^1_i)$ is in the definable closure of $\mor_{\cG_i}(a^0_j, a^1_j)$ when $j \geq i$.

By the way we defined our theory $T$, we can select a system of group isomorphisms $\lambda_i : H_i \rightarrow \mor_{\cG_i}(a^1_i, a^1_i)$ for $i \in I$ such that the following diagram commutes:

$$\begin{CD}
H_j @>\phi_{j,i}>> H_i\\
@VV\lambda_jV @VV\lambda_iV\\
\mor_{\cG_j}(a_j,a_j) @>\chi_{j,i}>> \mor_{\cG_i}(a_i,a_i)
\end{CD}$$

To finish the proof of the Proposition, it is enough to find a system of group isomorphisms $$\sigma_i : H_i \rightarrow \Aut(\mor_{\cG_i}(a^0_i, a^1_i) / \ov{a^0}, \ov{a^1})$$ such that the following diagram commutes:

$$\begin{CD}
H_j @>\phi_{j,i}>> H_i\\
@VV\sigma_jV @VV\sigma_iV\\
\Aut(\mor_{\cG_j}(a^0_j, a^1_j) / \ov{a^0}, \ov{a^1}) @>\rho_{j,i}>> \Aut(\mor_{\cG_i}(a^0_i, a^1_i) / \ov{a^0}, \ov{a^1})
\end{CD}$$

(Then by the discussion above, $\Aut(\widetilde{a^0 a^1} / \ov{a^0}, \ov{a^1})$ will be isomorphic to the inverse limit of the groups $H_i$, which is $G$.)

We define the maps $\sigma_i$ so that for any $h \in H_i$ and any $g \in \mor_{\cG_i}(a^0_i, a^1_i)$, $$\left[\sigma_i(h)\right](g) = \lambda_i(h) \circ g.$$

(Note that this rule determines a unique elementary permutation of $\mor_{\cG_i}(a^0, a^1)$ fixing $\acl(a^0) \cup \acl(a^1)$ pointwise.)  This is a group homomorphism since $$\left[\sigma_i(h_1 h_2)\right] (g) = \lambda_i(h_1 h_2) \circ g = \lambda_i(h_1) \circ \lambda_i(h_2) \circ g = \left[\sigma_i(h_1) \circ \sigma_i(h_2) \right](g).$$  Clearly $\sigma_i$ is injective, and it is surjective because of the following:

\begin{claim}
For any $f$ and $g$ in $\mor_{\cG_i}(a_i^0, a_i^1)$, there is a \emph{unique} elementary permutation $\sigma$ of $\mor_{\cG_i}(a^0, a^1)$ sending $f$ to $g$ and fixing $\acl(a^0) \cup \acl(a^1)$ pointwise.
\end{claim}

\begin{proof}
If $f = h \circ g $ for $h \in \mor_{\cG_i}(a_i^1, a_i^1)$, then $\sigma(f')$ must equal $h \circ f' $ for any $f' \in \mor_{\cG_i}(a_i^0, a_i^1)$.

\end{proof}

Finally, we must check that the maps $\sigma_i$ commute with $\phi_{j,i}$ and $\rho_{j,i}$.  Pick any $j \geq i$, $h \in H_j$ and $f \in \mor_{\cG_i}(a^0_i, a^1_i)$.  On the one hand,

$$\rho_{j,i}\left(\sigma_j(h) \right) (f) = \chi_{j,i}\left(\sigma_j(h)(f') \right), \textup{  where } \chi_{j,i}(f') = f$$

$$= \chi_{j,i}(\lambda_j(h) \circ f') = \chi_{j,i}(\lambda_j(h)) \circ \chi_{j,i}(f') = \chi_{j,i}(\lambda_j(h)) \circ f.$$

On the other hand,

$$\left[\sigma_i(\phi_{j,i}(h))\right] (f) = \lambda_i(\phi_{j,i}(h)) \circ f = \chi_{j,i} (\lambda_j(h)) \circ f.$$

These last two equations show that $\rho_{j,i} \circ \sigma_j = \sigma_i \circ \varphi_{j,i}$, as desired.
\end{proof}

\begin{remark}
These examples also show that homology groups of types are not always preserved by nonforking extensions.  In the example above, if $A$ is some algebraically closed parameter set containing a point in $p(\C)$ and $q$ is the nonforking extension of $p$ over $A$, then $q$ has $4$-amalgamation, and so (by Corollary~\ref{trivial_homology}) $H_2(q) = 0$.
\end{remark}

\section{Unstable examples}

In this section, we compute some more homology groups for unstable rosy examples.

\begin{example}\label{tet.free}
In this first example, as promised, we argue that all the homology groups of the  tetrahedron-free random ternary hypergraph are trivial, even though it does not have 4-amalgamation.  Let $T_{tet.free}$ be the theory of  such a graph with the ternary relation $\{R\}$. It is well-known that  $T_{tet.free}$ is $\omega$-categorical, simple, has weak elimination of imaginaries, and has $n$-amalgamation for all $n \neq 4$.  Let $p$ be the unique 1-type over $\emptyset$. We first claim that even though $T_{tet.free}$ does not have 4-amalgamation,
Lemma~\ref{shellcycle} still holds.

\begin{claim}\label{thesame}
In  $T_{tet.free}$, for $n\geq 2$, every
$(n-1)$-cycle $c=\sum_i k_if_i$ of type $p$ over $\emptyset$ is a sum of $(n-1)$-shells.
\end{claim}

We sketch the proof, which is almost the same as that of \ref{shellcycle} (but as here we do not
have 4-amalgamation, we need some  trick.) We even use the same notation, letting
 $g_{ij} = \bd^j f_i$ for $(i,j)\in I$ ($j<n$).  We shall  find $(n-1)$-simplices $h_{ij}$ satisfying the conditions
 described in Claim~\ref{cancelling}. Note that due to weak elimination of imaginaries we can assume each vertex of a simplex is just a point in the graph. Now the construction method will be the same: first, pick a point $a^*$ independent from all the points $g_{ij}(\{k\})$.  Then the edges  $h_{ij}(\{k,m\})$, where $m\notin \bigcup_{ij} s_{ij}$, are determined.  For the next level, we need a trick. Namely, given an edge of the from $\{b,c\}=g_{ij}(\{k,\ell\})$, we find
 a point $a=(h_{ij})^{\{m\}}_{\{k,\ell,m\}}(h_{ij}(\{m\}))$ (while we may take $(h_{ij})^{\{k,\ell \}}_{\{k,\ell,m\}}$ as an identity map of $\{b,c\}$) such that $a,b,c$ are distinct and $R(a,b,c)$ {\em does not} hold). Then we can proceed the next level of the construction for $h_{ij}$ as no matter what the triangle $g_{ij}(\{k,\ell,k'\})$ is (whether it satisfies $R$ or not), we can amalgamate the other three triangular faces which do not satisfy $R$. The rest of the proof is the same. We have proved Claim~\ref{thesame}.

 Now to show that $H_n(p)=0$ $(n\geq 1)$, it suffices to see that any $n$-shell is a boundary. Due to $(n+2)$-amalgamation, this is true for any $n\ne 2$. But any 2-shell is a boundary as well. The only case to check is that of a 2-shell $f=f_0-f_1+f_2-f_3$ with support $\{0,1,...,4\}$ such that $f_i(\{0, \ldots, \hat i, \ldots,4\})$ satisfies $R$. But by taking a suitable point with support $\{5\}$ distinct from all such faces, it easily follows $f$ is the boundary of a $3$-fan.

\end{example}

\begin{example}
\label{dlo}
Here we show the theory $T_{dlo}$ of dense linear ordering (without end points) is another example whose homology groups are all trivial even though it does not have 3-amalgamation.  Recall that it has elimination of imaginaries.
Let $p$ be the unique 1-type over $\emptyset$. It is not hard to see that $p$ has $n$-amalgamation for all $n\ne 3$. Now
we claim that, just like in Claim~\ref{thesame}, any $n$-cycle is a sum of $n$-shells.  The proof will be similar, and we use the same notation.  We want to construct the edges $h_{ij}$.
The trick this time is to take $a^*$ greater than all the points of the form   $a'=g_{ij}(\{k\})$.
Then given any edge  $\{b,c\}=g_{ij}(\{k,\ell\})$, where either $b<c$ or $c<b$, pick $a>b,c$.   Then since
$\tp(a'a^*)=\tp(ba)=\tp(ca)$, the construction of $h_{ij}$ on this level is compatible. For the rest of the construction, use $n$-amalgamation.

Due to the claim and $(n+2)$-amalgamation, all of the groups $H_n(p)$ are $0$ for $n\ne 1$.  Furthermore, $H_1(p)=0$ because any 1-shell is the boundary of a 2-fan (choose a point greater than all the vertices of all the terms in the 1-shell).

\end{example}

\begin{example}
\label{parity}
In \cite{KKT}, for each {\bf even} $n\geq 4$,   the theory $U_n$ of the Fraiss\'{e} limit of the following class
$K_n$ is introduced.

Let $R_n$ be an $n$-ary relation symbol.
We consider symmetric and irreflexive $R_n$-structures. For any
$R_n$-structure $A$ with a finite substructure $B$, let
${\mathrm {no}}_A(B)$ denote
the number (modulo $2$) of
$n$-element subsets of $B$ satisfying $R_n$.
If $A$ is clear from the context, we simply write ${\mathrm {no}}(B)$.

Let $(*)_n$ be the following condition on an $R_n$-structure $A$:
\begin{itemize}
\item[$(*)_n$] If $A_0$ is an $(n+1)$-element subset of $A$, then
${\mathrm {no}}(A_0)=0$.
\end{itemize}
Now  $K_n$ is  the class of all finite (symmetric and irreflexive)
$R$-structures satisfying $(*)_n$.

 It is shown in  \cite{KKT} that $U_n$ is $\omega$-categorical, supersimple of $SU$-rank 1, and has quantifier elimination, weak elimination of imaginaries,  and $n$-CA, but that $U_n$ does not have  $(n+1)$-amalgamation.

Now let $p_n$ be the unique  1-type of $U_n$.

\begin{claim}\label{H3}
$H_m(p_n)=0$ for $1\leq m<n-1$; $H_{n-1}(p_n)={\mathbb{Z}}_2.$
\end{claim}
\begin{proof}
Since $p_n$ has $n$-CA, due to \ref{trivial_homology}  we have $H_m(p_n)=0$ for $1\leq m<n-1$.

Now to compute $H_{n-1}$, we introduce  an augmentation map
$$\epsilon:S_{n-1}\cC(p_n) \to \mathbb{Z}_2$$
as follows: Let $f$ be an $(n-1)$-simplex of type $p$ with $\dom(f)=\P(s)$ with $|s|=n$.
Then we let $\epsilon(f)=1$ if and only if
$R_n(f(s))$
holds. The map $\epsilon$ obviously extends as a homomorphism $\epsilon: C_{n-1}\cC(p_n)\to \mathbb{Z}_2$.

 It follows from $(*)_n$ above that an $(n-1)$-shell $c$ is the boundary of $n$-simplex iff $\epsilon(c)=0$.
  Thus for any $(n-1)$-boundary $c$, we have $\epsilon(c)=0$.
 Hence $\epsilon$ induces a homomorphism  $\epsilon_*:H_{n-1}(p_n)\to \mathbb{Z}_2$.

 Note that  there is   an $(n-1)$-shell $d$ with support $\{0,...,n\}$ such that $\epsilon(d)=1$. Hence $\epsilon_*$ is onto.
By Theorem~\ref{Hn_shells} there is an $(n-1)$-shell $d'$ such that $[d]+[d]=[d']$.
But then $\epsilon(d')=0$ and $d'$ is an $(n-1)$-boundary, i.e. $[d]+[d]=0$.
Now let $c$ be an arbitrary  $(n-1)$-shell  with support $\{0,...,n\}$.
If $\epsilon(c)=0$, then $[c]=0$. If  $\epsilon(c)=1$, then by the same argument,
$[d]-[c]=0$, i.e. $[d]=[c]$.
We have verified Claim \ref{H3}.
\end{proof}

\end{example}

\begin{example}\label{ominh1}
Here we show that for any complete 1-type $p$ over $A=\acl(A)$  in an  o-minimal theory, $H_1(p)=0$. Basically we use a similar idea as in \ref{dlo}.

Let $T$ be any rosy theory, and let $p \in S(A)$ be any type in $T$ over $A$.

\begin{definition}
The type $p(x)$ has \emph{weak 3-amalgamation} if there is a type $q(x,y) \in S(A)$ such that:

\begin{enumerate}
\item Whenever $(a,b) \models q(x,y)$, then $a$ and $b$ are independent (over $A$) realizations of $p$; and
\item For any pair $(a,b)$ of independent realizations of $p$, there is a third realization $c$ of $p$ such that $c$ is independent from $ab$ and both $(a,c)$ and $(b,c)$ realize $q$.
\end{enumerate}
\end{definition}

So any Lascar strong type in a simple theory has weak $3$-amalgamation by the Independence Theorem.

\begin{lemma}
Any nonalgebraic $1$-type (of the home sort) in an o-minimal theory has weak $3$-amalgamation.
\end{lemma}

\begin{proof}

Recall that since $T$ is o-minimal, any $A$-definable unary function $f(x)$ is either eventually increasing (that is, there is some point $c$ such that if $c < x < y$ then $f(x) < f(y)$, eventually decreasing, or eventually constant.  If $f$ is eventually constant with eventual value $d$, then $d \in \dcl(A)$.

We say an $A$-definable function $f(x_1,...,x_n)$ {\em bounded within} $p$ if for any  $c_1,...,c_n\models p$, there is  $d$ realizing $p$ such that $d > f(c_1,...,c_n).$ We call
a pair of realizations $(a,b)$ of $p$ an \emph{extreme pair} if whenever $f(x)$ is bounded within $p$, then $b > f(a)$.

First note that by the compactness theorem, for any $a$ realizing $p$, there is a $b$ realizing $p$ such that $(a,b)$ is an extreme pair.  Also, if $b \in \dcl(aA) = \acl(aA)$, then there is an $A$-definable function $f: p(\mathfrak{C}) \rightarrow p(\mathfrak{C})$ such that $b = f(a)$, so since there is no maximal realization $c$ of $p$ (because such a realization $c$ would be in $\dcl(A)$ and we are assuming that $p$ is nonalgebraic), it follows that $(a,b)$ is \textbf{not} an extreme pair.  So any extreme pair is algebraically independent over $A$ and hence thorn-independent (see \cite{On}).

\begin{claim}
Any two extreme pairs have the same type over $A$.
\end{claim}

\begin{proof}
It is enough to check that if $(a,b)$ and $(a,c)$ are two extreme pairs, then $\tp(b/Aa) = \tp(c/Aa)$.  By o-minimality, any $Aa$-definable set $X$ is a finite union of intervals, and the endpoints $\{d_1, \ldots, d_n\}$ of these intervals lie in $\dcl(Aa)$.  So $d_i = f(a)$ for some $A$-definable function $f$, and as we already observed $b,c\ne d_i$. Hence it suffices to see $b>d_i$ iff $c>d_i$. Now
by the definition of an extreme pair,

$$\forall x \models p \,\,\, \exists y \models p \left[ y > f(x)\right] \Rightarrow b > f(a) = d_i.$$

Also,

$$\exists x \models p \, \, \, \forall y \models p \left[ y \leq f(x) \right] \Rightarrow \forall x \models p \, \, \, \forall y \models p \left[y \leq f(x) \right]$$

because any two realizations of $p$ are conjugate under an automorphism in $\Aut(\mathfrak{C}/A)$ which permutes $p(\mathfrak{C})$, and so

$$\exists x \models p \, \, \forall y \models p \left[y \leq f(x) \right] \Rightarrow b \leq f(a) = d_i.$$

The same reasoning applies with $c$ in place of $b$, so

$$b > d_i = f(a) \Leftrightarrow \forall x \models p \,\,\, \exists y \models p \left[ y > f(x)\right]$$

$$\Leftrightarrow c > f(a) = d_i.$$

\end{proof}

Let $q(x,y) = \tp(a',b' / A)$ for some extreme pair.  Condition (2) of the definition of weak $3$-amalgamation can be ensured by picking $c\models p$ so that $c > g(a,b)$ for any $A$-definable function $g(y,z)$ bounded within $p$, which is possible by the compactness theorem.

\end{proof}

\begin{claim}
If $p$ has weak $3$-amalgamation, then $H_1(p) = 0$.
\end{claim}

\begin{proof}

By Theorem~\ref{Hn_shells}
 it suffices to show that every (set) $1$-shell of type $p$ with support $\{0,1,2\}$ is a boundary of some $2$-fan of type $p$.  Let $c = f_{12} - f_{02} + f_{01}$ be such a shell, where $f_{ij}$ is a $1$-simplex with support $\{i,j\}$.  The condition that this is a shell implies that there are realizations $a_0, a_1,$ and $a_2$ of $p$ such that:

$$\bd_0 f_{01} = \bd_0 f_{02} = \acl_A(a_0),$$
$$\bd_0 f_{12} = \bd_1 f_{01} = \acl_A(a_1),$$
$$\bd_1 f_{12} = \bd_1 f_{02} = \acl_A(a_2).$$

(Note that actually the boundaries above technically should be $0$-shells, but $0$-shells are determined by their domain plus a realization of $p$.)

Pick any third realization $a_3$ of $p$ as a new vertex.  For $i = 0, 1,$ or $2$, we construct a $1$-simplex $f_{i3}$ based over $A$ with support $\{i,3\}$ by letting $f_{i3}(\{i\}) = \acl_A(a_i)$ and $f_{i3}(\{3\}) = \acl_A(a_3)$, and then letting $f_{i3}(\{i,3\}) = \acl_A(a'_i, a'_3)$ where $(a'_i, a'_3)$ realizes $q(x,y)$ as in the definition of weak $3$-amalgamation (with the obvious transition maps taking $a_i$ to $a'_i$ and $a_3$ to $a'_3$).

Finally, condition~(2) in the definition of weak $3$-amalgamation implies that there are $2$-simplices $f_{123}$, $f_{023}$, and $f_{013}$ whose boundaries are alternating sums of the corresponding $f_{ij}$'s: that is,

$$\bd f_{123} = f_{23} - f_{13} + f_{12},$$

$$\bd f_{023} = f_{23} - f_{03} + f_{02},$$

and

$$\bd f_{013} = f_{13} - f_{03} + f_{01}.$$

Now if $d$ is the $2$-chain $f_{013} + f_{123} - f_{023}$, then

$$\bd d = \left( f_{13} - f_{03} + f_{01} \right) + \left(f_{23} - f_{13} + f_{12} \right) - \left(f_{23} - f_{03} + f_{02} \right)$$

$$= f_{01} + f_{12} - f_{02} = c.$$

\end{proof}

\end{example}

\section{A non-commutative groupoid construction}


In singular homology theory, one of the differences between the fundamental group and $H_1$  is that the former is not necessarily commutative while  the latter is. In the authors' earlier papers \cite{GK}, \cite{gkk}, an analogue of  homotopy theory is developed but where the ``fundamental group'' in this context is always commutative. In this last section, by taking an approach closer to the original idea of homotopy theory, we suggest how to construct a different fundamental group in a non-commutative manner. More precisely, from a full symmetric witness to the failure of $3$-uniqueness in a stable theory,
 we construct a new groupoid $\CF$ whose ``vertex groups'' $\mor_{\CF}(a,a)$ need not be  abelian. In fact, we will show below that $\mor_{\cG}(a,a) \leq  \Z(\mor_{\CF}(a,a))$, where $\cG$ is the commutative groupoid constructed in \cite{GK} and \cite{gkk}.  We may call $\CF$ the {\em non-commutative groupoid } constructed from the full symmetric witness. But unlike the groupoid $\cG$, this new groupoid $\CF$ is definable only in certain cases (e.g. under $\omega$-categoricity); in general, it is merely invariant over some small set of parameters.

\medskip

Throughout this section, we take the notational convention described in section 4. We recall that
 $\Aut(A/B)$ is the group of
elementary maps from $A$ {\em onto} $A$ fixing $B$ pointwise.
In addition, $\Aut(\tp(f/B))$ means $\Aut(Y/B)$ where $Y$ is the solution set of
$\tp(f/B)$.
\medskip


\subsection{Finitary groupoid examples}

 Let $G$ be an arbitrary finite group.  Now let $T_G$ be the complete
stable  theory of the connected  finitary groupoid  $(O,M,.)$ with
the standard setting (so $.$ is the composition map between
morphisms)
 such that  $G_a:=\mor(a,a)$ is isomorphic to $G$  for any $a\in O$.
 Fix distinct $a,b\in O$ and a morphism $f_{0}\in \mor(a,b)$.

 Now by section 4  and weak elimination of imaginaries we know that
 $$H_2(O)=\Aut(\widetilde{ab}/\ov a,\ov
 b)=\Aut( \mor(a,b)/aG_abG_b).$$
Hence indeed (see section 4.2 in \cite{GK}, and note that
$\mor(a,b)\subseteq  \dcl(f_{0}G_a)$)
$$H_2(O) =\Aut(X/aG_abG_b).$$
where $X$ is the finite solution set
 of $\tp(f_{0}/aG_abG_b)$.

Now  for $f\in X$ there is unique $x\in G_a$ such that $f=f_0.x$,
and we claim that this $x$ must be in $\Z(G_a)$.

\begin{claim}
For $x\in G_a$, we have $g:=f_0.x\in X$ iff $x\in \Z(G_a)$.
 \end{claim}
\begin{proof}
($\Rightarrow$) Since $g\in X$, $f_0\equiv_{G_aG_b} g$. Then for any
$y\in G_a$, we have $$f_0.y.f^{-1}_0(\in
G_b)=g.y.g^{-1}=f_0.x.y.x^{-1}.f^{-1}_0.$$
 Hence $x\in \Z(G_a)$.

($\Leftarrow$) There is $z\in G_b$ such that $f_0=z.g$. Now since
$x\in \Z(G_a)$, for any $y\in G_b$ we have
$$g^{-1}.y.g.x^{-1}=f_0^{-1}.z.y.z^{-1}.f_0.x^{-1}=x^{-1}.f_0^{-1}.z.y.z^{-1}.f_0=g^{-1}.z.y.z^{-1}.g.x^{-1}.$$
Hence $y=z.y.z^{-1}$, i.e. $z\in \Z(G_b)$. Now the  argument in
 \cite[4.2]{GK} says there is an automorphism fixing $aG_abG_b$
 pointwise while sending $g$ to $f_0$. Hence $g\in X$.
\end{proof}

 \begin{claim}
$H_2(O)=\Z(G)$.
 \end{claim}
\begin{proof}
The proof will be similar to that of Proposition 2.15 in \cite{gkk}.
Note firstly that due to Claim 1, $\Z(G_a)$ acts  on $X$ as an
obvious manner. This action is clearly regular. Secondly
$\Aut(X/aG_abG_b)$ also regularly  acts on $X$. Also by the argument
in \cite[4.2]{GK} it easily follows  that two actions commute. Hence
they are the same group.
\end{proof}

Note that $f\equiv_{aG_a} f_{0}$ for any $f\in \mor(a,b)$ (see again
section 4.2 in \cite{GK}), i.e. $\mor(a,b)$ is the solution set of
$\tp(f_0/aG_a)$ or $\tp(f_0/\ov a)$. Moreover for $f\in \mor(a,b)$,
$f_0$ and $f$ are interdefinable over $\ov a$. We further claim the
following.

\begin{claim}\label{iso}
$G$ is isomorphic to $\Aut(\mor(a,b)/\ov{a})=\Aut(\mor(a,b)/aG_a)$.
Hence $H_2(O)=\Aut(\widetilde{ab}/\ov a,\ov
 b)=  \Aut( \mor(a,b)/aG_abG_b)=   \Z(\Aut(\mor(a,b)/aG_a)).$
\end{claim}
\begin{proof}
We know $G$ and $G_b$ are isomorphic. Now clearly we can consider
$\sigma\in G_b$ as an automorphism in $\Aut(\mor(a,b)/\ov{a})$ via
the map $f(\in \mor(a,b))\mapsto \sigma.f$. Now this correspondence
is clearly 1-1 and onto (both groups are finite). It is obvious that
the correspondence is an isomorphism.
\end{proof}
In the following section we try to search this phenomenon
 in the general stable theory context.   Namely given the abelian groupoid
 built  from a  {\em symmetric witness}
  introduced in \cite{GK}, we construct an extended groupoid
  possibly non-abelian but the abelian groupoid places in the
  center of the new groupoid.  In the case of
above $T_G$, as we seen the morphism group of the abelian groupoid
is $\Z(G)$, but in the extended one the morphism group is equal to
$G$.

\subsection{The non-commutative  groupoid $\CF$}

For the rest of this section, we work in a complete \emph{stable} theory $T$ with
monster model $\C=\C^{eq}$.

We will work with full symmetric witnesses to the failure of $3$-uniqueness as defined in Definition~\ref{full_symm_witness} above.

First we fix some notation that we will refer to throughout the rest of the section.  Let $(b_0, b_1, b_2, f'_{01}, \ldots)$ be a full symmetric witness to the failure of $3$-uniqueness, for convenience over the base set $\emptyset = \acl(\emptyset)$.  Recall from the discussion in section~4 above that from this witness we can construct a definable connected abelian groupoid $\cG$ such that $\operatorname{Ob}(\G)=p(\C)$ where $p=\tp(b_i)$, and there is a
canonical bijection $\pi$ from the finite solution set of
$\tp(f'_{01}/\ov{b_0} \cup \ov{b_1})$ to $\mor_{\G}(b_0,b_1)$ in such a way that
$f_{01}:=\pi(f'_{01})$ and $f'_{01}$ are interdefinable over  $b_{01}$, so
$\mor_{\G}(b_0,b_1)$ also is the solution set of
$\tp(f_{01}/\ov{b_0} \cup \ov{b_1})$ (equivalently of $\tp(f_{01}/\ov{b_0} \cup \ov{b_1})$). Moreover
the abelian group $\mor_{\G}(b_i,b_i)$ is isomorphic to
$\Aut(\tp(f_{01}/b_{01}))=\Aut(\tp(f'_{01}/\ov{b_0} \cup \ov{b_1}))$.

Now as promised, by extending the construction method given in \cite{GK} we find
another groupoid $\CF$ (which will be definable {\em only in a certain context})
from the same full symmetric witness $(b_0, b_1, b_2, f'_{01}, \ldots)$.
$\ob(\CF)$ will be the same as $\ob(\G)$.  But  $\CF$ need not be
abelian as  $\mor_{\CF}(b_i,b_i)$ will be $\Aut(Y_{01}/\ov b_0)$, where $Y_{01}$ is the
{\em possibly infinite} set
$$Y_{b_{01}}=Y_{01}:=\{f\in\dcl(f_{01},\ov{b_0})|\ f\equiv_{\ov b_0}f_{01}
\text{ and } \dcl(f\ov{b_0})=\dcl(f_{01}\ov{b_0})\}.$$ Note that
$\dcl(f_{01},\ov{b_0})=\dcl(f_{01}b_1\ov{b_0})$ since $b_1\in\dcl(f_{01})$.
Moreover  $Y_{01}$ and  $Y'_{01}$, the set defined the same way as $Y_{01}$ but substituting $b_1f'_{0}$ for $f_{01}$, are interdefinable. Furthermore, we shall see that $\mor_{\G}(b_i,b_i)\leq
\Z(\mor_{\CF}(b_i,b_i))$ (Claim \ref{centergp}). We will call $\CF$
the {\em non-commutative  groupoid} constructed from the symmetric witness.

\begin{claim}
The set $Y_{01}$ defined in the previous paragraph depends only on $b_0$ and $b_1$ and not on the choice of $f_{01} \in \mor_{\cG}(b_0, b_1)$.
\end{claim}

\begin{proof}
Given any other $g \in \mor_{\cG}(b_0, b_1)$, we have that $g \equiv_{\ov{b}_0} f_{01}$, as already discussed.  Also, $g = f_{01} \circ h$ for some $h \in \mor_{\cG}(b_0, b_0) \subseteq \acl(b_0)$, and so $g$ and $f_{01}$ are interdefinable over $\acl(b_0)$.  Now the result follows.
\end{proof}

For convenience,  fix independent $a,b\models p$ and $f_{ab}$ such
that $b_{01}f_{01}\equiv abf_{ab}$. We use $\Pi_{ab}$ to denote $\mor_{\G}(a,b)$, and use $\Pi_{a}$
for $\Pi_{aa}$. As mentioned above, $\Pi_{ab}$ is the
solution set of $\tp(f_{ab}/\ov a\ov b)$, on which
$G_{ab}:=\Aut(\tp(f_{ab}/\ov a\ov b))$ acts regularly. Hence
$G_{ab}$ and  $\Pi_a$ are canonically isomorphic \cite[2.15]{gkk}.

\begin{lemma}\label{uniformaction}
A set $C=\{c_i\}_i$ of realizations of $p$ with $b_0\Ind C$, and
$g_i\in\Pi_{b_0c_i}$ are given. Then for $\sigma\in \Pi_{b_0}$,
there is an automorphism $\mu=\mu_{\sigma}$ of $\C$ fixing each
$\ov{c_i}$ and $\ov{b_0}$ pointwise and $\mu(g_i)=g_i.\sigma$. Similarly, if  $D=\{d_i\}_i(\Ind b_0)$ is a set of realizations of
$p$ and $h_i\in\Pi_{d_ib_0}$, then there is an automorphism $\tau$
fixing $\ov{d_i}$ and $\ov{b_0}$ such that $\tau(h_i)=\sigma.h_i.$
\end{lemma}
\begin{proof} Take $d\models p$ independent from $b_0C$; and take $h\in
\Pi_{b_0d}$. For each $i$, there is  $h_i\in \Pi_{dc_i}$ such that
$g_i=h_i.h$. Now by stationarity we have $g_{0}\equiv_{\ov{b_0},
\ov{Cd}} g_0.\sigma$ witnessed by an automorphism $\mu$ sending
$g_0$ to $g_0.\sigma$ and fixing $\ov{b_0}, \ov{Cd}$. Then
$\mu(g_i)=\mu(h_i.h)=h_i.\mu(h)$ since $h_i\in \ov{Cd}$. Now there
is unique $\tau\in \Pi_{b_0}$ such that $\mu(h)=h.\tau$. Thus
$\mu(g_0)=g_0.\sigma=h_0.h.\tau$. Hence $\sigma=\tau$. Similarly
there is $\tau_i\in\Pi_{b_0}$ such that $\mu(g_i)=g_i.\tau_i$, and
then $\mu(g_i)=g_i.\tau_i=h_i.h.\sigma$. Hence $\tau_i=\sigma$, so
$\mu(g_i)=g_i.\sigma$ as desired. The second clause can be proved similarly.
\end{proof}

Now consider $F_{b_{01}}=F_{01}:=\Aut(Y_{01}/\ov{b_0})$ where
$Y_{01}$ is defined above.

\begin{claim}\label{centergp}
\be
\item $\Pi_{b_{01}}\subseteq Y_{01}$.

\item The action of $F_{01}$ on $Y_{01} $ (obviously by $\sigma(g)$ for $\sigma\in
F_{01}$ and $g\in Y_{01}$) is regular (so $|F_{01}|=|Y_{01}|$ but
can be infinite). Hence given $\mu\in G_{01}:=G_{b_{01}}$, there is
its unique extension in $F_{01}$ (we may identify those two).  Thus
$Y_{01}$ is $b_{01}$-invariant set.

\item  $G_{01}\leq\Z(F_{01}).$

\item For $\tau\in F_{01}$ and $f\in \Pi_{b_{01}}$, $e\in
\Pi_{b_0}$, we have $\tau(f.e)=\tau(f).e.$ Moreover if
$\sigma(f)=f.e$ for some $\sigma\in G_{01}$, then
$\sigma(f,\tau(f))=(f.e,\tau(f).e)$.  \ee
\end{claim}
\begin{proof} (1)  is clear.

(2) comes from the fact that for any  $g_0,g_1\in Y_{01}$, they are
interdefinable over $\ov{b_0}$, and $Y_{01} \subseteq \dcl(g_i\bar
b_0)=\dcl(f_{01}\bar b_0)$. Hence from (1),
 it follows $G_{01}$ is
a subgroup of $F_{01}$. The rest clearly follows.

(3) Suppose  $\sigma\in G_{01},\tau\in F_{01}$ are given.  Let
$g=\sigma(f_{01})=f_{01}.\sigma_0$  for some $\sigma_0\in \Pi_{b_0}$, and
let $h=\tau(f_{01})$. Then
$\tau(f_{01},g)=(\tau(f_{01}),\tau(g))=(h,\tau\circ \sigma(f_{01})).$ But
since $\tau$ fixes $\ov{b_0}$, it follows $\tau\circ
\sigma(f_{01})=\tau(f_{01}.\sigma_0)=\tau(f_{01}).\sigma_0=h.\sigma_0..$ Now
by Lemma \ref{uniformaction}, there is an automorphism $\sigma'$
fixing $\ov{b}_0$ and $\ov{b}_1$ and sending $(f_{01},h)$ to $(f_{01}.\sigma_0,h.\sigma_0)$, so
$\sigma(f_{01})=\sigma'(f_{01})=f_{01}.\sigma$. But then due to the
uniqueness of the extension of $\sigma$ in $F_{01}$, we must have
that $\sigma (h)=h.\sigma_0=\sigma'(h)$ as well. Thus
$\tau\circ\sigma(f_{01})=h.\sigma_0=\sigma(h)=\sigma\circ\tau(f_{01})..$
Then due to regularity, we conclude $\sigma\in \Z(F_{01})$.

(4) By (3),  $\tau(f.e)=\tau\circ\sigma(f)=\sigma\circ\tau(f)$. But
again by uniqueness with Lemma \ref{uniformaction},
$\sigma(\tau(f))=\tau(f).e$. Therefore
$\sigma(f,\tau(f))=(\sigma(f),\tau(\sigma(f)))=(f.e,\tau(f).e)$.
\end{proof}

In general $G_{01}$ need not be equal to $\Z(F_{01})$ (see the remarks
before Proposition \ref{isomorphic}).

We define $Y_{ab}$ just like $Y_{01}$ but with $b_{01}f_{01}$ replaced by $abf_{ab}$.

\begin{lemma}\label{wdfn}
Let $c\models p$ and $c\Ind ab$. Let $g\in \Pi_{ca}$. Then for $f\in
Y_{ab}$, it follows $h=f.g\in Y_{cb}$. Moreover  for $h_0=f_{ab}.g$,
we have
 $$h_0f_{ab}\equiv_{\ov{ac}}hf\mbox{ and }f_{ab}f\ov a\equiv_{\ov{b}} h_0h\ov
 c.$$
\end{lemma}
\begin{proof}
Note that $h_0\in \Pi_{cb}$. By stationarity, there is a
$\ov{ca}$-automorphism $\mu$ such that  $\mu(f_{ab})=f$. Then
$\mu(h_0)=\mu(f_{ab}.g)=\mu(f_{ab}).\mu(g)=f.g=h\in\ov{cb}$. We want
to see that $h,h_0$ are interdefinable over $\ov{c}$. Suppose not
say there is $h'\equiv_{\ov{c}h_0}h$ and $h'\ne h$. Then again by
stationarity there is a $\ov{ca}$-automorphism $\tau$ such that
$\tau(h_0h)=\tau(h_0h')$. Then for $f=h.g^{-1}$ and $f'=h'.g^{-1}$,
we have $f\ne f'$ but $\tau(f_{ab},f)=\tau(h_0.g^{-1},h.g^{-1})
=(h_0.g^{-1},h'.g^{-1})=(f_{ab},f'),$ a contradiction. Similarly one
can show that $h_0\in\dcl(\ov{c}h)$. Hence $h\in Y_{cb}$. Now $\mu$
witnesses $h_0f_{ab}\equiv_{\ov{ac}}hf$. To show $f_{ab}f\ov
a\equiv_{\ov{b}} h_0h\ov c$, choose $d(\models p)\Ind abc$. Now for
$k_0\in \Pi_{db}$,  by our proof there is $k\in Y_{db}$ such that
$f=k.(k^{-1}_0.f_{ab})$. Then
$h=k.k^{-1}_0.(f_{ab}.g)=k.k^{-1}_0.h_0.$ Now by stationarity,
$f_{ab}\ov a\equiv_{\ov{bd}}h_0\ov c$. Since $k,k_0\in \ov{bd}$, as
desired $f_{ab}f\ov a\equiv_{\ov{bd}} h_0h\ov c$.
\end{proof}

Now we start to construct the new groupoid mentioned. At the first
approximation, our $\mor_{\CF}(a,b)$ will be $Y_{ab}$. Beware that
$Y_{ab}(\supseteq \Pi_{ab})$ need not be definable  nor
type-definable. It is just an $ab$-invariant set. So our groupoid
$\CF$ will  only be invariant, and it will be definable only under additional hypotheses (e.g. $\omega$-categoricity).

As explained in \cite[2.9, 2.10]{gkk}, there is the binding group
$G$ (isomorphic to $G_{ab}$, and so to $\Pi_a$) acting on $\G$. In
general the action is {\em not} a structure automorphism of the
groupoid as for example $\id_a\in \Pi_a$ need not be fixed. But it
is so for $\Pi_{ab}$ (or more generally as in Lemma
\ref{uniformaction} above), i.e. for $\bsigma\in G$ we have
$f_{ab}\equiv \bsigma\cd f_{ab}=f_{ab}\cd \bsigma\in \Pi_{ab}$ (see
\cite[2.10]{gkk}; We use $\cd$ for the group action of $G$ to $\G$).
Hence there is an induced isomorphism $\rho_{ab}:G\to G_{ab}$ such
that $\rho_{ab}(\bsigma)(f)=\bsigma\cd f$ for any $f\in \Pi_{ab}$.
 We write
$\sigma_{ab}$ for $\rho_{ab}(\bsigma)$. But  when there is no chance
of confusion, we use $\sigma$ for both $\bsigma\in G$ and
$\sigma_{ab}\in G_{ab}$. Also, $\sigma_a$ denotes the unique element
in $\bsigma\cap \Pi_a$, as described in \cite[Definition 2.9]{gkk}.
Hence for $f\in\Pi_{ab}$, $\sigma(f)=\sigma\cd
f=\sigma_b.f=f.\sigma_a$.

\begin{claim}\label{tpreserv}
For $\sigma\in G$, and $f\in \Pi_{ab}$ and $g\in \Pi_{cd}$ with
$cd\equiv ab$, we have $f,\sigma\cdot f\equiv g,\sigma\cdot g$.
\end{claim}
\begin{proof} Choose $e\models p$ independent from $abcd$. By
\ref{uniformaction} above, $f,\sigma\cdot f\equiv h,\sigma\cdot h
\equiv k,\sigma\cdot k \equiv g,\sigma\cdot g$ where $h\in\Pi_{ae}$
and $k\in \Pi_{ce}$.
\end{proof}

Now let $F_{ab}:=\Aut(Y_{ab}/\ov{a})$. Then as in Claim
\ref{centergp}.(2), $G_{ab}\leq F_{ab}$. As just said for any
$cd\equiv ab$, there is the canonical isomorphism between
$\rho_{cd}\circ\rho^{-1}_{ab}:G_{ab}\to G_{cd}$.
 We somehow try to find the
canonically extended isomorphism between $F_{ab}$ and $F_{cd}$ as
well. We do this as follows. Let $Y_{ab}=\{g_i\}_i\cup \{g'_j\}_j$
and let $Y_{cd}=\{h_i\}_i\cup \{h'_j\}_j$ such that
$\Pi_{ab}=\{g_i\}_i$, $\Pi_{cd}=\{h_i\}_i$ and the sequences $\la
g_i\ra^\frown \la g'_j\ra ab\equiv \la h_i\ra^\frown \la h'_j\ra
cd$. Now due to regularity of the  action, for each $i$ or $j$ there
is unique $\mu^{ab}_i$ or $\mu^{ab}_j\in F_{ab}$ such that
$\mu_i(g_0)=g_i$ or $\mu_j(g_0)=g'_j$. Similarly we have
$\mu^{cd}_i$ or $\mu^{cd}_j\in F_{cd}.$

\begin{claim}
The correspondence $\mu^{ab}_i\mapsto\mu^{cd}_i$ or
$\mu^{ab}_j\mapsto\mu^{cd}_j$ is a well-defined isomorphism from
$F_{ab}$ to $F_{cd}$ extending $\rho_{cd}\circ\rho^{-1}_{ab}$.
\end{claim}
\begin{proof}
Assume $\{k_i\}_i\cup \{k'_j\}_j$ is another arrangement of $Y_{cd}$
such that $\la k_i\ra^\frown \la k'_j\ra \equiv_{cd}\la
h_i\ra^\frown \la h'_j\ra$. Then $k_0=\sigma(h_0)$ for some
$\sigma\in G_{cd}$. Thus by Claim \ref{centergp}, we have
$\sigma(h_0,\mu^{cd}_i(h_0))=(k_0,\mu^{cd}_i(k_0))$ and so
$h_0,\mu^{cd}_i(h_0)\equiv_{\ov c}k_0,\mu^{cd}_i(k_0)$. Then due to
interdefinability, we must have $\mu^{cd}_i(k_0)=k_i$. Similarly
$\mu^{cd}_j(k_0)=k'_j$. Hence the map is well-defined. It easily
follows that the map in fact is an isomorphism.  Moreover due to
\ref{tpreserv} we  see that it  extends
$\rho_{cd}\circ\rho^{-1}_{ab}$.
\end{proof}

Hence now we  fix an {\em extended binding group} $F\geq G$
isomorphic to $F_{01}$. Then there is a canonical isomorphism
$\rho^F_{cd}:F\to F_{cd} $ extending $\rho_{cd}$ in such a way that
$\rho^F_{cd}\circ(\rho^F_{ab})^{-1}$ is the correspondence defined
above. Now for $\bmu\in F$, we use $\mu_{ cd}$ or simply $\mu$ to
denote $\rho^F_{cd}(\bmu)$. Note that a mapping $\bmu\cdot
f:=\mu_{cd}(f)$ is clearly a regular action of $F$ on $Y_{cd}$
extending that of $G$ on $\Pi_{cd}$.

\begin{claim}
If $cd\Ind a$, then for $f\in\Pi_{cd}, g\in\Pi_{ac}$, we have
$\bmu\cdot (f.g)= (\bmu\cdot f).g$.
\end{claim}
\begin{proof}
This follows from Lemma \ref{wdfn}.
\end{proof}

Assume now $c(\models p)\Ind ab$, and $g\in Y_{ab}, h\in Y_{bc}$ are
given. We want to define a composition $h.g\in Y_{ac}$ extending
that for $\G$. Note now $g=\tau_0(g_0)$ and $h=\sigma_0(h_0)$ for
some $\mbox{\boldmath $\tau_0$}, \mbox{\boldmath $\sigma_0$}\in F$
and $g_0\in \Pi_{ab}, h_0\in\Pi_{bc}$. We define $h.g:=
(\mbox{\boldmath $\sigma_0$}\circ \mbox{\boldmath
$\tau_0$})\cdot(h_0.g_0)=\sigma_0\circ\tau_0(h_0.g_0).$

\begin{claim}\label{compost}
The composition map is well-defined, invariant under any automorphism of $\C$, and extends that of $\mor(\G)$. For any $f\in Y_{ac}$,
there is unique $h'\in Y_{bc}$ such that $f=h'.g$.
\end{claim}
\begin{proof}
Let $g=\tau_1(g_1)$ and $h=\sigma_1(h_1)$ for some $\mbox{\boldmath
$\tau_1$}, \mbox{\boldmath $\sigma_1$}\in F$ and $g_1\in \Pi_{ab},
h_1\in\Pi_{bc}$. Then since $\sigma^{-1}_0\circ\sigma_1(h_1)=h_0$
and $\tau^{-1}_0\circ\tau_1(g_1)=g_0$, due to uniqueness we have
that both $\mbox{\boldmath $\sigma^{-1}_0$}\circ \mbox{\boldmath
$\sigma_1$}$, $\mbox{\boldmath $\tau^{-1}_0$}\circ \mbox{\boldmath
$\tau_1$}$ are in $G$ so in the center of $F$. Now
$$\begin{array}{cllll}
\sigma_0\circ\tau_0(h_0.g_0)&=
&\sigma_0\circ\tau_0\circ\sigma^{-1}_0\circ\sigma_0(h_0.g_0)&=&
\sigma_0\circ\tau_0\circ\sigma^{-1}_0(\sigma_0(h_0).g_0)\\
&=&\sigma_0\circ\tau_0\circ\sigma^{-1}_0(\sigma_1(h_1).g_0)&=&
\sigma_0\circ\tau_0\circ(\sigma^{-1}_0\circ\sigma_1)(h_1.g_0)\\
&=&\sigma_1\circ\tau_0(h_1.g_0)&=
&\sigma_1\circ\tau_1\circ(\tau^{-1}_1\circ\tau_0)(h_1.g_0)\\
&=&\sigma_1\circ\tau_1(h_1.(\tau^{-1}_1\circ\tau_0)(g_0))&=&
\sigma_1\circ\tau_1(h_1.(\tau^{-1}_1(\tau_1(g_1))))\\
&=&\sigma_1\circ\tau_1(h_1.g_1).& &
\end{array}$$
Automorphism invariance clearly follows from the same property for
$\mor(\G)$ and the choice of the isomorphism $\rho^F_{ab}$. Moreover
by taking $\tau_0=\sigma_0=\id$, we see that the composition clearly
extends that for $\G$. Lastly $f=\tau(f_1)$ for some $f_1\in
\Pi_{ac}$. Now there is $h'_1\in \Pi_{bc}$ such that $f_1=h'_1.g_1.$
Put $h'=\tau\circ\tau_1^{-1}(h'_1)$. Then by the definition,
$f=h'.g$. For any $h''(\ne h')\in Y_{bc}$ it easily follows that
$f\ne h''.g$. Hence $h'$ is unique such element.
\end{proof}

The rest of the construction of $\CF$ will be similar to that of $\G$ in
\cite{GK}. $\mbox{Ob}(\CF)$ will be the same as
$\mbox{Ob}(\G)=p(\C)$. Now for arbitrary $c,d\models p$, an {\em
$n$-step directed path} from $c$ to $d$ is a sequence
$(c_0,g_1,c_1,g_2...,c_n)$ such that $c=c_0, d=c_n, c_{i-1}c_i\equiv
ab$ and $g_i\in Y_{c_{i-1}c_i}$. Let $D^n(c,d)$ be the set of all
$n$-step directed paths.  For $q=(c_0,g_1,c_1,g_2...,c_n)\in
D^n(c,d)$ and $r=(d_0,h_1,d_1,h_2...,d_m)\in D^m(c,d)$ we say they
are equivalent (write $r\sim s$) if for some  $c^*(\models p)\Ind
qr$ and $g^*\in Y_{c^*c}$, we have $g^*_n=h^*_m\in Y_{c^*d}$ where
$g^*_0=h^*_0=g^*$ and  $g^*_{i+1}=g_{i+1}.g_i^*$ ($i=0,...,n-1$) and
$h^*_{j+1}=h_{j+1}.h_j^*$  ($j=0,...,m-1$). Due to stationarity the
 relation is independent from the choices of $c^*$ and
$g^*$, and is an equivalence relation. Similarly to Lemma
\cite[2.12]{GK}, one can easily see using Claim \ref{compost} that
for any $q\in D^n(c,d)$, there is $r\in D^2(c,d)$ such that $q\sim
r$. Then $D^2(c,d)/\sim$ will be our $\mor_{\CF}(c,d)$, and
composition will be concatenation of paths. The identity morphism in
$\mor_{\CF}(c,c)$ can be defined just like in \cite[2.15]{GK}.
Now our groupoid $\CF$ clearly extends $\G$. An argument similar to that in
\cite[2.14]{GK} implies there is a canonical 1-1 correspondence between
$Y_{ab}$ and $\mor_{\CF}(a,b).$ But $\CF$ need not be definable nor
type-definable nor hyperdefinable. It is just an invariant groupoid.

As pointed out in \ref{centergp},   $Y_{ab}$ is $ab$-invariant. Now
if it is type-definable then as it is a bounded union of definable
sets, by compactness it indeed is definable and a finite set. (This
happens when $T$ is $\omega$-categorical.) For this case let us add
a bit more explanations that are not explicitly mentioned in
\cite{GK}. By compactness now, $\sim$ turns out to be definable:
Note that $D^2(p):=\bigcup \{D^2(c,d)|\ c,d\models p\}$ is
$\emptyset$-type-definable. Then there clearly is an
$\emptyset$-definable equivalence relation $E$ on $D^2(p)$ each of
whose class is of the form $D^2(c,d)$. In each $E$-class, there are
exactly $|Y_{ab}|$-many $\sim$-classes. Hence $\sim$ is
$\emptyset$-definable relatively on $D^2(p)$ as well. Hence $[g]\in
\mor_{\CF}(c,d)$ is an imaginary element and the maps $[g]\mapsto c$
or $d$ (the first and last components of $g$) are
$\emptyset$-definable {\em domain} and {\em range} maps. Moreover
for $[f]\in \mor_{\CF}(c,d)\subseteq \ov{cd}$, we have
$[f]\equiv_{\ov c}[g]$. Therefore $\CF$ is a (relatively)
$\emptyset$-definable groupoid.

\medskip

We return to the general context of an invariant $\CF$. For notational simplicity,
use $\Phi_{cd}$ to denote $\mor_{\CF}(c,d)$, and use $\Phi_c$ for
$\Phi_{cc}$.   We finish this note by stating some observations
regarding $\CF$. First of all, one can construct an example where
$\Pi_a$ is not equal to $\Z(\Phi_a)$ but where $\Phi_a$ is
abelian. In the groupoid example \cite[Section 4.2]{GK} (which is
quite similar to $T_G$ in section 1 of this note), we may add some
of irrelevant elements to each object and also extend each morphism
tuple by capturing those elements. Then by permuting the elements we
can have larger $\Phi_{\bar x}$ than $\Pi_{\bar x}$ where $\bar x$
is an extended object. Now then $\Phi_{\bar x}$ simply is some
direct product of $\Phi_{\bar x}$, so if $G$ in \cite[4.2]{GK} is
already abelian then $G=\Pi_{\bar x}$ need not be equal to
$\Z(\Phi_{\bar x})=\Phi_{\bar x}$. But of course if we work with $T_G$
as in subsection 6.1 or the example in \cite[4.2]{GK} as they are, then
$\Pi_x=\Z(G)$ and $\Phi_x=G$.

Now for $f\in Y_{ab}$, we write $\ul{f}$ for its canonically
corresponding element in $\Phi_{ab}$. Note  that for $\sigma\in
F_{ab}$ and $f\in Y_{ab}$, we have $\sigma(f)\in Y_{ab}$ and both
$\ul{f}, \ul{\sigma(f)}\in \Phi_{ab}$. But $\sigma(\ul f)$ need not
be in $\Phi_{ab}$. In general $\sigma(\ul f)\in \Phi_{a\sigma(b)}.$
We get now the following results for $\CF$ similarly to those of
$\G$.

\begin{proposition}\label{isomorphic}
The group $F_{ab}$ is isomorphic to $\Phi_a$. In fact for any
$\mu\in F_{ab}$, there is $\mu_b\in\Phi_b$ such that for any $f\in
Y_{ab}$, $\ul{\mu(f)}=\mu_b.\ul f$. Hence $\Phi_b=\{\mu_b|\ \mu\in
F_{ab}\}$.
\end{proposition}
\begin{proof}
The proof will be similar to that of \ref{iso}. Define a map
$\eta:\Phi_a\to F_{ab}$ such that for $\sigma\in \Phi_a$ and $f\in
Y_{ab}$, we let $\eta(\sigma)(f)=g$ where $\ul g=\sigma.\ul f$.
Clearly $\eta$ is a well-defined 1-1 map. It  is onto as well since
any $\mu\in F_{ab}$ is determined by $(f,\mu(f))$. But obviously for
some $\sigma'\in \Phi_a$, we have $\eta(\sigma')(f)=\mu(f)$. By
commutativity, it easily follows that this map is in fact an
isomorphism. Now we take $\mu_b=\eta^{-1}(\mu)$.
\end{proof}

\begin{proposition}
For $c(\models p)\Ind ab$ and $f\in Y_{ab},g\in Y_{bc}, h\in
Y_{ac}$, we have $h=g.f$ iff $\ul h=\ul{g}.\ul{f}$.
\end{proposition}
\begin{proof}
 Since the composition relation defined in
\ref{compost} is invariant relation, we can find an $\emptyset$-invariant relation
$\theta(x,y,z)$ such that for any $a'b'c'\equiv
abc$ and $f'\in Y_{a'b'},g'\in Y_{b'c'}, h'\in Y_{a'c'}$, we have
$h'=g'.f'$ iff $\theta(a'b'f',b'c'g',a'c'h')$ holds. Then the rest
proof of the proposition will be exactly the same as that of
\cite[2.12]{gkk}, hence we omit it.
\end{proof}

\end{document}